\documentclass[11pt,reqno]{amsart}

\usepackage{amsmath}
\usepackage{amsthm}
\usepackage{graphicx}
\usepackage{mathtools}
\usepackage{amsfonts}
\usepackage{amssymb}
\usepackage{hyperref}
\usepackage{bm}
\usepackage[margin=1 in]{geometry}
\usepackage{enumerate}
\usepackage[normalem]{ulem}
\usepackage{tikz}
\usepackage{tikz-cd}
\usepackage{xcolor}
\usetikzlibrary{matrix,arrows,positioning,automata}
  \usepackage{cancel}

\usepackage[numbers]{natbib}

\usepackage{todonotes}
\numberwithin{equation}{section}

\newtheorem{theorem}{Theorem}[section]
\newtheorem{lemma}[theorem]{Lemma}
\newtheorem{proposition}[theorem]{Proposition}
\newtheorem{corollary}[theorem]{Corollary}
\newtheorem{assumption}[theorem]{Assumption}

\theoremstyle{definition}
\newtheorem{definition}[theorem]{Definition}
\newtheorem{remark}[theorem]{Remark}

\def\E{{\mathbb E}}

\def\R{{\mathbb R}}
\def\N{{\mathbb N}}

\def\P{{\mathcal P}}

\def\L{{\mathcal L}}

\def\F{{\mathcal F}}
\def\tr{{\mathrm{tr}}}

\newcommand{\sP}{\mathcal{P}}

\newcommand{\bx}{\mathbf{x}}
\newcommand{\bX}{\mathbf{X}}
\newcommand{\bY}{\mathbf{Y}}
\newcommand{\bZ}{\mathbf{Z}}

\newcommand{\by}{\mathbf{y}}
\newcommand{\sL}{\mathcal{L}}
\newcommand{\bP}{\mathbb{P}}
\newcommand{\spt}{\mathcal{P}_2(\R^n)}
\newcommand{\sF}{\mathcal{F}}
\newcommand{\ov}{\overline}
\newcommand{\sA}{\mathcal{A}}
\newcommand{\bbF}{\mathbb{F}}
\newcommand{\wt}{\widetilde}

\DeclareMathOperator*{\argmin}{arg\,min}

\newcommand\cA{\mathcal A}

\newcommand\cC{\mathcal C}

\newcommand\cF{\mathcal F}

\newcommand\cL{\mathcal L}

\newcommand\cP{\mathcal P}

\newcommand\cS{\mathcal S}

\newcommand\cU{\mathcal U}

\newcommand\cW{\mathcal W}

\def \E{\mathbb{E}}
\def \F{\mathbb{F}}

\def \H{\mathbb{H}}

\def \L{\mathbb{L}}

\def \N{\mathbb{N}}

\def \P{\mathbb{P}}

\def \R{\mathbb{R}}

\def \e{\mathrm{e}}
\def \stwo{\mathcal{S}^2}
\def \ltwo{\mathcal{L}^2}

\newcommand{\bz}{\bm z}

\newcommand\blfootnote[1]{%
  \begingroup
  \renewcommand\thefootnote{}\footnote{#1}%
  \addtocounter{footnote}{-1}%
  \endgroup
}

\thanks{J.J. is supported by the NSF under Grant No. DGE1610403.
L.T is partially supported by the NSF under Grants DMS-2005832 and CAREER DMS-2143861. Any opinions, findings and conclusions or recommendations expressed in this material are those of the authors and do not necessarily reflect the views of the NSF}

\title[Convergence for mean field games with controlled volatility]{Quantitative convergence for displacement monotone mean field games with controlled volatility}
\author{Joe Jackson and Ludovic Tangpi} 
\address{The University of Texas at Austin}
\email{jjackso1@utexas.edu}
\address{Princeton University}
\email{ludovic.tangpi@princeton.edu}

\begin{document}

\blfootnote{The authors wish to thank Alp\'ar R. M\'esz\'aros for a helpful exchange which led to several improvements of the paper.}

\begin{abstract}
We study the convergence problem for mean field games with common noise and controlled volatility. We adopt the strategy recently put forth by Lauri\`ere and the second author, using the maximum principle to recast the convergence problem as a question of ``forward-backward propagation of chaos", i.e (conditional) propagation of chaos for systems of particles evolving forward and backward in time. Our main results show that displacement monotonicity can be used to obtain this propagation of chaos, which leads to quantitative convergence results for open-loop Nash equilibria for a class of mean field games. Our results seem to be the first (quantitative or qualitative) which apply to games in which the common noise is controlled. The proofs are relatively simple, and rely on a well-known technique for proving well-posedness of FBSDEs which is combined with displacement monotonicity in a novel way. To demonstrate the flexibility of the approach, we also use the same arguments to obtain convergence results for a class of infinite horizon discounted mean field games.
\end{abstract}

\maketitle

\tableofcontents

\section{Introduction}

Consider an $N$-player stochastic differential game in which player $i \in\{ 1,\dots,N\}$ chooses an $\R^k$-valued control $\alpha^i$ which impacts the $\R^n$-valued state process $X^i$ via the dynamics 
\begin{align*}
dX_t^i = b(X_t^i, \alpha_t^i, m_{\bX_t}^N) dt + \sigma(X_t^i, \alpha_t^i, m_{\bX_t}^N) dW_t^i + \sigma^0(X_t^i, \alpha_t^i, m_{\bX_t}^N) dW_t^0, \quad X_0^i = \xi^i, 
\end{align*}
with $(\xi^i)_{i = 1,\dots,N}$ i.i.d. with common law $m_0$. 
Player $i$ chooses the (open-loop) control $\alpha^i$ in order to minimize the cost functional 
\begin{align*}
J^{N,i}(\bm \alpha) = J^{N,i}(\alpha^1,\dots,\alpha^N) = \E\bigg[ \int_0^T L(X_t^i,\alpha_t^i,m_{\bX_t}^N) dt + G(X_T^i,m_{\bX_T}^N) \bigg].
\end{align*}
Here we write $\bX^N = (X^1,\dots,X^N)$ and throughout the paper we use the notation
$$m^N_{\bx}:= \frac1N\sum_{i=1}^N\delta_{x^i}$$ for any $\bx = (x^1,\dots, x^N) \in (\R^n)^N$.
Here $W^0,\dots,W^N$ are independent Brownian motions, and the data consists of sufficiently smooth functions $b, \sigma, \sigma^0, L,  G$ (precise assumptions will be made later). Our main results are concerned with the convergence as $N \to \infty$ of sequences of $N$-player Nash equlibria.
It is expected that for $N$ large, this $N$-player game is well approximated by a certain limiting model, known as a mean field game.
% to the equilibrium of the corresponding mean field game. 
We formulate this limiting model precisely in subsection \ref{subsec:gamesetup}, but a rough description is as follows: a measure-valued process $m = (m_t)_{0 \leq t \leq T}$ representing the (conditional) distribution of the population's state is fixed and the representative agent chooses an $\R^k$--valued control $\alpha$
in order to minimize
\begin{align}
  J_m(\alpha) = \E\bigg[  \int_0^T  L(X_t, \alpha_t, m_t) dt + G(X_T,m_T) \bigg]
\end{align}
subject to the controlled state dynamics
\begin{align}  \label{mfdynamicsintro}
dX_t =  b(X_t, \alpha_t, m_t) dt +\sigma(X_t,\alpha_t, m_t)d W_t + \sigma^0(X_t, \alpha_t, m_t)d W_t^0, \quad X_0 = \xi \sim m_0. 
\end{align}
A mean field Nash equilibrium or simply a mean field equilibrium (MFE) is a pair $(m, \alpha)$ such that $m$ is a measure flow, $\alpha$ is a minimizer of $J_{m}$ with corresponding state process $X$, and it holds that
\begin{align*}
 m_t = \sL(X_t | \sF_t^0), 
\end{align*}
where $\bbF^0 = (\sF_t^0)_{0 \leq t \leq T}$ is the filtration of $W^0$, and we set $\sF^0 = \sF_T$.

The well-posedness of this limiting model and its connection to the $N$-player games above is the subject of mean field game (MFG) theory, which was initiated by \citet{lasry2006jeux,lasry2006jeux2,lasry2007mean} and \citet{huang2006large,huang2007large}. Owing to the fact that MFGs are usually more tractable (especially in numerical simulations), the theory and applications of MFGs has grown enormously in the past decade.
We refer to \citet{cardelbook1,cardelbook2} and \citet{bensoussan2013mean} as well as the references therein for an exposition of the theory, and to \citet{CarmonaSurvery20} for a survey of (financial economics) applications.
This paper's main focus is the MFG convergence problem, i.e. the challenge of rigorously justifying and quantifying the convergence of $N$-player Nash equilibria towards the corresponding MFE. This is an issue at the heart of the MFG theory.
Along the way, we will also derive new existence results for MFGs as well as uniqueness results for both MFGs and finite player games.

\subsection{Our results}
We now present a synopsis of the main results of this work.
Suppose we are given for each $N \in \N$ a Nash equilibrium $\bm \alpha^N = (\alpha^{N,1},\dots,\alpha^{N,N})$ for the $N$-player game. Let $\bX^N = (X^{N,1},\dots,X^{N,N})$ be the corresponding equilibrium state process.
Our main result, Theorem \ref{thm:pocvolcontrol}, shows that under certain regularity and monotonicity conditions on the data, we have 
\begin{align} \label{xestintro}
\E\Big[ \sup_{0 \leq t \leq T} |X_t^{N,i} - X_t^i|^2 \Big] \leq \frac{C}{N} 
\end{align}
where $(X^i)_{i = 1,\dots,N}$ is a sequence of $\sF^0$-conditionally independent ``copies" of the mean field equilibrium state process $X$ constructed through an explicit coupling procedure, see sub-section \ref{subsec:probsetup} and in particular Definition \ref{def:condindcopies}. Because $(X^i)_{i = 1,\dots,N}$ are $\sF^0$-conditionally independent, we can interpret \eqref{xestintro} as a \textit{conditional} version of propagation of chaos. Indeed, if we abuse notation slightly and write $m_t$ also for the common $\sF_t^0$-conditional law of the random variables $(X^i_t)_{i = 1,\dots,N}$, then for each $k \in \N$, \eqref{xestintro} implies
\begin{align} \label{wassestintro}
\E\bigg[\cW_2^2\Big( \sL(X_t^{N,1},\dots,X_t^{N,k}), m_t^{\otimes k}\Big) \bigg] \leq \frac{Ck}{N}.
\end{align}
for some constant $C$ independent of $k$ and $N$.
Thus for fixed $k$, \eqref{xestintro} implies that for $N$ large and conditionally on $\sF_t^0$, $X_t^{N,1},\dots,X_t^{N,k}$ are approximately i.i.d. and their common conditional law is determined by the mean field equilibrium $m$. Under some additional regularity assumptions, we also show that
\begin{align} \label{alphaestintro}
\E\bigg[ \int_0^T |\alpha_t^{N,i} - \alpha_t^i|^2 dt\bigg] \leq \frac{C}{N}, 
\end{align}
where $(\alpha^i)_{i = 1,\dots,N}$ are $\sF^0$-conditionally independent copies of the optimizer $\alpha$ from the mean field game (see Definition \ref{def:condindcopies}). 

The estimates \eqref{xestintro} and \eqref{alphaestintro} are established by exploiting various notions of displacement monotonicity. We remind the reader that for $c \in \R$, a function $U = U(x,m)$ is $c$-displacement monotone
if the inequality 
\begin{align} \label{def:dispsemimonotone}
\E\bigg[ \Big(D_x U(\xi^1,\sL(\xi^1)) - D_x U(\xi^2, \sL(\xi^2)) \Big) \cdot \big(\xi^1 - \xi^2 \big) \bigg] \geq c \E\big[|\xi^1 - \xi^2|^2\big] 
\end{align}
holds for any sufficiently integrable random variables $\xi^1,\xi^2$ with respective laws $\sL(\xi^1), \sL(\xi^2)$.
Moreover $U$ is called displacement semi-monotone if $U$ is $c$-displacement monotone for some $c \in \R$, displacement monotone if $U$ is $0$-semi-monotone, and strictly displacement monotone if $U$ is $c$-displacement-monotone for some $c > 0$. We now discuss in more detail how we use displacement monotonicity and its generalizations to establish the estimates \eqref{xestintro} and \eqref{alphaestintro} in three distinct regimes.
\newline
\paragraph{\textbf{The case of constant volatility}.} We first establish (in Theorem \ref{thm:pocdriftcontrol}) a propagation of chaos result in the special case of linear drift control and separated running cost, i.e. the case 
\begin{align*}
b(x,a,m) = a, \quad \sigma(x,a,m) = \Sigma, \quad  \sigma^0(x,a,m) = \Sigma^0 \quad L(x,a,m) = L_0(x,a) + F(x,m). 
\end{align*}
There are two reasons to treat this case separately. First, the proof in this case is especially simple. Second, and more importantly, in this case the monotonicity condition simplifies and we are able to get a sharper result which reveals a tradeoff between convexity of $L_0$ in $(x,a)$, displacement (semi-)monotonicity of $F$ and $G$, and the size of the time horizon $T$. In particular, the main monotonicity condition is that $L_0$ is $C_L$-convex in $(x,a)$ for some $C_L > 0$, $F$ is $C_F$-displacement semi-monotone, $G$ is $C_G$-displacement semi-monotone, and 
\begin{align} \label{constantvolstruct}
C_L + T (C_G \wedge 0) + \frac{T^2 (C_F \wedge 0)}{2} > 0. 
\end{align}
Here $a \wedge b = \min \{a,b\}$. Note that $C_G$ and $C_F$ are allowed to be negative, so that the monotonicity condition does not actually require displacement monotonicity of $F$ or $G$ provided that $C_L$ is large enough or $T$ is small enough. Thus our proof strategy provides new insights even in this reasonably well-understood special case.
\newline
\paragraph{\textbf{The case of affine dynamics}.} We also treat separately the case of ``affine dynamics", by which we mean the case that the dynamics are affine in $(x,a)$, i.e. 
\begin{align*}
b(x,a,m) = B_0 + B_1 a + B_2 x, \quad \sigma(x,a,m) = \Sigma_0 + \Sigma_1 a + \Sigma_2 x, \quad \sigma^0(x,a,m) = \Sigma^0_0 + \Sigma^0_1 a + \Sigma_2^0 x
\end{align*}
but $L = L(x,m,a)$ is not necessarily separated. The main structural condition in this case is a strict version of condition
(H8) in \cite{Mes-Mou21}. In particular, we require
\begin{align} \label{affinestructintro}
    &\E\Big[\Big(D_x L(X,\alpha,\sL(X)) - D_x L(\ov{X}, \ov{\alpha}, \sL(\ov{X}) \Big) \cdot (X - \ov{X}) \nonumber \\ &\qquad + \Big(D_a L(X,\alpha,\sL(X)) - D_a L(\ov{X}, \ov{\alpha}, \sL(\ov{X}) \Big) \cdot (\alpha - \ov{\alpha}) \Big] \geq C_L \E[|\alpha - \ov{\alpha}|^2]
\end{align}
for all random variables $X, \ov{X}, \alpha, \ov{\alpha}$ of appropriate dimension and some constant $C_L > 0$, and we also assume that $G$ is displacement monotone. We note that our results in the affine case do not cover our results in the constant volatility case, because we do not see the same tradeoff between monotonicity, convexity and the size of the time horizon \eqref{constantvolstruct}. In particular, our main result in the affine case, Theorem \ref{thm:pocaffine}, implies convergence in the constant volatility case only when $C_F, C_G \geq 0$. 
\newline
\paragraph{\textbf{The general case}.}
In the general case (i.e. without assuming that the state dynamics have affine coefficients or the cost is separated) we are still able to obtain the estimates \eqref{xestintro} and \eqref{alphaestintro} under a new structural condition on the Hamiltonian of the game. This
structural condition is stated precisely in Assumption \ref{assump:controlledvolstruct}, but roughly speaking it states that $G$ is \textit{strictly} displacement monotone, and the Hamiltonian $H$ (defined in \eqref{def:hamiltonian}) satisfies the monotonicity condition 
\begin{align} \label{hammonotoneintro}
  \E\bigg[ &- \big(D_x H (X,Y,Z,Z^0,\sL(X)) - D_x H(\ov{X},\ov{Y},\ov{Z},\ov{Z}^0, \sL(\ov{X})) \big) \cdot (X - \ov{X})  \nonumber 
  \\  &+ \big(D_y H(X,Y,Z,Z^0,\sL(X)) - D_y H(\ov{X},\ov{Y},\ov{Z},\ov{Z}^0, \sL(\ov{X})) \big) \cdot (Y - \ov{Y}) \nonumber \\ &+ \big(D_z H (X,Y,Z,Z^0,\sL(X)) - D_z H(\ov{X},\ov{Y},\ov{Z},\ov{Z}^0, \sL(\ov{X})) \big) \cdot (Z - \ov{Z}) \nonumber 
  \\ & + \big(D_{z^0} H(X,Y,Z,Z^0,\sL(X)) - D_{z^0} H(\ov{X},\ov{Y},\ov{Z},\ov{Z}^0, \sL(\ov{X})) \big) \cdot (Z^0 - \ov{Z}^0) \bigg] \nonumber \\
& \quad \quad \leq - C_H \E\big[|X - \ov{X}|^2\big]
\end{align}
for some $C_H > 0$ and for any random variables $X, \ov{X}, Y, \ov{Y}, Z, \ov{Z}, Z^0, \ov{Z}^0$ of appropriate dimensions. 

The condition \eqref{hammonotoneintro} may appear technical, but it can be viewed as a natural generalization of several conditions which have already appeared in the literature. First, if we are in the special case of linear drift control: 
\begin{align*}
    b(x,a,m) = a, \quad \sigma(x,a,m) = \sigma^0(x,a,m) = I_{n \times n}, 
\end{align*}
then our ``full" Hamiltonian $H$ can be written as 
\begin{align*}
    H(x,y,z,z^0,m) = H_0(x,y,m) + \tr(z + z^0), 
\end{align*}
with $H_0$ the ``reduced" Hamiltonian given by 
\begin{align*}
    H_0(x,y,m) = \inf_{a} \Big(L(x,a,m) + y \cdot a \Big).
\end{align*}
In this case, it is easy to see that \eqref{hammonotoneintro} is (a strict version of) a generalization of the displacement monotonicity condition (2.8) of \cite{Mes-Mou21}, which is in turn a generalization of the displacement monotonicity condition proposed in \cite[Definition 3.4]{Gang-Mes-Mou-Zhang22}. Thus our condition \eqref{hammonotoneintro} can also be viewed as a generalization of the main structural condition of \cite{Gang-Mes-Mou-Zhang22} to the setting of controlled volatility. We note that \cite{Gang-Mes-Mou-Zhang22} and \cite{Mes-Mou21} are concerned with the wellposedness of the master equation and the MFG system, respectively, and in particular neither addresses the convergence problem.

The condition \eqref{hammonotoneintro} could also be viewed as a generalization of a structural condition  commonly appearing in the literature on finite-dimensional forward backward stochastic differential equations (FBSDEs). In particular, if $b, \sigma$, and $\sigma^0$ are independent of $m$, and $L$ has a separated form 
\begin{align*}
    L(x,a,m) = L_0(x,a) + F(x,m), 
\end{align*}
then we can write 
\begin{align*}
    H(x,y,z,z^0,m) = H_0(x,y,z) + F(x,m), 
\end{align*}
where 
\begin{align*}
    H_0(x,y,z,z^0) = \inf_{a} \Big(L_0(x,a) + b(x,a) \cdot y + \sigma(x,a) \cdot z + \sigma^0(x,a) \cdot z^0 \Big).
\end{align*}
In this case, it is easy to check that \eqref{hammonotoneintro} is satisfied if $F$ is (strictly) displacement monotone and $H_0$ satisfies 
\begin{align*}
    &-\Big(D_xH_0(x,y,z,z^0) - D_xH_0(\ov{x},\ov{y},\ov{z},\ov{z}^0) \Big) \cdot (x - \ov{x}) 
    + \Big(D_yH_0(x,y,z,z^0) - D_yH_0(\ov{x},\ov{y},\ov{z},\ov{z}^0) \Big) \cdot (y - \ov{y}) \\
    &+ 
    \Big(D_zH_0(x,y,z,z^0) - D_zH_0(\ov{x},\ov{y},\ov{z},\ov{z}^0) \Big) \cdot (z - \ov{z})
    + 
    \Big(D_{z^0}H_0(x,y,z,z^0) - D_{z^0}H_0(\ov{x},\ov{y},\ov{z},\ov{z}^0) \Big) \cdot (z - \ov{z}^0) \leq 0.
\end{align*}
This is exactly the monotonicity condition typically used to study stochastic control problems with controlled volatility through the maximum principle and FBSDE methods, see e.g. Section 3 of \cite{pengwumonotone} for more details.

While our main results are concerned with the convergence problem, we also establish the existence and uniqueness of equilibria for the mean field game under the same assumptions. This is done by characterizing them by the maximum principle, and solving the resulting McKean-Vlasov FBSDE via the method of continuation. These arguments are presented in the Appendix. Finally, to demonstrate the flexibility of our proof strategy, we show in Section \ref{subsec:infinite} how to adapt our methods to study the convergence problem in the setting of infinite horizon (discounted) mean field games.

\subsection{Comparison with the literature} 
\label{subec:comparison}
The first general results on the convergence of large population stochastic differential games to mean field games are due to \citet{lacker2016general} and \citet{fischer2017connection} who showed that sequences of laws of Nash equilibria in the $N$--player games admit subsequential weak limits that are (weak) solutions of the mean field game.
This statement was proved under modest conditions on the coefficients and information structure of the games, allowing for instance common noise, non--separated running cost and non-constant, (but uncontrolled) volatility.
An essential step to achieve these results is to extend the set of admissible controls to include \emph{relaxed} controls.
This approach has been generalized to considering games with closed loop controls (see \citet{Lac-LeFlem2023,lacker2020convergence}, \citet{Djete21} and \citet{Iseri-Zhang21})
and gives a fairly complete \emph{qualitative} understanding of the asymptotic behavior of symmetric stochastic differential games as the number of players goes to infinity.

Extending the qualitative, asymptotic statements of \cite{Djete21,fischer2017connection,lacker2016general,lacker2020convergence,Lac-LeFlem2023} to derive \emph{nonasymptotic rates of convergence} is a challenging problem which remains largely open.
The work of \citet{cardaliaguet2019master} provided a major breakthrough.
These authors derived a convergence rate of the value function of the $N$--player game to that of the mean field game, as well as rates for the convergence of $m^N_{\bX_t}$ to $m_t$.
These results follow as a consequence of a fine analysis of the so called \emph{master equation}, a partial differential equation written on the space of measures describing the value function of the mean field game.
Notably, \citet{cardaliaguet2019master} showed that bounding the (second) measure derivative of the solution of the master equation allows to derive quantitative convergence results.
See also \citet{Card17} and \citet{delarue2020master,delarue2019master} for extensions of this idea, notably allowing to obtain a central limit theorem and various concentration bounds.
A key condition guaranteeing bounds on the gradient of the solution of the master equation is the \emph{Lasry--Lions monotonicity} condition.
\begin{equation*}
  \int_{\R^n}\big(G(x,m^1) - G(x,m^2)\big)d(m^1 -m^2)(x)\ge0.
\end{equation*}
Essentially, this condition suggests that players reduce their cost by moving to less congested areas.
Unfortunately, even under this condition it is often very challenging to obtain (global in time) classical solution of the master equation without additional constraints on the data of the game.
For instance, only \citet{chassagneux2014probabilistic} discuss classical solvability of the master equation with non--constant volatility parameter.

A stochastic-analytic approach to the mean field game convergence problem was initiated by \citet{laurieretangpi} who argued that the problem can be reduced to a classical \emph{propagation of chaos} issue, but for interacting particles evolving forward and backward in time.
This approach has the benefit of not appealing to the master equation, but the results of \cite{laurieretangpi} require either that $T$ is small enough or that the drift $b$ is ``sufficiently dissipative" in $x$ (see the condition on the constant $K_b$ appearing in Theorem 2.2 of \cite{laurieretangpi}).
Moreover, the important case of games with common noise has not been investigated using this approach. We note that this ``forward-backward propagation of chaos" approach has since been used in other settings, for example in the recent work \cite{bayraktar2022propagation}, where it is used to study the convergence problem in the setting of graphon games.
We also refer to \cite{possamai2021non} for a similar method based on BSDEs (for games in the probabilistic weak formulation) but also requiring smallness or structural conditions on the terminal cost $G$.

The current work uses the fully probabilistic approach based on forward-backward propagation of chaos to derive quantitative mean field game convergence results.
The main novelties of our results are that (a) we treat simultaneously controlled volatility, common noise, and non-separable Hamiltonians and (b) we derive strong convergence of both the states and the controls with dimension-free rates (i.e. with rates that do not deteriorate with the dimension $n$ of the state process).
We emphasize that to the best of our knowledge, our results are the first convergence results (qualitative or quantitative) in the setting that the volatility of the common noise (the coefficient $\sigma^0$ in the game described above) is controlled.

We also provide new existence and uniqueness results for finite population games and mean field games.
Since the works of \citet{cardaliaguet2019master} and \citet{Car-Del-Lack16}, there has been an intensive activity around the solvability of mean field games with common noise.
\citet{ahuja2016} first recognized the importance of displacement monotonicity for wellposedness of MFG with common noise, and we refer also to  \cite{Gang-Mes22,Gang-Mes-Mou-Zhang22,Mes-Mou21} for other contributions related to displacement monotonicity. See also \cite{Mou-Zhang22} for alternative monotonicity conditions.
Solvability of games with controlled volatility has remained an intriguing challenge.
The present paper fills this gap by proving in Theorem \ref{thm:existence} that in addition to displacement monotonicity, the condition \eqref{hammonotoneintro} on the Hamiltonian of the game allows to prove existence and uniqueness of the controlled--volatility MFG.
Our argument relies on the continuation method of \cite{pengwumonotone} in the theory of FBSDEs.

The proofs are rather direct and seem versatile enough to be adapted to several cases.
As an illustration, we use similar techniques to address the mean field game convergence problem for infinite horizon problems in Section \ref{subsec:infinite}.

\subsection{Proof strategy and the role of displacement monotonicity}  \label{subsec:proofstrategy}

We approach the convergence problem for mean field games via the maximum principle, and in particular we follow the blueprint put forth in \cite{laurieretangpi}: 
\begin{enumerate}
\item Characterize the $N$-player game in terms of an $N$-dimensional FBSDE
\item Characterize the mean field game in terms of a (conditional) McKean-Vlasov FBSDE
\item Prove that (the solutions of) the $N$-dimensional FBSDEs converge toward (the solution of) the McKean-Vlasov FBSDE via stochastic-analytic methods
\end{enumerate}
The first two steps in the present setting are by now fairly standard, so it is the third step, which could be called \textit{forward-backward propagation of chaos}, which is the key to the argument. It is also in the third step where our approach differs significantly from that of \cite{laurieretangpi}. Indeed, in \cite{laurieretangpi} the third step is executed only under the assumption that the time horizon is sufficiently small or the drift is ``sufficiently dissipative". By contrast, the present work establishes forward-backward propagation of chaos by relying on the monotonicity conditions (\eqref{constantvolstruct}, \eqref{affinestructintro}, and \eqref{hammonotoneintro}).

To explain in more detail our strategy, we start by recalling the synchronous coupling approach introduced by \citet{sznitman} for establishing propagation of chaos for particle systems of the form
\begin{align} \label{intropartsystem}
dX_t^{N,i} = b(X_t^{N,i}, m_{\bX_t^N}^N) dt + dW_t^i, \quad X_0^{N,i} = \xi^i. 
\end{align}
Here $(\xi^i)_{i \in \N}$ is a sequence of i.i.d. initial conditions with common law $m_0$, $W^1,\dots,W^N$ are independent $d$-dimensional Brownian motions, and the unknown is $\bX^N = (X^{N,1},\dots,X^{N,N})$.
 It is expected that for fixed $k$ and $N$ large, $(X_t^{N,1},\dots,X_t^{N,k}) \approx m_t^{\otimes k}$,
where $m_t = \sL(X_t)$ is the law of $X_t$, and the process $X$ solves the McKean-Vlasov SDE 
\begin{align} \label{intromkeqn}
dX_t = b(X_t, \sL(X_t)) dt + dW_t, \quad X_0 = \xi \sim m_0. 
\end{align}
This is one interpretation of ``propagation of chaos" (at least in the special setting of i.i.d. initial data, rather than the more general case of $m_0$-chaotic initial data).
Synchronous coupling means producing a sequence of independent copies of the solution to \eqref{intromkeqn}, by solving for each $i \in \N$ the equation 
\begin{align} \label{intromkeqncopies}
  dX_t^i = b(X_t^i, \sL(X_t^i)) dt + dW_t^i, \quad X_0^i = \xi^i, 
\end{align}
where the driving Brownian motion $W^i$ and the initial condition $\xi^i$ are the same as the ones appearing in \eqref{intropartsystem}.

If \eqref{intromkeqn} admits a unique solution, then $(X_t^i)_{i \in \N}$ are i.i.d. with common law $m_t$, and convergence results for empirical measures thus imply that 
\begin{align} 
  m_t \approx m_{\bX_t}^N, %\coloneqq \frac{1}{N} \sum_{i = 1}^N \delta_{\ov{X}_t^i}. 
\end{align}
where $m_{\bX_t}^N = \frac{1}{N} \sum_{i = 1}^N \delta_{X_t^i}$.
This easily implies that the tuple $(X^1,\dots,X^N)$ almost satisfies the $N$-particle SDE \eqref{intropartsystem}, i.e. solves \eqref{intropartsystem} up to a small error term, at least if $ b(x,\cdot)$ is globally Lipschitz in $m$ (using for instance the $2$-Wasserstein metric). 
To conclude that $X_t^i \approx X_t^{N,i}$ for $N$ large (and hence propagation of chaos), it suffices to show that the SDE \eqref{intropartsystem} enjoys a certain dimension-free stability property. When $b$ is Lipschitz in both arguments, this can be accomplished by studying the process $\sum_{i = 1}^N |X_t^{N,i} - X_t^i|^2$ and applying Gronwall's inequality. Making this whole argument quantitative with the help of \cite[Theorem 1]{fournier2015rate} ultimately leads to an estimate of the form
\begin{align*}
\E\Big[\sup_{0 \leq t \leq T} |X_t^{N,i} - X_t^i|^2 \Big] \leq Cr_N%^{- \gamma}
\end{align*}
for an explicitly given rate $r_N$.% appropriate constant $\gamma > 0$.

As pointed out in \cite{laurieretangpi}, the same synchronous coupling idea applies in our setting, and we now explain this point in detail (without common noise for simplicity). In our setting, the $N$-particle SDE is replaced by an $N$-dimensional FBSDE, and the McKean-Vlasov SDE \eqref{intromkeqn} is replaced by a McKean-Vlasov FBSDE (presented in Section \ref{subsec:smp}). Using the same coupling idea discussed above, we can construct independent copies of the solution of the limiting McKean-Vlasov FBSDE, which we denote by $(X^i, Y^i, Z^i)_{i \in \N}$. We can again rely on standard results on convergence of empirical measures to conclude that for $N$ large, the triple 
\begin{align*}
(\ov{\bX}^N, \ov{\bY}^N, \ov{\bZ}^N) = \Big((X^1,\dots,X^N), (Y^1,\dots,Y^N), (Z^1,\dots,Z^N) \Big)
\end{align*}
almost solves the FBSDE describing the $N$-player game, whose solution we denote by 
\begin{align*}
(\bX^N,\bY^N,\bZ^N) = \big((X^{N,1},\dots,X^{N,N}), (Y^{N,1},\dots,Y^{N,N}), (Z^{N,1},\dots,Z^{N,N}) \big). 
\end{align*}
That is, $(\ov{\bX}^N,\ov{\bY}^N, \ov{\bZ}^N)$ solves the same equation as $(\bX^N,\bY^N,\bZ^N)$ but with a small error term. To obtain a quantitative answer to the convergence problem, we need only some dimension-free stability for the FBSDE solved by $(\bX^N,\bY^N,\bZ^N)$. 
The main obstacle in executing this ``forward-backward propagation of chaos" argument is that unlike for the SDE case discussed above, standard regularity conditions will not be enough to guarantee that the $N$-dimensional FBSDE of interest enjoys any global in time dimension-independent stability. Either a small time horizon (as in \cite{laurieretangpi}) or an additional structural condition must be assumed in order to obtain the relevant stability property.

The main insight of the present paper is that the monotonicity conditions discussed above provide the relevant stability for the $N$-dimensional FBSDE. The role of displacement monotonicity is revealed by studying the dynamics of the process 
\begin{align} \label{lyap}
\sum_{i = 1}^N \big(X_t^{N,i} - X_t^{i} \big) \cdot \big(Y_t^{N,i} - Y_t^{i} \big), 
\end{align}
a strategy which is inspired by a well-known technique for handling classical FBSDEs with ``monotone" coefficients. This technique was introduced in \cite{hupengmonotone}, and then extended and applied to stochastic control and stochastic differential games in \cite{pengwumonotone}. It has since been generalized in many directions, and we refer to Section 8.4 of the recent textbook \cite{Zhang} and the references therein for more discussion. Here it  turns out that the relevant monotonicity condition (i.e. the condition which makes studying \eqref{lyap} an effective strategy) is precisely \eqref{hammonotoneintro} (together with the displacement monotonicity of $G$). In fact, the authors became aware during the preparation of this manuscript that the same technique was used recently in \cite{bayraktar2022propagation} (see in particular Section 4.2) in a similar way to obtain propagation of chaos for graphon games, though in a setting without common noise and controlled volatility, and without making use of displacement monotonicity (instead imposing in Assumption 4.2 (ii) a smallness-type condition). See also the proof of Theorem 4.9 in \cite{jacksonlacker} for a similar technique executed in a setting related to mean field control.

To understand the role of our monotonicity conditions, e.g. condition \eqref{hammonotoneintro}, the key is to observe what happens when $X,\ov{X}, Y, \ov{Y}, Z, \ov{Z}, Z^0, \ov{Z}^0$ are discrete random variables, and their joint distribution is given by
\begin{align*}
\bP[X = x^i, \ov{X} = \ov{x}^i, Y = y^i, \ov{Y} = \ov{y}^i, Z = z^i, \ov{Z} = \ov{z}^i ] = \frac{1}{N}, \quad i = 1,\dots,N,
\end{align*}
for some
\begin{align*}
        &\bx = (x^1,\dots,x^N), \,\, \ov{\bx} = (\ov{x}^1,\dots,\ov{x}^N), \,\, \by = (y^1,\dots,y^N), \,\, \ov{\by} = (\ov{y}^1,\dots,\ov{y}^N) \in (\R^n)^N, \\
        &\bz = (z^1,\dots,z^N), \,\, \ov{\bz} = (\ov{z}^1,\dots,\ov{z}^N), \,\, \bz^0 = (z^{0,1},\dots,z^{0,N}), \,\, \ov{\bz}^0 = (\ov{z}^{0,1},\dots,\ov{z}^{0,N}) \in (\R^{n \times d})^N, 
    \end{align*}
    Testing the inequality \eqref{hammonotoneintro} in this case leads to \begin{align} \label{hammonotone2intro}
        \sum_{i = 1}^N \bigg( &- \big(D_x H(x^i,y^i,z^i,z^{0,i},m_{\bx}^N) - D_x H(\ov{x}^i,\ov{y}^i,\ov{z}^i,\ov{z}^{0,i},m_{\ov{\bx}}^N) \big) \cdot (x^i - \ov{x}^i) 
        \nonumber \\
        &+  \big(D_y H(x^i,y^i,z^i,z^{0,i},m_{\bx}^N) - D_y H(\ov{x}^i,\ov{y}^i,\ov{z}^i,\ov{z}^{0,i}, m_{\ov{\bx}}^N) \big)\cdot (y^i - \ov{y}^i) 
        \nonumber \\
         &+  \big(D_z H(x^i,y^i,z^i,z^{0,i},m_{\bx}^N) - D_z H(\ov{x}^i,\ov{y}^i,\ov{z}^i,\ov{z}^{0,i},m_{\ov{\bx}}^N) \big) \cdot (z^i - \ov{z}^i)
        \nonumber  \\
         &+  \big(D_{z^0} H(x^i,y^i,z^i,z^{0,i},m_{\bx}^N) - D_{z^0} H(\ov{x}^i,\ov{y}^i,\ov{z}^i,\ov{z}^{0,i},m_{\ov{\bx}}^N) \big)  \cdot (z^{0,i} - \ov{z}^{0,i})  \bigg) \nonumber \\
         &\qquad \leq - C_H \sum_{i = 1}^N |x^i - \ov{x}^i|^2.
    \end{align}
It turns out that \eqref{hammonotone2intro} is exactly what is needed to obtain the relevant stability by studying the dynamics of the process \eqref{lyap}. We refer to the proofs of Propositions \ref{prop:mainestdrift}, \ref{prop:mainestaffine} and \ref{prop:mainest} for the details.

\subsection{Outline of the paper} 

In section \ref{sec:results}, we discuss some notations and preliminaries, describe the stochastic differential games under consideration, and finally state precisely our main results. Section \ref{sec:proofs} contains the proofs of our main convergence results. In particular, subsection \ref{subsec:smp} reviews the relevant version of the stochastic maximum principle, subsection \ref{subsec:technical} collects some technical lemmas, and subsections \ref{subsec:proofmaindrift} and \ref{subsec:proofmain} complete the proofs of our main convergence results. Section \ref{subsec:infinite} discusses an extension to infinite horizon mean field games. Finally, the appendix contains a well-posedness result for a class of McKean-Vlasov FBSDEs.

\section{Preliminaries and main results} \label{sec:results}

\subsection{Some general notations}
We start by mentioning that we will use $|\cdot|$ to denote the usual Euclidean norm on $\R^n$. We will also use $|\cdot|$ to denote the Frobenius norm for matrices, i.e. $|A| = \tr(A^{\top} A)$ for $A \in \R^{n \times d}$. The symbol $``\cdot"$ will be used both for the Euclidean inner product and for the Frobenius inner product, so that for $A, B \in \R^{n \times d}$, $A \cdot B = \tr(A^{\top} B)$.
Moreover, {In this work $\N$ is the set of non-zero integers.}

We will write $\sP(\R^n)$ for the space of probability measures on $\R^n$, and $\sP_q(\R^n)$ for the space of probability measures with finite $q^{th}$ moment, i.e. $m \in \sP_q(\R^n)$ if 
\begin{align*}
M_q(m) \coloneqq \int_{\R^n} |x|^q m(dx) < \infty. 
\end{align*}
The space $\sP_q(\R^n)$ is endowed with the usual Wasserstein metric defined as
\begin{align}
\label{eq:def.Wasserstein}
  \cW_q(m^1,m^2) \coloneqq \inf_{(X^1,X^2)} \Big(\E\big[|X^1-X^2|^q \big] \Big)^{1/q}, 
\end{align}
where the infimum is taken over all pairs of random variables $(X^1,X^2)$ defined on a common probability space and such that $X^i \sim  m^i$. 

\subsection{Probabilistic setup}
\label{subsec:probsetup}
We fix throughout the paper a time horizon $T > 0$ and a measure $m_0 \in \spt$. The measure $m_0$ will serve as the initial distribution for the relevant controlled diffusions. For simplicity (mainly to avoid any measure-theoretic subtleties) we find it useful to work on a canonical probability space. In particular, our $N$-player game will be set on the probability space 
\begin{align*}
(\Omega^N, \sF^N, \bP^N),
\end{align*}
where 
\begin{align*}
&\Omega^N = (\R^n)^N \times \cC^{N+1}, \text{ where } \cC = C([0,T]; \R^n), \\
&\sF^N \text{ is the Borel $\sigma$-algebra on $\Omega^N$}, \\
&\bP^N = (m_0)^{\otimes N} \otimes (\bP^0)^{\otimes (N+1)}, \text{ where $\bP^0$ is the Weiner measure on $\cC$.}
\end{align*}
We write $(\bx, \bm \omega) = (x^1,\dots,x^N, \omega^0, \omega^1,\dots,\omega^N)$ for a generic element of $\Omega^N$, where $x^i \in \R^n$ and $\omega^i = (\omega^i(t))_{0 \leq t \leq T} \in \cC$. We define the canonical random variables and processes 
\begin{align*}
\xi^i(\bx, \bm \omega) = x^i, \quad W^i_t(\bx, \bm \omega) = \omega^i(t).
\end{align*}
So $(\xi^i)_{i = 1,\dots,N}$ are i.i.d. random variables with common law $m_0$, and the $(W^i)_{i = 0,\dots,N}$ are independent Brownian motions. We endow $(\Omega^N, \sF^N, \bP^N)$ with the filtration $\bbF^N = (\sF^N_t)_{0 \leq t \leq T}$, which is the augmented filtration generated by $\xi^1,\dots,\xi^N$ and $W^0, W^1,\dots,W^N$. We also define $\bbF^0 = (\sF^0_t)_{0 \leq t \leq T}$ to be the augmented filtration of the Brownian motion $W^0$, and set $\sF^0 = \sF^0_T$. We will use $\E[\cdot]$ to denote expectations with respect to any of the measures $\bP^N$. We note that we are using various notations here which suppress dependence on $N$, e.g. this is the case with the random variables denoted $\xi^i$ and the Brownian motions denoted $W^i$, as well as the filtrations denoted $\bbF^0$. This notational convention is quite natural, however, since the filtered probability space $(\Omega^N,\sF^N,\bbF^N, \bP^N)$ embeds naturally into $(\Omega^M,\sF^M,\bbF^M, \bP^M)$ for $M >  N$ and so objects defined on $(\Omega^N,\sF^N,\bbF^N, \bP^N)$ extend in a natural way to objects defined on $(\Omega^M,\sF^M,\bbF^M, \bP^M)$. For example, a random variable $\eta$ defined on $\Omega^N$ extends to one on $\Omega^M$ via $\eta(x^1,\dots,x^N,\dots,x^M, \omega^0,\dots,\omega^N,\dots,\omega^M) = \eta(x^1,\dots,x^N, \omega^0,\dots,\omega^N)$. Thus our notational convention to suppress dependence on $N$ for various objects simply expresses a preference not to distinguish between e.g. a random variable defined on $\Omega^N$ and its natural extension to $\Omega^M$. 

To discribe the limiting model we will need only two Brownian motions.
In fact, it will be set on the filtered probability space $(\Omega, \sF, \bbF, \bP) \coloneqq (\Omega^1,\sF^1, \bbF^1, \bP^1)$. We write $(x,\omega^0, \omega)$ for the general element of $\Omega$ and set $W \coloneqq W^1$ when we are working with this space. One advantage of working with these explicitly constructed spaces is that it will facilitate a coupling procedure, through which any random variable (or process) defined on $\Omega$ determines a sequence of random variables (or processes) defined on $\Omega^N$. This coupling procedure is essential later on, where it is used to prove our main ``conditional propagation of chaos" results. Therefore we outline the procedure carefully here. Observe that given an $\cF_t$--measurable random variable $\eta$ (defined on $\Omega$), we can write (up to identifying random variables which are almost surely equal) 
\begin{align*}
\eta(x, \omega^0,\omega) = \wt{\eta}(x, \omega^0_{[0,t]}, \omega_{[0,t]}) = \wt{\eta}(\xi^1, W^0_{[0,t]}, W_{[0,t]})(x, \omega^0,\omega)
\end{align*} 
where 
%$W_{[0,t]}$ denotes the path of the Brownian motion $W$ on $[0,t]$ (and likewise for $W^0_{[0,t]}$), and 
$\wt{\eta}$ is a measurable map on $\R^n \times \cC_t \times \cC_t$, with $\cC_t = C([0,t] ; \R^d)$ endowed with the usual Borel $\sigma$-algebra. This representation of $\eta$ facilitates the production of a number of (conditionally) independent copies of $\eta$. In particular, for a fixed $N \in \N$ and $i \in \{1,\dots,N\}$, we can define a random variable $\eta^i$ on $\Omega^N$ by the formula
\begin{align*}
\eta^i(\bx, \bm \omega) = \wt{\eta}(x^i, \omega^0_{[0,t]}, \omega^i_{[0,t]}) = \wt{\eta}(\xi^i, W^0_{[0,t]}, W^i_{[0,t]})(\bx, \bm \omega). 
\end{align*}
Obviously for each $N$ the random variables $(\eta^1,\dots,\eta^N)$ defined on the probability space $(\Omega^N, \sF^N, \bP^N)$ in this manner are i.i.d. conditionally on $\bbF^0$, with common conditional law
 \begin{equation} \label{conditionallaw}
  \Big(\cL(\eta^i\mid \cF^0_t)(\bx, \bm \omega) \Big)(A) =
   \P\Big[\wt{\eta}\big(\xi, W^0_{[0,t]}(\bx, \bm \omega), W^1_{[0,t]}\big) \in A \Big].
\end{equation}
The right-hand side of \eqref{conditionallaw} is the one place in the paper where using $W^0$ for a Brownian motion defined on $\Omega$ and $\Omega^N$ is confusing. So to clarify, the $W^0$ appearing in the right side of the equality \eqref{conditionallaw} is the one defined on $\bbF^N$, and it is being evaluated at the general element $(\bx, \bm \omega)$ of $\Omega^N$, and the probability on the right side could be written as 
\begin{align*}
\P\Big[\wt{\eta}\big(\xi, W^0_{[0,t]}(\bx, \bm \omega), W^1_{[0,t]}\big) \in A \Big] = \bP\bigg[ \Big\{(x,\omega^0, \omega) \in \Omega : \wt{\eta}\big(x, W_{[0,t]}^0(\bx, \bm \omega), \omega\big)  \in A \Big\} \bigg].
\end{align*}
This procedure naturally extends to processes, though it is slightly more subtle. Luckily, working with the canonical space makes things easier. Indeed, given an $\F$--progressive process $(X_t)_{t\in [0,T]}$, \cite[Proposition 1.2.1]{Zhang} shows that $X$ has a version which is progressive with respect to the \textit{raw} (i.e. not $\bP$-completed) filtration $\wt{\bbF}$ of $(\xi, W^0,W)$. 
But with this version fixed, it is clear that we have
\begin{equation*}
  X_t(x,\omega^0,\omega) = \wt{X}_t(\xi^1, W^0_{[0,t]}, W^1_{[0,t]}) (x,\omega^0,\omega),
\end{equation*}
where $\wt{X}_t : \R^d \times \cC_t \times \cC_t \to \R$ is defined simply as 
\begin{align*}
  \wt{X}_t(x, \omega^0, \omega) = X_t(x , \wt{\omega^0}, \wt{\omega}), \quad \text{where} \quad  \wt{\omega}(s) = \omega(s) 1_{s \leq t} + \omega(t) 1_{s > t}, \quad \wt{\omega^0}(s) = \omega^0(s) 1_{s \leq t} + \omega^0(t) 1_{s > t}.
\end{align*}
%\todo[inline]{something seems off here. We have $\tilde \omega=\omega$ right? JJ: $\wt{\omega}$ is supposed to be equal to $\omega$ on $[0,t]$ but then extended by a constant on $[t,T]$}
Moreover, \cite[Lemma 1.2.1]{Zhang} also shows that the $\wt{\bbF}$-progressive modification of $X$ can be chosen continuous if $X$ is continuous, so that the family of maps $\wt{X}_t$ are such that $t \mapsto \wt{X}_t(x, \omega^0_{[0,t]},\omega_{[0,t]})$ are continuous for each fixed $x \in \R^n, \omega,\omega^0 \in \cC$. Moreover, if we define
$(X^i_t)_{i\in \{1,\dots,N\}}$ by $X^i_t := \wt{X}_t(\xi^i, W^0_{[0,t]}, W^i_{[0,t]})$, then it is clear that for each fixed $t$, $(X^i_t)_{i\in \{1,\dots,N\}}$ form a family of i.i.d. random variables conditionally on $\sF^0_t$ for all $t \in [0,T]$.
This observation will play a key role in what follows, so we summarize the preceding discussion with a Lemma and Definition. 

\begin{lemma} \label{lem:coupling}
Given a $\bbF$-progressive (resp. continuous and $\bbF$-adapted) process $X$, there is a modification of $X$ which remains $\bbF$-progressive (resp. continuous and $\bbF$-adapted), and such that 
\begin{align*}
X_t(x, \omega, \omega^0) = \wt{X}_t(x, W^0_{[0,t]}, W_{[0,t]})
\end{align*}
for each $t, x, \omega, \omega^0$ and some collection of measurable maps $\wt{X}_t$ defined on $\R^d \times \cC_t \times \cC_t$. Moreover for each $N \in \N$, the processes $(X^i)_{i\in \{1,\dots,N\}}$ by $X^i_t := \wt{X}_t(\xi^i, W^0_{[0,t]}, W^i_{[0,t]})$ are $\bbF^N$-progressive (resp. continuous and $\bbF^N$-adapted), and for each fixed $t$, $(X^i_t)_{i\in \{1,\dots,N\}}$ form a family of i.i.d. random variables conditionally on $\sF^0_t$ 
\end{lemma}

\begin{definition} \label{def:condindcopies}
Given a $\bbF$-progressive (resp. continuous and $\bbF$-adapted) process $X$ and an $N \in \N$, the processes $(X^i)_{i = 1,\dots,N}$ defined in Lemma \ref{lem:coupling} are called the conditionally independent copies of $X$.
\end{definition}

We have one more measure-theoretic preliminary to discuss. Given a sufficiently regular and square integrable process $X$ defined on $\Omega$, we would like to choose regular version of the $\sP_2(\R^d)$-valued process $\big(\sL(X_t \mid \sF_t^0)\big)_{0 \leq t \leq T}$. Luckily, \cite[Lemma 2.5]{cardelbook2} gives exactly what we need. We state a consequence of their result here. 

\begin{lemma} \label{lem:conditionallaw}
Given a continuous and $\bbF$-adapted process taking values in some $\R^n$ and satisfying 
\begin{align*}
\E\Big [\sup_{0 \leq t \leq T} |X_t|^2 \Big] < \infty, 
\end{align*}
we can choose for each $t$ a version of $\sL(X_t \mid \sF_t^0)$ such that the process $\big(\sL(X_t \mid \sF_t^0)\big)_{0 \leq t \leq T}$ has continuous paths in $\spt$. 
\end{lemma}
In what follows, we will always work with such a version.

\subsection{Calculus on the space of measures}

As usual in the theory of mean field games, we will make use of a calculus for functions defined on $\spt$, which is explained in detail in \cite[Chapter 5]{cardelbook1}. Given a function $U :  \R^d \times \spt \ni (x,m) \mapsto  U(x,m)\in \R $ for some $n,d\in \N$, we will denote by $D_x U(x,m)$ the (standard) derivative in $x$ and we will say that $U$ is $C^1$ if $D_xU$ exists and is jointly continuous and there is a continuous map 
\begin{align*}
  \R^d \times \spt \times \R^n \ni (x,m,y) \mapsto D_m U(x,m,y) \in \R^n
\end{align*}
such that for each fixed $(x,m)$, $y \mapsto D_m U(x,m,y)$ is a version of the (Lions-) Wasserstein derivative of $U(x,\cdot)$ at $m$. 
As explained in \cite{cardelbook1}, this means that for each fixed $x$, the map 
\begin{align*}
  \L^2(\Omega) \ni \xi \mapsto  U(x,\sL(\xi))
\end{align*}
is (Fr\'echet) differentiable, and its gradient at the point $\xi \in \L^2(\Omega)$ is the random vector $ D_m U(x, \sL(\xi), \xi)$.

\subsection{The stochastic differential games}
\label{subsec:gamesetup}
Let us now introduce the stochastic differential games that will be analyzed in this work.
We present the game in a general framework.
Specific assumptions will be made on the coefficient to derive our respective results.
Let us fix dimensions $d,k,n \in \N$. In what follows, $d$ will be the dimension of the relevant Brownian motions, $m$ will be the dimension of the controlled diffusion, and admissible controls will be $\R^k$-valued. The data for our $N$-player and mean field games will consist of five functions
\begin{align*}
b &: \R^n \times \R^k \times \spt \to \R^n, \\
\sigma, \sigma^0 &: \R^n \times \R^k \times \spt \to \R^{n \times d}, \\
L &: \R^n \times \R^k \times \spt \to \R, \\
 G &: \R^n \times \spt \to \R, 
\end{align*}
as well as the initial measure $m_0 \in \sP_2(\R^n)$.
Although we do not make it explicit to simplify the notation, the functions $b,\sigma,\sigma^0$, and $L$ can be made time dependent if one so desires.
The conditions assumed on these functions below will have to hold uniformly in $t$.

\subsubsection{The $N$-player game}
In the $N$-player game, the state $X^i$ of player $i \in \{1,\dots,N\}$ is modeled by the controlled dynamics
\begin{align} \label{xdynamics}
  dX_t^i = b(X_t^i, \alpha_t^i, m_{\bX_t}^N) dt + \sigma(X_t^i, \alpha_t^i, m_{\bX_t}^N) dW_t^i + \sigma^0(X_t^i, \alpha_t^i, m_{\bX_t}^N) dW_t^0, \quad 
X_0^i = \xi^i. 
\end{align}
where the controlled process $\alpha$ belongs to the set of admissible controlled $\cA_N$ defined as
\begin{align} \label{def:an}
  \sA_N \coloneqq \bigg\{\text{$\bbF^N$-progressive $\R^k$-valued processes $\alpha = (\alpha_t)_{0 \leq t \leq T}$ such that $\E\Big[\int_0^T |\alpha_t|^2 dt\Big] < \infty$ } \bigg\}.
\end{align}
Here $\bX = (X^1,\dots,X^N)$, and given $\bx = (x^1,\dots,x^N) \in (\R^d)^N$, we denote $m_{\bx}^N = \frac{1}{N} \sum_{i = 1}^N \delta_{x^i}$.

Player $i$ aims at minimizing the cost functional $J^{N,i} : \sA_N^N \to \R$ defined by
\begin{align} \label{nplayerobj}
  J^{N,i}(\bm \alpha) =  J^{N,i}(\alpha^1,\dots,\alpha^N) = \E\bigg[ \int_0^T L(X_t^i, \alpha_t^i, m_{\bX_t}^N) dt + G(X_T^i, m_{\bX_T}^N) \bigg].
\end{align}
A Nash equilibrium for the $N$-player game is a tuple $\bm \alpha = (\alpha^1,\dots,\alpha^N) \in \sA_N^N$ such that for each $i \in \{1,\dots,N\}$ and each $\alpha \in \sA_N$, 
\begin{align*}
J^{N,i}(\bm \alpha ) \leq J^i(\bm \alpha^{-i}, \alpha), 
\end{align*}
where $(\bm \alpha^{-i}, \alpha)$ is shorthand for the tuple $(\alpha^1,\dots,\alpha^{i-1}, \alpha, \alpha^{i+1},\dots,\alpha^N)$. We note that given $\bm \alpha = (\alpha^1,\dots,\alpha^N) \in \sA_N^N$, we refer to the solution $\bX = (X^1,\dots,X^N)$ of \eqref{xdynamics} as the \textit{state process corresponding to $\bm \alpha$}. For simplicity, we refer to the game described above simply as \textit{the $N$-player game } which is in fact determined by the data $(b,\sigma, \sigma^0, L,G,  m_0)$.  

\subsubsection{The mean field game} 
In the corresponding mean field game, we work on the filtered probability space $(\Omega, \sF, \bbF, \bP)$ introduced in subsection \ref{subsec:probsetup}. An $\bbF^0$-progressive $\sP_2(\R^n)$-valued process $m = (m_t)_{0 \leq t \leq T}$ representing the (conditional) distribution of the population's state is fixed and the representative agent chooses an $\R^k$--valued control $\alpha$ from the set $\sA$ of admissible controls defined as
\begin{align} \label{def:a}
  \sA \coloneqq \bigg\{\text{$\bbF$-progressive $\R^k$--valued processes $\alpha = (\alpha_t)_{0 \leq t \leq T}$ such that $\E\Big[\int_0^T |\alpha_t|^2 dt\Big] < \infty$ } \bigg\}
\end{align}
in order to minimize the cost $J_m :\sA \to \R$ defined\footnote{Our assumptions below will guarantee that indeed $J^{N,i}$ and $J_m$ are finite valued.} by
\begin{align} \label{mfopt}
  J_m(\alpha) = \E\bigg[  \int_0^T L(X_t, \alpha_t, m_t) dt + G(X_T,m_T) \bigg]
\end{align}
subject to the controlled state dynamics
\begin{align} \label{dynamicsdef}
dX_t =  b(X_t, \alpha_t, m_t) dt +\sigma(X_t,\alpha_t, m_t)d W_t + \sigma^0(X_t, \alpha_t, m_t) d W_t^0, \quad X_0 = \xi. 
\end{align}
A mean field Nash equilibrium or simply a mean field equilibrium (MFE) is a pair $(m, \alpha)$ where $m$ is a continuous and $\bbF^0$-adapted $\sP_2(\R^d)$-valued process and $\alpha \in \sA$ is a minimizer of $J_{m}$ with corresponding state process $X$, such that
\begin{align} 
  m_t = \sL(X_t | \sF_t^0), \,\, d\bP\text{--a.e. for each } t\in [0,T].
\end{align}
We note that given an $\alpha \in \sA$, we call the solution $X$ to \eqref{dynamicsdef} the \textit{state process corresponding to $\alpha$}.
 For simplicity, in the remainder of the paper we refer to the limiting model described here simply as \textit{the mean field game}, which is determined by the data $(b,\sigma^0, \sigma, L,G,m_0)$.

\subsection{Main results} 

Now we state precisely the main results of the paper, first in the case of constant volatility, then in the case of affine dynamics, and then in the general case.

\subsubsection{Main results for the constant volatility case} 
\label{subsec:driftcontrol}

We consider in this section the case that $k = n$, and
\begin{align} \label{def:lineardrift}
  b(x,a,m) = a,\quad \sigma(x,a,m)= \Sigma, \quad \sigma^0(x,a,m) =\Sigma^0, \text{ and} \quad L(x,a,m) = L_0(x,a) + F(x,m) %: \R^n \times \R^n \to \R
%F, G &: \R^d \times \spt \to \R,
\end{align}
for some fixed matrices $\Sigma, \Sigma^0\in \R^{n\times d}$. So, the data of the game in this setting consists of $\Sigma, \Sigma^0 \in \R^{n \times d}$, together with $L_0 : \R^n \times \R^n \to \R$, $F,G : \R^n \times \spt \to \R$ as well as $m_0 \in \sP_2(\R^n)$. In this setting, it is convenient to work with the reduced Hamiltonian $H_0 = H_0(x,y) : \R^n \times \R^n \to \R$ given by 
\begin{align} \label{def:reducedham}
H_0(x,y) = \inf_{a \in \R^n} \big( L_0(x,a) + a \cdot y\big).
\end{align}

To state our results precisely, we begin by stating the regularity assumptions which will be used in our analysis.

\begin{assumption}[Basic regularity requirements in the constant volatility case] \label{assump:driftcontrolreg}
The data $b, \sigma, \sigma^0$ are of the form \eqref{def:lineardrift}. The Lagrangian $L_0$ is $C^1$, and the derivatives $D_xL_0$ and $D_a L_0$ are Lipschitz continuous in all arguments. The coupling terms $F$ and $G$ are $C^1$, and the maps 
\begin{align*}
&\R^n \times \spt \ni (x,m) \mapsto D_x F(x,m), \quad \R^n \times \spt \times \R^n \ni (x,m,y) \mapsto D_m F(x,m,y), \\
&\R^n \times \spt \ni (x,m) \mapsto D_x G(x,m), \quad \R^n \times \spt \times \R^n \ni (x,m,y) \mapsto D_m G(x,m,y)
\end{align*}
are each Lipschitz continuous in all arguments.
\end{assumption}

Next, we state the main structural assumption on the game. 

\begin{assumption}[Main structural condition in the constant volatility case] \label{assump:driftcontrolstruct}
There is a constant $C_L > 0$ such that
\begin{align} \label{jointconvex}
\big(D_x L_0(x,a) - D_x L_0(\ov{x},\ov{a})\big)\cdot(x - \ov{x}) + \big(D_a L_0(x,a) - D_a L_0(\ov{x},\ov{a})\big)\cdot(a - \ov{a}) \geq C_L |a - \ov{a}|^2
\end{align}
for all $x,\ov{x},a,\ov{a} \in \R^n$. 
There are (possibly negative) real numbers $C_{F}$ and $C_{G}$ such that $F$ is $ C_{F}$-displacement monotone and $G$ is $ C_{G}$-displacement monotone, and we have
\begin{align} \label{convexmonotone}
  C_L + (C_{G} \wedge 0) T + \frac{(C_{F} \wedge 0) T^2}{2}  > 0.
\end{align} 
\end{assumption}

\begin{remark}
The inequality \eqref{convexmonotone} represents the key structural condition on the game. It reveals a precise trade-off between convexity of the Lagrangian $L$, displacement (semi-)monotonicity of the coupling terms $F$ and $G$, and the size of the time horizon $T$. In particular, if $F$ and $G$ are displacement monotone (i.e. $C_F, C_G \geq 0$), then \eqref{convexmonotone} automatically holds since $C_L > 0$. On the other hand for any finite value of $C_G$ and $C_F$, \eqref{convexmonotone} holds when $T$ is small enough. 
\end{remark}

We need one additional assumption, which is needed only to apply a result of \cite{CST}.

\begin{assumption}[Smoothness condition for the constant volatility case] \label{assump:cstconstvol}
For each fixed $x$, the maps $m \mapsto D_x  F(x,m)$ and $m \mapsto D_x G(x,m)$ admit derivatives 
\begin{align*}
D_m D_x F(x,m,y) = D_m\big[D_x F(x,\cdot)](m,y), \quad D_{mm} D_x F(x,m,y,z) = D_{mm} \big[D_x F(x,\cdot)\big](m,y,z), \\
D_m D_x G(x,m,y) = D_m\big[D_x G(x,\cdot)](m,y), \quad D_{mm} D_x G(x,m,y,z) = D_{mm} \big[D_x G(x,\cdot)\big](m,y,z)
\end{align*}
which are bounded and Lipschitz continuous, uniformly in $x$.
\end{assumption}

Here is our main result in the context of constant volatility.

\begin{theorem} \label{thm:pocdriftcontrol}
Let Assumptions \ref{assump:driftcontrolreg}, \ref{assump:driftcontrolstruct}, and \ref{assump:cstconstvol} hold. Suppose that there is an MFE $(m,\alpha)$, and for each $N \in \N$ there is a Nash equilibirum $\bm \alpha^N = (\alpha^{N,1},\dots,\alpha^{N,N})$ for the $N$-player game. Let $\bX^N = (\bX^{N,1},\dots,\bX^{N,N})$ denote the state process corresponding to $\bm \alpha^N$. Then there is a constant $C$ independent of $N$ such that for each $N \in \N$ and $i \in \{1,\dots,N\}$, we have
\begin{align} \label{mainestdrift}
\E\bigg[\int_0^T |\alpha_t^{N,i} - \alpha_t^{i}|^2 dt \bigg] \leq \frac{C}{N}, 
\end{align}
where $(\alpha^i)_{i= 1,\dots,N}$ are conditionally independent copies of $\alpha$ as in Definition \ref{def:condindcopies}.
As a consequence, there is a constant $C$ independent of $N$ such that for each $N \in \N$  and $i \in \{1,\dots,N\}$, we have 
\begin{align*}
\E\Big[\sup_{0 \leq t \leq T} |X_t^{N,i} - X_t^i|^2\Big] \leq \frac{C}{N}, 
\end{align*}
where $(X^i)_{i = 1,\dots,N}$ are conditionally independent copies of $X$.
\end{theorem}

\begin{remark}
While we have stated a convergence result under simpler conditions in the special case of constant volatility, we have elected not to give a corresponding existence statement. This is because existence in the case of linear drift control and constant volatility is already fairly well understood, and existence under similar assumptions has already been obtained in \cite{ahuja2016}.
\end{remark}

\begin{remark}
We should also point out that aside from the direct argument and the trade-off between displacement semi-monotonicity, convexity and smallness that we exhibit, the main novelty of this result is that we do not assume $\sigma$ and $\sigma^0$ to be uniformly non--degenerate.
Indeed, if $\sigma$ and $\sigma^0$ are non--degenerate and the data is sufficiently smooth and displacement monotone, it follows from \cite[Theorem 5.1]{Gang-Mes-Mou-Zhang22} that the (second) measure derivative of the solution of the master equation is bounded.
This bound together with the method of proof of \cite[Section 6.2]{cardaliaguet2019master} should allow to obtain a convergence rate for the value function and a propagation of chaos result (at least for closed-loop Nash equilibria).
\end{remark}

\subsubsection{Main results for affine dynamics}

Now we consider the case of affine dynamics, i.e. the case
\begin{align} \label{affinedata}
    b(x,a,m) = B_0 + B_1 a + B_2 x, \quad \sigma(x,a,m)  = \Sigma_0 + \Sigma_1 a + \Sigma_2 x, \quad \sigma^0(x,a,m) = \Sigma^0_0 + \Sigma^0_1 a + \Sigma^0_2 x
\end{align}
for some 
\begin{align*}
    B_0 \in \R^n, \quad B_1 \in \R^{n \times k}, \quad B_2 \in \R^{n \times n}, \quad \Sigma_0, \Sigma_0^0 \in \R^{n \times d} \quad \Sigma_1,\Sigma_1^0 \in (\R^{n \times d})^{k}, \quad \Sigma_2, \Sigma_2^0 \in (\R^{n \times d})^{n}.
\end{align*}
We mention that we are interpreting $\Sigma_1 \in (\R^{n \times d})^k$ as $\Sigma_1 = (\Sigma_1^1,\dots,\Sigma_1^k)$ with each $\Sigma_1^j \in \R^{n \times d}$, and $\Sigma_1 a = \sum_{j = 1}^k a^j\Sigma_1^j \in \R^{n \times d}$. We interpret the other products in \eqref{affinedata} similarly.

Here are the basic regularity requirements in this setting.

\begin{assumption}[Basic regularity requirements in the case of affine dynamics] \label{assump:affinereg}
The data $b, \sigma, \sigma^0$ are of the form \eqref{affinedata}. The Lagrangian $L$ is $C^1$, and the derivatives $D_x L$, $D_aL$, $D_mL$ are each Lipschitz continuous in all arguments. The map $G$ is $C^1$ and its derivatives $D_x G$ and $D_m G$ are Lipschitz continuous in all arguments.
\end{assumption}

We now present the main structural condition on the Lagrangian $L$, which is the same one used in \cite{Mes-Mou21}. 

\begin{assumption}[Main structural condition in the case of affine dynamics] \label{assump:affinestruct}

The terminal condition $G$ is displacement monotone, and there is a constant $C_L > 0$ such that
\begin{align} \label{affinestruct}
\E\Big[ &\Big(D_x L(X,\alpha, \sL(X)) - D_x L(\ov{X}, \ov{\alpha}, \sL(\ov{X})) \Big) \cdot (X - X') \nonumber \\
&\qquad +  \Big(D_a L(X,\alpha, \sL(X)) - D_aL(\ov{X}, \ov{\alpha}, \sL(\ov{X})) \Big) \cdot (\alpha - \ov{\alpha})  \Big] \geq C_L \E[|\alpha - \ov{\alpha}|^2]
\end{align}
holds for any square integrable random variables $X, \ov{X}$ and $\alpha, \ov{\alpha}$ taking values in $\R^n$ and $\R^k$, respectively.
\end{assumption}

We now state an additional regularity assumption analogous to Assumption \ref{assump:cstconstvol} in the constant volatility case.

\begin{assumption}[Smoothness for affine dynamics] \label{assump:cstaffine}
For each fixed $x,a$, the maps $m \mapsto D_x  L(x,a,m)$ and $m \mapsto D_a L(x,a,m)$ admit derivatives 
\begin{align*}
D_m D_x L(x,a,m,p) &= D_m\big[D_x L(x,a,\cdot)](m,p), \quad 
D_{mm} D_x L(x,a,m,p,q) &= D_{mm} \big[D_x L(x,a,\cdot)\big](m,p,q),
\\
D_m D_a L(x,a,m,p) &= D_m\big[D_a L(x,a,\cdot)](m,p), \quad 
D_{mm} D_a L(x,a,m,p,q) &= D_{mm} \big[D_a L(x,a,\cdot)\big](m,p,q)
\end{align*}
which are bounded and Lipschitz continuous in $m$, uniformly in $(x,a)$. Similarly, for each fixed $x$ the map $m \mapsto D_x G(x,m)$ admits two derivatives 
\begin{align*}
D_m D_x G(x,m,y) = D_m\big[D_x G(x,\cdot)](m,y), \quad D_{mm} D_x G(x,m,y,z) = D_{mm} \big[D_x G(x,\cdot)\big](m,y,z)
\end{align*}
which are bounded and Lipschitz continuous in $m$, uniformly in $x$. 
\end{assumption}

We now present our main results in the case of affine dynamics. The first is an existence result.

\begin{theorem} \label{thm:existenceaffine}
Let Assumption \ref{assump:affinereg} and \ref{assump:affinestruct} hold. Then there exists a unique mean field equilibirum.
\end{theorem}

The proof of this statement is discussed in the appendix, see Remark \ref{rem.proof.affine}
\begin{remark}
  Theorem \ref{thm:existenceaffine} follows from the solvability of a conditional McKean-Vlasov FBSDE (see Lemma \ref{lem:mpmfggen}) thanks to a relevant version of the stochastic maximum principle.
  Note that the convexity properties of the maps $x\mapsto G(x,m)$  and $(x,a)\mapsto b(x,a,m) \cdot y + \sigma(x,a,m) \cdot z + \sigma^0(x,a,m) \cdot z^0 + L(x,a,m)$ for each $m$ which are needed to make the maximum principle sufficient are already guaranteed by displacement monotonicity of $G$ and the condition \eqref{assump:affinestruct}, see e.g. \cite{Mes-Mou21}.
\end{remark}
The following is our main convergence result in the case of affine dynamics.

\begin{theorem} \label{thm:pocaffine}
Let Assumptions \ref{assump:affinereg}, \ref{assump:affinestruct}, and \ref{assump:cstaffine} hold. Suppose that there is an MFE $(m,\alpha)$, and for each $N \in \N$ there is a Nash equilibirum $\bm \alpha^N = (\alpha^{N,1},\dots,\alpha^{N,N})$ for the $N$-player game. Let $\bX^N = (\bX^{N,1},\dots,\bX^{N,N})$ denote the state process corresponding to $\bm \alpha^N$. Then there is a constant $C$ independent of $N$ such that for each $N \in \N$ and $i \in \{1,\dots,N\}$, we have
\begin{align*}
\E\bigg[\int_0^T |\alpha_t^{N,i} - \alpha_t^{i}|^2 dt \bigg] \leq \frac{C}{N}, 
\end{align*}
where $(\alpha^i)_{i= 1,\dots,N}$ are conditionally independent copies of $\alpha$ as in Definition \ref{def:condindcopies}.
As a consequence, there is a constant $C$ independent of $N$ such that for each $N \in \N$  and $i \in \{1,\dots,N\}$, we have 
\begin{align*}
\E\Big[\sup_{0 \leq t \leq T} |X_t^{N,i} - X_t^i|^2\Big] \leq \frac{C}{N}, 
\end{align*}
where $(X^i)_{i = 1,\dots,N}$ are conditionally independent copies of $X$.
\end{theorem}

\subsubsection{Main results for the general case}

We now turn to the general case, and to state results we need to introduce the Hamiltonian
\begin{align} 
H   : \R^n \times \R^n \times\R^{n\times d} \times \spt \to \R
\end{align}
 defined by 
\begin{align} \label{def:hamiltonian}
  H(x,y,z,z^0,m) := \inf_{a \in \R^k} \bigg\{b(x,a,m) \cdot y + \sigma(x,a,m) \cdot z + \sigma^0(x,a,m) \cdot z^0 + L(x,a,m) \bigg\}. 
\end{align}
We recall that we use ``$\cdot$" to denote the Frobenious inner product between matrices in addition to the usual inner product on vectors, so that e.g. $\sigma(x,a,m) \cdot z := \tr(\sigma_t^\top(x,a,m) z)$, with $\sigma^\top$ denoting the transpose of the matrix $\sigma$. 
%In order to explain our results we need to state several assumptions. The first is concerned with basic regularity required on the data. 
We now present and discuss the necessary assumptions.

\begin{assumption}[Basic regularity requirements] \label{assump:controlledvolreg}
The functions $b$, $\sigma$, $\sigma^0$ are $C^1$ with bounded derivatives. The Lagrangian $L$ is $C^1$ and the derivatives $D_x L$, $D_a L$, $D_m L$ are each Lipschitz continuous in all arguments.
The Hamiltonian $H$ is $C^1$, and its derivatives
\begin{align*}
    D_x H, \,\, D_y H, \,\, D_z H, \,\, D_{z^0} H, \,\, D_m H
\end{align*}
are each Lipschitz continuous 
 $\R^n \times \R^n \times \R^{n \times d} \times \spt$.  
The terminal condition $G$ is $C^1$, and its derivatives $D_x G$ and $D_mG$ are Lipschitz in all arguments.
\end{assumption}

\begin{remark}
The assumption that $b$, $\sigma$, $\sigma^0$ are Lipschitz ensures that the dynamics of the state process (see \eqref{xdynamics}) are well-posed for each choice of $\alpha$, and the smoothness assumptions on $b$, $\sigma$, $\sigma^0$, and $L$ ensure that the maximum principle can be applied (see Lemmas \ref{eq:nplayerfbsde} and \ref{lem:mpmfggen}). 

%The regularity assumptions on $L$, $F$, and $G$ are stringent but not surprising given that we proceed through the maximum principle (the maximum principle for the $N$-player game cannot even be stated without access to the derivatives $D_x L$, $D_xF$, $D_xG$, $D_m F$, $D_m G$).

The Lipschitz continuity of various derivatives of $H$ and $G$ are technical conditions required in the proof of our main convergence result, Theorem \ref{thm:pocvolcontrol}. Lipschitz continuity of $D_m H$ and $D_m G$ is required only for the uniqueness result (Lemma \ref{lem:uniqueness} below), which is used in the proof of Theorem \ref{thm:pocvolcontrol} only to guarantee a certain symmetry of the $N$-player Nash equilibria (see Lemma \ref{lem:exchangeable}) - thus if symmetry is added as an assumption, Theorem \ref{thm:pocdriftcontrol} can be recovered without the assumption that $D_m H$ and $D_m G$ are Lipschitz.
\end{remark}

Our next assumption is the main structural condition required on the game.

\begin{assumption}[Displacement monotonicity] \label{assump:controlledvolstruct}
There is a real number $C_H > 0$ such that
\begin{align} \label{hamiltonianmonotone}
  \E\bigg[ &- \big(D_x H (X,Y,Z,Z^0,\sL(X)) - D_x H(\ov{X},\ov{Y},\ov{Z},\ov{Z}^0, \sL(\ov{X})) \big) \cdot (X - \ov{X})  \nonumber 
  \\  &+ \big(D_y H(X,Y,Z,Z^0,\sL(X)) - D_y H(\ov{X},\ov{Y},\ov{Z},\ov{Z}^0, \sL(\ov{X})) \big) \cdot (Y - \ov{Y}) \nonumber \\ &+ \big(D_z H (X,Y,Z,Z^0,\sL(X)) - D_z H(\ov{X},\ov{Y},\ov{Z},\ov{Z}^0, \sL(\ov{X})) \big) \cdot (Z - \ov{Z}) \nonumber 
  \\ & + \big(D_{z^0} H(X,Y,Z,Z^0,\sL(X)) - D_{z^0} H(\ov{X},\ov{Y},\ov{Z},\ov{Z}^0, \sL(\ov{X})) \big) \cdot (Z^0 - \ov{Z}^0) \bigg] \nonumber \\
& \quad \quad \leq - C_H \E\big[|X - \ov{X}|^2\big]
\end{align}
holds for any square-integrable random variables $X,\ov{X}, Y,\ov{Y}, Z, \ov{Z}, Z^0, \ov{Z}^0$ with $X, \ov{X}, Y, \ov{Y}$ taking values $\R^n$ and $Z, \ov{Z}, Z^0, \ov{Z}^0$ taking values in $\R^{n \times d}$. In addition, the terminal cost $G$ is strictly displacement monotone, i.e. $C_G$-displacement monotone for some $C_G > 0$. Finally, for each $(x,y,z,z^0, m)$ there is a unique $\hat{a}(x,y,z,z^0,m) \in \R^k$ which minimizes 
\begin{align} \label{hconvex}
  a \mapsto b(x,a,m) \cdot y + \sigma(x,a,m)\cdot z  + \sigma^0(x,a,m)\cdot z^0 + L(x,a,m). 
\end{align}
\end{assumption}

There is a final technical condition we will need to make on the regularity of the maps $F$ and $G$. It is required only so that we can apply a result of \cite{CST} in order to bound quantities like 
\begin{align*}
D_x G(x,m) - \E\Big[D_x G\big(x,\frac{1}{N} \sum_{i = 1}^N \delta_{\xi^i}\big) \Big]
\end{align*}
when $(\xi^i)_{i = 1,\dots,N}$ are i.i.d. with law $m$ (see Lemma \ref{lem:errorterms}).

\begin{assumption}[Smoothness condition in the general case] \label{assump:cst}
For each fixed $(x,y,z, z^0)$, the maps $m \mapsto D_x H(x,y,z,z^0,m)$, $m \mapsto D_y H(x,y,z,z^0,m)$, $m \mapsto D_z H(x,y,z,z^0,m)$, and $m \mapsto D_{z^0} H(x,y,z,z^0,m)$ admit two Lions derivatives which are bounded and Lipschitz continuous in $m$, uniformly in $(x,y,z,z^0)$. That is, the derivatives
\begin{align*}
D_m D_x H(x,y,z,z^0,m,p) &= D_m\big[D_x H(x,y,z,z^0,\cdot)](m,p), \\
D_{mm} D_x H(x,y,z,z^0,m,p,q) &= D_{mm} \big[D_x H(x,y,z,z^0,\cdot)\big](m,p,q)
\end{align*}
exist and the maps $(m,p) \mapsto D_m D_x H(x,y,z,z^0,m,p)$, $(m,p,q) \mapsto D_{mm} D_x H(x,y,z,z^0,m,p,q)$ are bounded and Lipschitz in $m$ uniformly in $(x,y,z,z^0)$, and similarly, the functions $D_y H$, $D_z H$, and $D_{z^0} H$ are bounded and Lipschitz continuous, uniformly in $(x,y,z,z^0)$.
Also, for each fixed $x$, the map $m \mapsto D_x G(x,m)$ admits two bounded derivatives 
\begin{align*}
D_m D_x G(x,m,y) = D_m\big[D_x G(x,\cdot)](m,y), \quad D_{mm} D_x G(x,m,y,z) = D_{mm} \big[D_x G(x,\cdot)\big](m,y,z)
\end{align*}
which are bounded and Lipschitz continuous in $m$, uniformly in $x$. 
\end{assumption}

\begin{remark}
As mentioned above, the existence of linear derivatives for $D_xH$ and $D_x G$ satisfying certain linear growth conditions is used only to apply \cite[Theorem 2.11]{CST} in the proof of Theorems \ref{thm:pocdriftcontrol} and \ref{thm:pocvolcontrol}. If the MFE which is supposed to exist in the statement of Theorem \ref{thm:pocdriftcontrol} has some higher integrability, e.g. $\sup_{0 \leq t \leq T} \E[M_q(m_t)]<\infty$  for some $q > 2$, then we can remove entirely the conditions concerning the Lions' derivatives of $D_x F$ and $D_x G$ and still obtain a quantitative convergence rate by relying instead on the results of \cite{fournier2015rate}, see Corollary \ref{cor:q.moment}. But while we are able to prove the existence in Theorem  \ref{thm:existence} of MFE, it seems very subtle to show that the corresponding state has higher moments when the initial condition $m_0$ does. We use the additional smoothness in Assumption \ref{assump:cst} to get around this issue.
\end{remark}
% \todo[inline]{somewhere perhaps after the proof (if it works), I propose to state two corollaries. One assuming $X$ has finite $q$-moment and one assuming the coupling is simpler, e.g. $F(X_t,\E[X_t\mid \sF^0_t])$. JJ: sounds good to me! Please feel free to add these as you are going through that par of the document}
We are now prepared to state our main results in the general case. The first is concerned with well-posedness of the mean field game. 

\begin{theorem} \label{thm:existence}
Let Assumption \ref{assump:controlledvolreg} and \ref{assump:controlledvolstruct} hold. Suppose in addition that
\begin{align*}
(x,a) \mapsto b(x,a,m) \cdot y + \sigma(x,a,m) \cdot z + \sigma^0(x,a,m) \cdot z^0 + L(x,a,m)
\end{align*}
is convex for each fixed $m$. Then there exists a unique mean field equilibirum.
\end{theorem}

The proof of this result follows by unique solvability of a (conditional) McKean-Vlasov FBSDE, as is customary in the probabilistic approach to MFG \cite{cardelbook2}.
Observe that the convexity conditions assumed in Theorem \ref{thm:existence} are not needed to guarantee well-posedness of the FBSDE, but rather to ensure that the maximum principle gives a sufficient condition for optimality and hence existence of the MFE.

Next we state our main convergence result.

\begin{theorem} \label{thm:pocvolcontrol}
Let Assumptions \ref{assump:controlledvolreg}, \ref{assump:controlledvolstruct}, and \ref{assump:cst} hold. Suppose that there is a MFE $(m,\alpha)$, and for each $N$ there is a Nash equilibrium $\bm \alpha^N = (\alpha^{N,1},\dots,\alpha^{N,N})$ for the $N$-player game. Let $\bX = (X^{N,1},\dots,X^{N,N})$ be the state process corresponding to $\bm \alpha^N$, and let $X$ be the state process corresponding to $\alpha$. Then there is a constant $C$ independent of $N$ such that for each $N \in \N$ and $i \in \{1,\dots,N\}$, 
\begin{align*}
\E\Big[\sup_{0 \leq t \leq T} |X_t^{N,i} - X_t^i|^2\Big] \leq \frac{C}{N}, 
\end{align*}
where $(X^i)_{i = 1,\dots,N}$ are conditionally independent copies of $X$ (see Definition \ref{def:condindcopies}). Suppose that in addition the optimizer $\widehat{\alpha} = \widehat{\alpha}(x,y,z,z^0,m)$ satisfies
\begin{enumerate}
\item $\hat{\alpha}$ is Lipschitz
\item For each fixed $(x,y,z,z^0)$, the map $m \mapsto \widehat{\alpha}(x,y,z,z^0,m)$ admits derivatives 
\begin{align*}
D_m \widehat{\alpha}(x,y,z,z^0,m,p) = D_m\big[\widehat{\alpha}(x,y,z,z^0, \cdot)](m,p), \\
 D_{mm} \widehat{\alpha}(x,y,z,z^0,m,p,q) = D_{mm}\big[\widehat{\alpha}(x,y,z,z^0, \cdot)](m,p,q)
\end{align*}
which are bounded and Lipschitz continuous, uniformly in $(x,y,z,z^0)$.
\end{enumerate}
Then there is a constant $C$ independent of $N$ such that
\begin{align*}
\E\bigg[\int_0^T |\alpha_t^i - \alpha_t^{N,i}|^2 dt \bigg] \leq \frac{C}{N}, 
\end{align*}
where $(\alpha^i)_{i = 1,\dots,N}$ are conditionally independent copies of $\alpha$.
\end{theorem}

As mentioned above, the strong regularity conditions of Assumption \ref{assump:cst} can be replaced by higher moments, as explained in the following corollaries of the proof of Theorem \ref{thm:pocvolcontrol}. 

\begin{corollary}
\label{cor:q.moment}
 Let the hypotheses of Theorem \ref{thm:pocvolcontrol} hold, except possibly Assumption \ref{assump:cst}. Suppose that the MFE $(m,\alpha)$ is such that $M_q(m_0) < \infty$ and moreover $\sup_{t\in [0,T]}\E[M_q(m_t)]<\infty$ for some $q>2$. Then there is a constant $C$ independent of $N$ such that for each $N \in \N$ and $i \in \{1,\dots,N\}$, we have
\begin{align*}
\E\Big[\sup_{0 \leq t \leq T} |X_t^{N,i} - X_t^i|^2\Big]\le  Cr_{N,q}
\end{align*}
where $X^{N,i}$, $X^i$ are as in the statement of Theorem \ref{thm:pocvolcontrol}, and $r_{N,q}$ is given by 
 \begin{align} \label{rnqdef}
  r_{N,q} := 
  \begin{cases}
    N^{-1/2} + N^{-(q-2)/q} & n < 4 \text{ and } q \neq 4 \\
    N^{-1/2} \log(1 + N) + N^{-(q-2)/q} & n = 4 \text{ and } q \neq 4 \\
    N^{-2/n} + N^{-(q-2)/2} & n > 4.
  \end{cases}
\end{align}
If in addition the optimizer $\widehat{\alpha} = \widehat{\alpha}(x,y,z,z^0,m)$ of the Hamiltonian is Lipschitz, then there is a constant $C$ independent of $N$ such that
\begin{align*}
\E\bigg[\int_0^T |\alpha_t^i - \alpha_t^{N,i}|^2 dt \bigg] \leq C r_{N,q},  
\end{align*}
where $\alpha^i$ and $\alpha^{N,i}$ are as in the statement of Theorem \ref{thm:pocvolcontrol}.
\end{corollary}

\begin{remark}
In order to apply Corollary \eqref{cor:q.moment}, one must verify that the MFE has ``higher moments" in an appropriate sense. One would expect that if $M_q(m_0) < \infty$, then $\sup_{t \in [0,T]} \E[M_q(m_t)] < \infty$, but we are unable to verify this in general. However, in the case of uncontrolled volatility, it is typically possible to verify (e.g. via regularity estimates for the MFG system) that if the data is sufficiently regular with bounded second derivatives, then any MFE $(m,\alpha)$ satisfies $|\alpha| \leq C(1 + |X|)$, with $X$ being the corresponding state process, from which it follows that $\sup_{t \in [0,T]} \E[M_q(m_t)] < C(1 + M_q(m_t))$. Thus in the case of uncontrolled volatility, it is straightforward to find conditions under which Corollary \ref{cor:q.moment} can be applied. The subtlety in the case of uncontrolled volatility is that while the method of continuation shows that the McKean-Vlasov FBSDE \eqref{eq:mkvfbsde} is well-posed in an $L^2$-sense, the same techniques do not lead to an $L^q$ theory. In fact, well-posedness in $L^q$ for fully coupled FBSDEs is known to be challenging even in the classical (non-McKean-Vlasov) setting, see e.g. \cite{XIE2020123642} and the references therein.
\end{remark}

\section{Proofs of the main convergence results}
\label{sec:proofs}
In this section we present the proofs of Theorems \ref{thm:pocdriftcontrol}, \ref{thm:pocaffine} and \ref{thm:pocvolcontrol}. Some of the results in this section will be stated simultaneously for the three cases of interest discussed in the previous section, and so it will be useful to state the following Assumption.
\begin{assumption}[Regularity + structural conditions] \label{assump:mixed}
Either
\begin{itemize}
    \item Assumptions \ref{assump:driftcontrolreg} and \ref{assump:driftcontrolstruct} hold, or
    \item Assumptions \ref{assump:affinereg} and \ref{assump:affinestruct} hold, or 
    \item Assumptions \ref{assump:controlledvolreg} and \ref{assump:controlledvolstruct} hold.
\end{itemize}
\end{assumption}

We note that under our Assumptions \ref{assump:driftcontrolstruct}, \ref{assump:affinestruct}, or \ref{assump:controlledvolstruct}, there is for each $(y,z,z^0,m)$ a unique minimizer $\hat{\alpha}(x,y,z,z^0,m)$ of 
\begin{align} \label{alphadef}
a \mapsto b(x,a,m) \cdot y + \sigma(x,a,m) \cdot z + \sigma^0(x,a,m) \cdot z^0 + L(x,a,m), 
\end{align}
and the map $\hat{\alpha}$ is continous. Indeed, this is included in Assumption \ref{assump:controlledvolstruct}, while it can easily be proved under either Assumption \ref{assump:driftcontrolstruct} or Assumption \ref{assump:affinestruct} (see e.g. Lemma 2.1 of \cite{cardelsicon}). Thus in the remainder of the paper, we denote by $\hat{\alpha}$ the unique minimizer of \eqref{alphadef} whenever Assumptions \ref{assump:driftcontrolstruct}, \ref{assump:affinestruct}, or \ref{assump:controlledvolstruct} are in force.

\subsection{The stochastic maximum principle}
\label{subsec:smp}

We start by recalling that
Nash equilibria for the $N$-player game are characterized by the FBSDE
\begin{align} \label{eq:nplayerfbsde}
  \begin{cases} \vspace{.1cm}
    dX_t^{N,i} = D_y H(X_t^{N,i}, Y_t^{N,i}, Z_t^{N,i,i}, Z_t^{N,i,0}, m_{\bX_t^N}^N) dt + D_z H(X_t^{N,i}, Y_t^{N,i}, Z_t^{N,i,i}, Z_t^{N,i,0}, m_{\bX_t^N}^N) dW_t^i   \\ \vspace{.1cm}
  \hspace{2cm} + D_{z^0} H(X_t^{N,i}, Y_t^{N,i}, Z_t^{N,i,i}, Z_t^{N,i,0}, m_{\bX_t^N}^N) dW_t^0, \\ \vspace{.1cm}
dY_t^{N,i} = -  \Big(D_x H(X_t^{N,i}, Y_t^{N,i}, Z_t^{N,i,i}, Z_t^{N,i,0}, m_{\bX_t^N}^N) + \frac{1}{N} D_m H(X_t^{N,i}, Y_t^{N,i}, Z_t^{N,i,i}, Z_t^{N,i,0}, m_{\bX_t^N}^N) \Big)  dt  \\
\hspace{2cm} + \sum_{j = 1}^N Z_t^{N,i,j}  dW_t^j + Z_t^{N,i,0}  dW_t^0 \vspace{.1cm}  \\
X_0^i = \xi^i, \quad Y_T^i = D_x G(X_T^{N,i}, m_{\bX_T^N}^N) + \frac{1}{N} D_m G(X_T^{N,i}, m_{\bX_T^N}^N, X_T^{N,i}).
\end{cases} 
\end{align}
By a solution to \eqref{eq:nplayerfbsde}, we mean a tuple
\begin{align*}
(\bX^N, \bY^N, \bZ^N) = \big( (X^{N,1},\dots,X^{N,N}), (Y^{N,1},\dots,Y^{N,N}), (\bZ^{N,1},\dots,\bZ^{N,N})\big)
\end{align*} 
of $\bbF^N$-progressive processes such that 
\begin{enumerate}
\item $X^{N,i}$ and $Y^{N,i}$ are continuous and adapted $\R^n$-valued processes for each $i = 1,\dots,N$,
\item $\bZ^{N,i} = (Z^{N,i,0},Z^{N,i,1},\dots,Z^{N,i,N})$, with each $Z^{N,i,j}$ for $i = 1,\dots,N$, $j = 0,\dots,N$ being a $\bbF^N$-progressive $\R^{n \times d}$-valued process
\item $(\bX^N,\bY^N,\bZ^N)$ satisfy the integrability condition 
\begin{align} \label{intconditions2}
\E\bigg[\sup_{0 \leq t \leq T} \big( |\bX_t^N|^2 + |\bY_t^N|^2 \big) + \int_0^T |\bZ_t^N|^2 dt \bigg] < \infty
\end{align}
as well as (the integral form of) the equation \eqref{eq:nplayerfbsde}. 
\end{enumerate}

In particular, we have the following Lemma, which follows from \cite[Theorem 2.15]{cardelbook1}. 
\begin{lemma}
\label{lem:N.player.smp}
  Suppose that Assumption \ref{assump:mixed} holds. 
  Let ${\bm \alpha}^N = (\alpha^{N,1},\dots,\alpha^{N,N})$ be a Nash equilibrium for the $N$-player game. Then for each $i=1,\dots,N$, we have
\begin{align*}
  \alpha_t^i = \widehat{\alpha}(X_t^{N,i}, Y_t^{N,i}, Z_t^{N,i,i}, Z_t^{N,i,0}, m_{\bX_t^N}^N), \,\, dt \times d\bP \,\, a.e.
\end{align*}
for some solution 
\begin{align*}
(\bX^N, \bY^N, \bZ^N) = \big( (X^{N,1},\dots,X^{N,N}), (Y^{N,1},\dots,Y^{N,N}), (\bZ^{N,1},\dots,\bZ^{N,N}) \big)
\end{align*}
to \eqref{eq:nplayerfbsde}.
\end{lemma}

%\begin{remark}
%    The reader may notice that the FBSDE appearing in Lemma \ref{lem:nplayergameest} seems slightly different from the one in Theorem 2.15 of \cite{cardelbook1}. But they are in fact the same, thanks to the fact that
%    \begin{align*}
%        &D_x H(x,y,z,z^0,m) = D_x L\big(x,\hat{a}(x,y,z,z^0,m), m\big), \\
%        &D_y H(x,y,z,z^0,m) = b\big(x,\hat{a}(x,y,z,z^0,m),m \big), \\
%         &D_z H(x,y,z,z^0,m) = \sigma\big(x,\hat{a}(x,y,z,z^0,m),m \big), \\
%          &D_{z^0} H(x,y,z,z^0,m) = \sigma^0\big(x,\hat{a}(x,y,z,z^0,m),m \big), 
%    \end{align*}
 %   a fact which is easily proved and is commonly used in the literature (see e.g. Section 3 of \cite{pengwumonotone}).
%\end{remark}

\begin{remark} \label{rmk:fbsdedriftcontrol}
Under Assumptions \ref{assump:driftcontrolreg} and \ref{assump:driftcontrolstruct} (i.e. in the case of constant volatility), Lemma \ref{lem:N.player.smp} shows that any Nash equilibrium $\bm \alpha = (\alpha^1,\dots,\alpha^N)$ for the $N$-player game must satisfy 
\begin{align*}
    \alpha_t^{N,i} = D_y H_0(X_t^{N,i}, Y_t^{N,i})
\end{align*}
for some tuple $(\bX^N, \bY^N, \bZ^N, \bZ^{N,0})$ satisfying
\begin{align} \label{eq:nplayerfbsdedrift}
  \begin{cases} \vspace{.1cm}
    dX_t^{N,i} = \alpha_t^{N,i} dt + \Sigma dW_t^i + \Sigma^0 dW_t^0   \\ \vspace{.1cm}
dY_t^{N,i} = -  \Big(D_x L_0(X_t^{N,i}, \alpha_t^{N,i}) + D_x F(X_t^{N,i}, m_{\bX_t^N}^N) \\
\qquad + \frac{1}{N} D_m F(X_t^{N,i}, m_{\bX_t^N}^N, X_t^{N,i}) \Big)  dt  
+ \sum_{j = 1}^N Z_t^{N,i,j}  dW_t^j + Z_t^{N,i,0}  dW_t^0 \vspace{.1cm}  \\
X_0^i = \xi^i, \quad Y_T^i = D_x G(X_t^{N,i} m_{\bX_T^N}^N) + \frac{1}{N} D_m G(X_T^{N,i}, m_{\bX_T^N}^N, X_T^{N,i})
\end{cases} 
\end{align}
where $H_0$ is the reduced Hamiltonian defined by \eqref{def:reducedham}. 
\end{remark}

\begin{remark}
\label{rmk:fbsdeaffine}
Under Assumptions \ref{assump:affinereg} and \ref{affinestruct} (i.e. in the case of affine dynamics), Lemma \ref{lem:N.player.smp} shows that any Nash equilibrium $\bm \alpha^N = (\alpha^{N,1},\dots,\alpha^{N,N})$ for the $N$-player game must take the form 
\begin{align*}
    \alpha_t^{N,i} = \widehat{\alpha}(X_t^{N,i}, Y_t^{N,i}, Z_t^{N,i},Z_t^{N,i,0}, m_{\bX_t}^N)
\end{align*}
for some tuple $(\bX^N,\bY^N, \bZ^N, \bZ^{N,0})$ satisfying
\begin{align} \label{eq:nplayerfbsdeaffine}
  \begin{cases} \vspace{.1cm}
    dX_t^{N,i} = \Big(B_0 + B_1 \alpha_t^{N,i} + B_2 X_t^{N,i} \Big) dt + \Big( \Sigma_0 + \Sigma_1 \alpha_t^{N,i} + \Sigma_2 X_t^{N,i}\Big) dW_t^i \\ \vspace{.1cm} \qquad \qquad + \big(\Sigma^0_0 + \Sigma^0_1 \alpha_t^{N,i} + \Sigma^0_2 X_t^{N,i}\big) dW_t^0   \\ \vspace{.1cm}
dY_t^{N,i} = -  \Big(D_x L(X_t^{N,i},\alpha_t^{N,i}, m_{\bX_t^N}^N) + \frac{1}{N} D_m L(X_t^{N,i},\alpha_t^{N,i} m_{\bX_t^N}^N, X_t^{N,i})  \\ \vspace{.1cm}
\qquad \qquad + B_2^\top Y_t^{N,i} + \Sigma_2^\top Z_t^{N,i,i} + (\Sigma_2^0)^\top Z_t^{N,i,0}  \Big)  dt \\ \vspace{.1cm} \qquad \qquad  
+ \sum_{j = 1}^N Z_t^{N,i,j}  dW_t^j + Z_t^{N,i,0}  dW_t^0 \vspace{.1cm}  \\
X_0^i = \xi^i, \quad Y_T^i = D_x G(X_t^{N,i} m_{\bX_T^N}^N) + \frac{1}{N} D_m G(X_T^{N,i}, m_{\bX_T^N}^N, X_T^{N,i}).
\end{cases} 
\end{align}
We interpret the term $\Sigma_2^T Z_t^{N,i,i}$ as an element of $\R^n$ whose $j^{th}$ element is given by $\Sigma_2^j \cdot Z_t^{N,i,i}$, and similar notational conventions will be used in what follows.
\end{remark}

Similarly, the maximum principle relates equilibria of the mean field game to solutions of the (conditional) McKean--Vlasov FBSDE 
\begin{align} \label{eq:mkvfbsde}
  \begin{cases} \vspace{.1cm}
    dX_t = D_y H(X_t, Y_t, Z_t, Z_t^0, \sL(X_t | \sF_t^0)) dt + D_z H(X_t, Y_t, Z_t, Z_t^0, \sL(X_t | \sF_t^0)) dW_t  \\ \vspace{.1cm}
    \qquad + D_{z^0} H(X_t, Y_t, Z_t, Z_t^0, \sL(X_t | \sF_t^0)) dW_t^0,  \\ \vspace{.1cm}
  dY_t = - D_x H(X_t, Y_t, Z_t, Z_t^0, \sL(X_t | \sF_t^0)) dt + Z_t  dW_t + Z_t^{0}  dW_t^0, \\
  X_0 = \xi, \quad Y_T = D_x G\big(X_t, \sL(X_T | \cF_T^0)\big).%, \vspace{.1cm} \\
  %m_t = . 
\end{cases}
\end{align}
By a solution to \eqref{eq:mkvfbsde}, we mean a tuple $(X,Y,Z,Z^0)$ such that $X$ and $Y$ are continuous and adapted $\R^n$-valued processes, $Z$ and $Z^0$ are progressive $\R^{n \times d}$-valued processes
 satisfying \eqref{eq:mkvfbsde} and the integrability condition 
\begin{align} \label{intconditionmkv}
\E\bigg[\sup_{0 \leq t \leq T} \big(|X_t|^2 + |Y_t|^2\big) + \int_0^T \Big(|Z_t|^2 +|Z_t^0|^2 \Big) dt \bigg] < \infty. 
\end{align}

More precisely, we have the following lemma, which is a straightforward generalization of \cite[Theorem 3.27]{cardelbook1}.

\begin{lemma} \label{lem:mpmfggen}
  Suppose that Assumption \ref{assump:mixed} holds, and let $(m,\alpha)$ be a MFE. Then
  \begin{align}
  \label{eq:char.mfe}
   \alpha_t = \widehat{\alpha}(X_t,Y_t,Z_t,Z_t^0,m_t) \,\, dt \times d\bP \,\, a.e. \,\,
   \end{align}
   for some solution $(X, Y, Z,Z^0)$ of \eqref{eq:mkvfbsde}.
   Suppose in addition that $x \mapsto G(x,m)$ is convex for each $m$ and
   \begin{align*}
   (x,a) \mapsto b(x,a,m) \cdot y + \sigma(x,a,m) \cdot z + \sigma^0(x,a,m) \cdot z^0 + L(x,a,m)
   \end{align*}
   is convex for each fixed $(y,z,z^0,m)$, and
   $\hat{\alpha}$ has linear growth in $(x,y,z,z^0,m)$. Then if $(X, Y, Z,Z^0)$ is any solution of \eqref{eq:mkvfbsde}, then $(m,\alpha)$ with
\begin{align*}
m_t = \sL(X_t | \sF_t^0), \quad \alpha_t = \widehat \alpha(X_t,Y_t,Z_t,Z_t^0,m_t)
\end{align*}
 is a MFE.
\end{lemma}
\begin{proof}
  This characterization result can be derived by standard arguments, see e.g. \cite[Theorem 3.27]{cardelbook1} for the case without common noise.
  In fact, it suffices to apply the (standard) stochastic maximum principle for stochastic control problems (without games). Indeed, let $(m,\alpha)$ be a MFE.
  Then, applying the usual stochastic maximum principle for the stochastic control problem with (random) costs $f_t(x,a) := L(x, a,m_t)$ and $g(x) := G(x,m_T)$, it follows that $\alpha$ satisfies \eqref{eq:char.mfe}.
  Reciprocally, it follows from \cite[Theorem 2.16]{cardelbook1} again with $N=1$ and with the costs $f_t(x,a) := L(x, a, \cL(X_t\mid \cF^0_t))$ and $g(x) := G(x, \cL(X_T\mid \cF^0_T))$ 
  that $\alpha$ given by  \eqref{eq:char.mfe} must be optimal, which implies, by definition, that $(m,\alpha)$ is a MFE, where $m$ is a continuous version of $(\cL( X_t\mid \cF^0_t))_{t\in [0,T]}$.
\end{proof}

\begin{remark}
In the case that Assumptions \ref{assump:driftcontrolreg} and \ref{assump:driftcontrolstruct} hold (i.e. the case of constant volatility) Lemma \ref{lem:mpmfggen} shows that any MFE $(m,\alpha)$ must be of the form 
\begin{align*}
    m_t = \sL(X_t | \sF_t^0), \quad \alpha_t = D_y H(X_t, Y_t), 
\end{align*}
for some tuple $(X,Y,Z,Z^0)$ satisfying
\begin{align} \label{eq:mkvfbsdedrift}
  \begin{cases} \vspace{.1cm}
    dX_t = \alpha_t dt + \Sigma dW_t + \Sigma^0 dW_t^0,  \\ \vspace{.1cm}
  dY_t = - \Big(D_x L_0(X_t, \alpha_t) + D_x F\big(X_t, \cL(X_t\mid \cF^0_t)\big) \Big)dt + Z_t  dW_t + Z_t^{0}  dW_t^0, \\
  X_0 = \xi, \quad Y_T = D_x G\big(X_T, \sL(X_T | \cF_T^0)\big).%, \vspace{.1cm} \\
  %m_t = . 
\end{cases}
\end{align}
\end{remark}

\begin{remark}
\label{rmk:fbsdeaffinemfe}
Under Assumptions \ref{assump:affinereg} and \ref{affinestruct} (i.e. in the case of affine dynamics), Lemma \ref{lem:mpmfggen} shows that any $(m,\alpha)$ must take the form
\begin{align*}
    m_t = \sL(X_t | \sF_t^0), \quad \alpha_t = \widehat{\alpha}(X_t,Y_t,Z_t,Z_t^0, \sL(X_t | \sF_t^0)), 
\end{align*}
for some tuple $(X,Y,Z,Z^0)$ satisfying
\begin{align} \label{eq:nplayerfbsdeaffinemfe}
  \begin{cases} \vspace{.1cm}
    dX_t = \Big(B_0 + B_1 \alpha_t + B_2 X_t \Big) dt + \Big( \Sigma_0 + \Sigma_1 \alpha_t + \Sigma_2 X_t\Big) dW_t^i + \Big(\Sigma^0_0 + \Sigma^0_1 \alpha_t + \Sigma^0_2 X_t \Big) dW_t^0   \\ \vspace{.1cm}
dY_t = - \Big( D_x L(X_t,\alpha_t, m_t) + B_2^\top Y_t + \Sigma_2^\top Z_t + (\Sigma_2^0)^\top Z_t^0 \Big) dt
+ Z_t  dW_t + Z_t^0  dW_t^0 \vspace{.1cm}  \\
X_0 = \xi, \quad Y_T = D_x G(X_T,m_T).
\end{cases} 
\end{align}
\end{remark}

\subsection{Technical Lemmas} 
\label{subsec:technical}

\subsubsection{Projecting the monotonicity conditions}

Here we record three Lemmas concerned with consequences of the displacement monotonicity conditions given in Assumptions \ref{assump:driftcontrolstruct}, \ref{assump:affinestruct}, and \ref{assump:controlledvolstruct}. The proofs are trivial, but we state them as Lemmas because they will be used repeatedly.

\begin{lemma} \label{lem:dispmonotoneprojected}
Suppose that $U = U(x,m) : \R^n \times \spt \to \R^n$ is $C$-displacement monotone. Then we have 
\begin{align*}
\sum_{i = 1}^N \Big(D_x U(x^i,m_{\bx}^N) - D_xU(\ov{x}^i,m_{\ov{\bx}}^N) \Big) \cdot (x^i - \ov{x}^i ) \geq C \sum_{i = 1}^N |x^i - \ov{x}^i|^2
\end{align*}
for any $\bx = (x^1,\dots,x^N),\,\, \ov{\bx} = (\ov{x}^1,\dots,\ov{x}^N) \in (\R^n)^N$.
\end{lemma}

\begin{proof}
    By definition, we have 
    \begin{align*}
    \E\Big[\big(D_x U(X , \sL(X)) - D_xU(\ov{X}, \sL(\ov{X})\big) \cdot (X - \ov{X}) \Big] \geq C \E[|X - \ov{X}|^2]
    \end{align*}
    for any square-integrable random variables $X$ and $\ov{X}$ taking values in $\R^n$. 
    To prove the claim for any given $\bx = (x^1,\dots,x^N),\,\, \ov{\bx} = (\ov{x}^1,\dots,\ov{x}^N) \in (\R^n)^N$, test this inequality on discrete random variables $X$ and $\ov{X}$ whose joint distribution is given by 
    \begin{align*}
    \bP\big[ X = x^i, \ov{X} = \ov{x}^i \big] = 1/N, \quad i = 1,\dots,N.
    \end{align*}
\end{proof}

\begin{lemma} \label{lem:dispmonotoneprojectedaffine}
Suppose that Assumption \ref{assump:affinestruct} holds. Then we have 
\begin{align*}
    &\sum_{i = 1}^N \bigg( \Big(D_xL(x^i, a^i, m_{\bx}^N) - D_xL(\ov{x}^i, \ov{a}^i, m_{\ov{\bx}}^N) \Big) \cdot (x^i - \ov{x}^i) \\ 
    &\qquad + \Big(D_aL(x^i, a^i, m_{\bx}^N) - D_aL(\ov{x}^i, \ov{a}^i, m_{\ov{\bx}}^N) \Big) \cdot (a^i - \ov{a}^i) \bigg) \geq C_L \sum_{i = 1}^N |a^i - \ov{a}^i|^2 
\end{align*}
for any $\bx = (x^1,\dots,x^N),  \ov{\bx} = (\ov{x}^1,\dots,\ov{x}^N) \in (\R^n)^N$, $\bm a = (a^1,\dots,a^N)$, $\bm{\ov{a}} = (\ov{a}^1,\dots,\ov{a}^N) \in (\R^k)^N$. 
\end{lemma}

\begin{proof}
Test the inequality \eqref{affinestruct} on random variables $X, \ov{X}, \alpha, \ov{\alpha}$ with joint distribution 
\begin{align*}
\bP \big[ X = x^i, \ov{X}= y^i, \alpha = a^i, \ov{\alpha} = \ov{a}^i \big] = 1/N, \quad i = 1,\dots,N.
\end{align*}
\end{proof}

\begin{lemma} \label{lem:finitedimmonotone}
    Let Assumption \ref{assump:controlledvolstruct} hold. Then
    we have 
    \begin{align*}
        \sum_{i = 1}^N \bigg( &- \big(D_x H(x^i,y^i,z^i,z^{0,i},m_{\bx}^N) - D_x H(\ov{x}^i,\ov{y}^i,\ov{z}^i,\ov{z}^{0,i},m_{\ov{\bx}}^N) \big) \cdot (x^i - \ov{x}^i) 
        \\
        &+  \big(D_y H(x^i,y^i,z^i,z^{0,i},m_{\bx}^N) - D_y H(\ov{x}^i,\ov{y}^i,\ov{z}^i,\ov{z}^{0,i}, m_{\ov{\bx}}^N) \big)\cdot (y^i - \ov{y}^i) 
        \\
         &+  \big(D_z H(x^i,y^i,z^i,z^{0,i},m_{\bx}^N) - D_z H(\ov{x}^i,\ov{y}^i,\ov{z}^i,\ov{z}^{0,i},m_{\ov{\bx}}^N) \big) \cdot (z^i - \ov{z}^i)
         \\
         &+  \big(D_{z^0} H(x^i,y^i,z^i,z^{0,i},m_{\bx}^N) - D_{z^0} H(\ov{x}^i,\ov{y}^i,\ov{z}^i,\ov{z}^{0,i},m_{\ov{\bx}}^N) \big)  \cdot (z^{0,i} - \ov{z}^{0,i})  \bigg) \\
         &\qquad \leq - C_H \sum_{i = 1}^N |x^i - \ov{x}^i|^2
    \end{align*}
    for any 
    \begin{align*}
        &\bx = (x^1,\dots,x^N), \,\, \ov{\bx} = (\ov{x}^1,\dots,\ov{x}^N), \,\, \by = (y^1,\dots,y^N), \,\, \ov{\by} = (\ov{y}^1,\dots,\ov{y}^N) \in (\R^n)^N, \\
        &\bz = (z^1,\dots,z^N), \,\, \ov{\bz} = (\ov{z}^1,\dots,\ov{z}^N), \,\, \bz^0 = (z^{0,1},\dots,z^{0,N}), \,\, \ov{\bz}^0 = (\ov{z}^{0,1},\dots,\ov{z}^{0,N}) \in (\R^{n \times d})^N, 
    \end{align*}
\end{lemma}

\begin{proof}
    Test the condition \eqref{hamiltonianmonotone} on discrete random variables $X,\ov{X}, Y, \ov{Y}, Z, \ov{Z} ,Z^0, \ov{Z}^0$ with joint distribution given by
    \begin{align*}
        \bP[X = x^i, \ov{X} = \ov{x}^i, Y = y^i, \ov{Y} = \ov{y}^i, Z = z^i, \ov{Z} = \ov{z}^i, Z^0 = z^{0,i}, \ov{Z}^0 = \ov{z}^{0,i}] = \frac{1}{N}, \quad i = 1,\dots,N.
    \end{align*}
\end{proof}

\subsubsection{Uniqueness and exchangeabity properties of the $N$-player games}

We will now establish some properties of the finite-player games which will play an important role in the proofs of the convergence results.

\begin{lemma} \label{lem:uniqueness}
Suppose that Assumption \ref{assump:mixed} holds.
Then there is a number $N_0 \in \N$ such that for any $N \geq N_0$, there is at most one Nash equilibrium for the N-player game.
\end{lemma}

\begin{proof}
%  Uniqueness for the mean field game follows from (the proof technique of) Theorem 3.2 of \cite{ahuja2016}.
%  We omit the proof here.
 By Lemma \ref{lem:N.player.smp}, it suffices to show that for $N$ large enough, the FBSDE \eqref{eq:mkvfbsde} has a unique solution.
 \newline \newline \noindent
  \textbf{Case 1 (Assumptions \ref{assump:controlledvolreg} and \ref{assump:controlledvolstruct} hold):} In this case, we can simply note that Lipschitz continuity of $D_m H$ together with the structural condition \eqref{assump:controlledvolstruct} guarantee that for $N$ large enough, the driver of the FBSDE \eqref{eq:nplayerfbsde} satisfies condition (H2.2) in \cite{pengwumonotone} with $\beta_1$ therein such that $\beta_1 > 0$. 
  Assumption \ref{assump:controlledvolreg}, meanwhile, easily gives the Lipschitz continuity conditions appearing in (H2.1) of \cite{pengwumonotone}. Thus for $N$ large enough, the coefficients of \eqref{eq:nplayerfbsde} satisfy the hypotheses of \cite[Theorem 2.2]{pengwumonotone}, which shows that there is at most one Nash equilibrium for $N$ large enough.
  \newline \newline \noindent
   \textbf{Case 2 (Assumptions \ref{assump:driftcontrolreg} and \ref{assump:driftcontrolstruct} hold):} In this case, we cannot directly rely on results from the finite-dimensional FBSDE literature, so we derive the result here. We proceed in three steps.
   \newline  \noindent  
   \textit{Step 1 - Applying the maximum principle:}
  Suppose that we have two equlibria $\bm \alpha^N = (\alpha^{N,1},\dots,\alpha^{N,N})$ and $\bm{\wt{\alpha}} = (\wt{\alpha}^{N,1},\dots,\wt{\alpha}^{N,N})$. 
  By strict convexity of $L_0$, it follows that $D_yH_0(x,y)$ is the unique minimizer of $a\mapsto L_0(x,a) + a\cdot y$.
  Thus, applying Lemma \ref{lem:N.player.smp} yields that $\bm \alpha^N$ and $\bm{\wt\alpha}^N$ both take the form
  \begin{equation*}
    \alpha^{N,i}_t = D_yH_0(X^{N,i}_t, Y^{N,i}_t)\quad \text{and}\quad \wt\alpha^{N,i}_t = D_yH_0(\wt X^{N,i}_t, \wt Y^{N,i}_t)\,\, dt\times d\P\text{--a.s.}
  \end{equation*}
  where the processes 
\begin{align*}
(\bX^N, \bY^N, \bZ^N) = \big( (X^{N,1},\dots,X^{N,N}), (Y^{N,1},\dots,Y^{N,N}), (\bZ^{N,1},\dots,\bZ^{N,N})\big), \\
(\bm{\wt{X}}^N, \bm{\wt{Y}}^N, \bm{\wt{Z}}^N) = \big( (\wt{X}^{N,1},\dots,\wt{X}^{N,N}), (\wt{Y}^{N,1},\dots,\wt{Y}^{N,N}), (\bm{\wt{Z}}^{N,1},\dots,\bm{\wt{Z}}^{N,N})\big)
\end{align*} 
both satisfy \eqref{eq:nplayerfbsdedrift}.
We set
\begin{align*}
  \Delta X_t^i := X_t^{N,i} - \wt{X}_t^{N,i}, \quad \Delta Y_t^i := Y_t^{N,i} - \wt{Y}_t^{N,i}, \quad \Delta \alpha_t^i := \alpha_t^{N,i} - \wt{\alpha}_t^{N,i}, 
\end{align*}
and as usual, $m_{\bX_t^N}^N = \frac{1}{N} \sum_{i = 1}^N \delta_{X_t^{N,i}}$ and $m_{\wt \bX_t^N}^N = \frac{1}{N} \sum_{i = 1}^N \delta_{\wt X_t^{N,i}}$.

\noindent 
\textit{Step 2 - Computing $d \Delta X_t^i \cdot \Delta Y_t^i$:}
Notice that since $\alpha_t^{N,i}$ minimizes $a \mapsto  L_0(X_t^{N,i}, a) + a \cdot Y_t^{N,i}$, we get 
\begin{align*}
Y_t^{N,i} = - D_a L_0(X_t^{N,i}, \alpha_t^{N,i}), \text{ and similarly }
\wt{Y}_t^i = - D_a L_0(\wt{X}_t^{N,i}, \wt{\alpha}_t^{N,i}).  
\end{align*}
Using these identities, a tedious but straightforward computation shows that 
\begin{align} \label{xydynamicsunqiueness}
\notag
 & d \Delta X_t^i \cdot \Delta Y_t^i \\\notag
 &= - \bigg(\Delta \alpha^i_t \cdot  \big(D_a L_0(X_t^{N,i}, \alpha_t^{N,i}) - D_a L_0(\wt{X}_t^{N,i}, \wt{\alpha}_t^i) \big) + \Delta X_t^i \cdot \big(D_x L_0(X_t^{N,i}, \alpha_t^{N,i}) - D_x L_0(\wt{X}_t^{N,i}, \wt{\alpha}_t^{N,i}) \big) \nonumber \\
  &+ \Delta X_t^i \cdot \big(D_x F(X_t^{N,i}, m_{\bX_t^N}^N) - D_x F(\wt{X}_t^{N,i}, m_{\wt{\bX}_t}^N)\big) \nonumber \\
  &+ \frac{1}{N} \Delta X_t^i \cdot \big(D_m F(X_t^{N,i}, m_{\bX_t^N}^N, X_t^{N,i}) - D_m F(\wt{X}_t^{N,i}, m_{\bm{\wt{X}}_t^N}^N, \wt{X}_t^{N,i}) \big) \bigg) dt + dM_t^i,
\end{align}
with $M^i$ a martingale.
The terminal condition, meanwhile, can be written as
\begin{align*}
  \Delta X_T^i \cdot \Delta Y_T^i &= \Delta X_T^{i} \cdot \big(D_x G(X_T^{N,i}, m_{\bX_T^N}^N) - D_x G(\wt X_T^{N,i}, m_{\wt\bX_T}^N) \big) \\
  &\qquad + \frac{1}{N} \Delta X_T^i \cdot  \big(D_m G(X_T^{N,i}, m_{\bX_T^N}^N, X_T^{N,i}) - D_m G(\wt X_T^{N,i}, m_{\wt{\bX}_T^N}^N, \wt{X}_T^{N,i}) \big).
\end{align*}
\noindent
\textit{Step 3 - Using the monotonicity condition:}
Now we integrate \eqref{xydynamicsunqiueness}, sum over $i$, and take expectations to get\footnote{Here and throughout the paper, we use the shorthand notation $\sum_{i}$ for $\sum_{i=1}^N$.}
\begin{align*}
  -\E\bigg[& \int_0^T \sum_i \bigg(\Delta \alpha_t^i \cdot \big(D_a L_0(X_t^{N,i}, \alpha_t^{N,i}) - D_a L_0(\wt{X}_t^{N,i}, \wt{\alpha}_t^{N,i}) \big)\\
  & + \Delta X_t^i \cdot \big(D_x L_0(X_t^{N,i}, \alpha_t^{N,i}) - D_x L_0(\wt{X}_t^{N,i}, \wt{\alpha}_t^{N,i}) \big) + \Delta X_t^i \cdot \big(D_x F(X_t^{N,i}, m_{\bX_t^N}^N) - D_x F(\wt{X}_t^{N,i}, m_{\wt{\bX}^N_t}^N)\big) \\ 
    & + \frac{1}{N}\Delta X_t^{N,i} \cdot \big(D_m F(X_t^{N,i}, m_{\bX_t^N}^N, X_t^{N,i}) - D_m F_t(\wt{X}_t^{N,i}, m_{\wt\bX_t^N}^N, \wt{X}_t^{N,i}) \big) \bigg) dt \bigg] \\
  &= \E\bigg[ \sum_i \Delta X_T^i \cdot \Delta Y_T^i\bigg] 
  = \E\bigg[\sum_i \Delta X_T^i \cdot \big(D_x G(X_T^{N,i}, m_{\bX_T^N}^N) - D_x G(\wt{X}_T^{N,i}, m_{\wt\bX_T}^N) \big) \\
& \qquad + \frac{1}{N} \sum_i \Delta X_T^i \cdot  \big(D_m G(X_T^{N,i}, m_{\bX_T^N}^N, X_T^{N,i}) - D_m G(\wt{X}_T^{N,i}, m_{\wt{\bX}_T^N}^N, \wt{X}_T^{N,i}) \big)\bigg]. 
\end{align*}
We rearrange and use the definitions of $C_L$, $C_{F}$, $C_{G}$ together with Lemma \ref{lem:dispmonotoneprojected} to get
\begin{align} \label{smallnessunique}
  C_L &\E \bigg[\int_0^T \sum_i |\Delta \alpha_t^i|^2 dt \bigg] \nonumber \\
    &\leq \E \bigg[ \int_0^T \Big( -C_{F} \sum_i |\Delta X_t^i|^2 + \frac{1}{N} \sum_i |D_m F(X_t^{N,i}, m_{\bX_t^N}^N, X_t^{N,i}) - D_m F(\wt{X}_t^{N,i}, m_{\bm{\wt{X}}_t^N}^N, \wt{X}_t^{N,i})||\Delta X_t^i|\Big) dt \nonumber \\
    & \quad \quad - C_{G} \sum_i |\Delta X_T^i|^2 + \frac{1}{N} \sum_i |D_m G(X_T^{N,i}, m_{\bX_T^N}^N, X_T^{N,i}) - D_m G(\wt{X}_T^{N,i}, m_{\bm{\bX}_T^N}^N, \wt{X}_T^{N,i})| |\Delta X_T^i| \bigg] \nonumber \\
  &\leq -\Big((C_{G} \wedge 0)T + \frac{(C_{F} \wedge 0) T^2}{2} \Big) \E\bigg[ \int_0^T\sum_i |\Delta \alpha_t^i|^2 dt\bigg] + \frac{C}{N} \E \bigg[ \int_0^T \sum_i |\Delta X_t^i|^2  dt + \sum_i |\Delta X_T^i|^2 \bigg] \nonumber \\
  &\leq -\Big((C_{G} 
  \wedge 0)T + \frac{(C_{F} \wedge 0) T^2}{2} - \frac{C}{N} \Big) \E\bigg[ \int_0^T\sum_i |\Delta \alpha_t^i|^2 dt\bigg]
\end{align}
where we have used repeatedly the fact that $\Delta X_t^i = \int_0^t \Delta \alpha_s^i ds$, and so 
\begin{align}
  \E[|\Delta X_t^i|^2] = \E\bigg[\Big|\int_0^t  \Delta \alpha_s^i ds\Big|^2\bigg] \leq t \E\bigg[\int_0^t |\Delta \alpha_s^i|^2 ds\bigg], 
\end{align}
as well as the fact that global Lipschitz continuity of $D_m F$ implies
\begin{align*}
  |D_m F(x^i, m_{\bx}^N,x^i) - D_m F(\wt{x}^i, m_{\wt{\bx}}^N, \wt{x}^i)|^2 \leq C |x^i - \wt{x}^i|^2 + \frac{1}{N} |\bx - \wt{\bx}|^2
\end{align*}
and likewise for $D_m G$.
From \eqref{smallnessunique}, we conclude that for some $C$ independent of $N$ we have 
\begin{align*}
  \Big(C_L + (C_{G} \wedge 0) T + \frac{(C_{F} \wedge 0) T^2}{2} - \frac{C}{N}\Big) \E \bigg[ \int_0^T \sum_i |\Delta \alpha_t^i|^2 dt \bigg] \leq 0
\end{align*}
for each $N$. 
Thus for $N$ sufficiently large we have $\Delta \alpha_t^{N,i} = 0$ $dt\times d\P$--a.s. for all $i = 1,\dots, N$.
From here, it is easy to conclude that in fact $\Delta X_t^{N,i} = 0$ $\P$-a.s. for all $t$ and for all $i=1,\dots,N$.
Hence, $(\bX^N, \bY^N, \bZ^N, \bZ^{N,0}) = (\wt{\bX}^N, \wt{\bY}^N, \wt{\bZ}^N, \wt{\bZ}^{N,0})$. 
\newline \newline \noindent 
\textbf{Case 3 (Assumption \ref{assump:affinereg} and \ref{affinestruct} hold):} The proof is very similar to case 2, and so we follow the same three steps but keep the argument brief. 
\newline \noindent 
\textit{Step 1 - Applying the maximum principle:} This time, Lemma \ref{lem:N.player.smp} shows that $\bm \alpha^N$ and $\wt{\bm \alpha}^N$ take the form 
\begin{align*}
    \alpha_t^{N,i} = \widehat{\alpha}(X_t^{N,i}, Y_t^{N,i}, Z_t^{N,i,i}, Z_t^{N,i,0},m^N_{\bm X^N_t}), \quad \wt{\alpha}_t^{N,i} = \widehat{\alpha}(\wt{X}_t^{N,i}, \wt{Y}_t^{N,i}, \wt{Z}_t^{N,i,i}, \wt{Z}_t^{N,i,0},m^N_{\wt{\bm X}^N_t}) ,
\end{align*}
for some tuples 
\begin{align*}
(\bX^N, \bY^N, \bZ^N) = \big( (X^{N,1},\dots,X^{N,N}), (Y^{N,1},\dots,Y^{N,N}), (\bZ^{N,1},\dots,\bZ^{N,N})\big), \\
(\bm{\wt{X}}^N, \bm{\wt{Y}}^N, \bm{\wt{Z}}^N) = \big( (\wt{X}^{N,1},\dots,\wt{X}^{N,N}), (\wt{Y}^{N,1},\dots,\wt{Y}^{N,N}), (\bm{\wt{Z}}^{N,1},\dots,\bm{\wt{Z}}^{N,N})\big)
\end{align*} 
both satisfying \eqref{eq:nplayerfbsdeaffine}. 
\newline 
\noindent \textit{Step 2 - Computing $d \Delta X_t^i \cdot \Delta Y_t^i$:} We now use the same notation as above and once again compute $d \Delta X_t^i \cdot \Delta Y_t^i$. 
The fact that $\alpha_t^{N,i}$ maximizes 
\begin{align*}
    a \mapsto Y_t^{N,i} \cdot \big( B_1 a \big) + Z_t^{N,i,i} \cdot \big(\Sigma_1 a \big) + Z_t^{N,i,0} \cdot \big( \Sigma_1^0 a \big) + L(X_t^{N,i}, m_{\bX_t^N}^N, a)
\end{align*}
leads to 
\begin{align} \label{dacomp}
    D_a L(X_t^{N,i}, \alpha_t^{N,i}, m_{\bX_t^N}^N) = - \Big( B_1^\top Y_t^{N,i} + \Sigma_1^\top Z_t^{N,i,i} + (\Sigma_1^0)^\top Z_t^{N,i,0} \Big) 
\end{align}
and likewise the same argument gives
\begin{align} \label{dacomp2}
    D_a L(\ov{X}_t^{N,i}, \ov{\alpha}_t^{N,i}, m_{\ov{\bX}_t^N}^N) = - \Big( B_1^\top \ov{Y}_t^{N,i} + \Sigma_1^\top \ov{Z}_t^{N,i,i} + (\Sigma_1^0)^\top \ov{Z}_t^{N,i,0} \Big).
\end{align}
This time we use \eqref{dacomp} and \eqref{dacomp2} to find 
\begin{align} \label{xydynamicsunqiueness2}
  \nonumber d \Delta X_t^i \cdot \Delta Y_t^{i} = - &\bigg(\Delta \alpha^i_t \cdot  \big(D_a L(X_t^{N,i}, \alpha_t^{N,i}, m_{\bX_t^N}^N) - D_a L(\wt{X}_t^{N,i}, \wt{\alpha}_t^{N,i}, m_{\wt{\bX}_t^N}^N) \big) \\
  &+ \Delta X_t^i \cdot \big(D_x L(X_t^{N,i}, \alpha_t^{N,i}, m_{\bX_t^N}^N) - D_x L(\wt{X}_t^{N,i}, \wt{\alpha}_t^{N,i}, m_{\bX_t^N}^N) \big) \nonumber \\
  &+ \frac{1}{N} \Delta X_t^i \cdot \big(D_m F(X_t^{N,i}, m_{\bX_t^N}^N, X_t^{N,i}) - D_m F(\wt{X}_t^{N,i}, m_{\bm{\wt{X}}_t^N}^N, \wt{X}_t^{N,i}) \big) \bigg) dt + dM_t^i.
\end{align}
\textit{Step 3 - Using the monotonicity condition:}
Summing over $i$, applying Lemma \ref{lem:dispmonotoneprojectedaffine}, and taking expectations leads to 
\begin{align} \label{smallnessunique2}
  C_L &\E \bigg[\int_0^T \sum_i |\Delta \alpha_t^i|^2 dt \bigg] \nonumber \\
    &\leq \E \bigg[ \int_0^T \Big( \frac{1}{N} \sum_i |D_m F(X_t^{N,i}, m_{\bX_t^N}^N, X_t^{N,i}) - D_m F(\wt{X}_t^{N,i}, m_{\bm{\wt{X}}_t^N}^N, \wt{X}_t^{N,i})||\Delta X_t^i|\Big) dt \nonumber \\
    & \quad \quad  + \frac{1}{N} \sum_i |D_m G(X_T^{N,i}, m_{\bX_T^N}^N, X_T^{N,i}) - D_m G(\wt{X}_T^{N,i}, m_{\bm{\bX}_T^N}^N, \wt{X}_T^{N,i})| |\Delta X_T^i| \bigg] \nonumber \\
  &\leq  \frac{C}{N} \E \bigg[ \int_0^T \sum_i |\Delta X_t^i|^2  dt + \sum_i |\Delta X_T^i|^2 \bigg] \nonumber \\
  &\leq  \frac{C}{N} \E\bigg[ \int_0^T\sum_i |\Delta \alpha_t^i|^2 dt\bigg]
\end{align}
where we used the fact that $G$ is displacement monotone and the fact that 
\begin{align*}
    \E\Big[\sup_{0 \leq t \leq T} |\Delta X_t^{N,i}|^2 \Big] \leq C \E\bigg[\int_0^T |\Delta \alpha_t^{N,i}|^2 dt\bigg]. 
\end{align*}
Taking $N$ large enough completes the proof.
\end{proof}

\begin{lemma} \label{lem:exchangeable}
  Let $N \in \N$, and suppose that there is exactly one Nash equilibrium $\bm{\alpha}^N = (\alpha^{N,1},\dots,\alpha^{N,N})$ for the $N$-player game. 
  Let $\bX^N = (X^{N,1},\dots,X^{N,N})$ denote the state process corresponding to $\bm \alpha^N$, let $(\alpha^i)_{i = 1,\dots,N}$ denote independent copies of some MFE $\alpha$, and $(X^i)_{i = 1,\dots,N}$ denote independent copies of the state process $X$ corresponding to $\alpha$.
  Then, we have 
\begin{align} \label{alphasymmetry}
    \E\bigg[\int_0^T |\alpha_t^{N,i} - \alpha_t^i|^2 dt \bigg] =  \E\bigg[\int_0^T |\alpha_t^{N,j} - \alpha_t^j|^2 dt \bigg], \quad \forall i,j = 1,\dots,N \end{align}
and 
\begin{align} \label{xsymmetry}
    \E\Big[\sup_{0 \leq t \leq T} |X_t^{N,i} - X_t^i|^2  \Big] =  \E\Big[\sup_{0 \leq t \leq T} |X_t^{N,j} - X_t^j|^2 \Big], \quad \forall i,j = 1,\dots,N.
\end{align}
\end{lemma}

\begin{proof}
Using (the reasoning leading to) Lemma \ref{lem:coupling}, it is straightforward to find a version of $\bm \alpha^N$ such that for each $t \in [0,T]$ and $i = 1,\dots,N$, we have 
\begin{align*}
    \alpha_t^{N,i} = \widehat{\alpha}^{N,i}_t\big(\boldsymbol{\xi}, W^0_{[0,t]},\mathbf{W}_{[0,t]}\big) % = \widehat{\alpha}^{N,i}_t\big(\xi^i,\xi^1,\dots,\xi^{i-1},\xi^{i+1},\dots,\xi^N,W^i_{[0,t]}, W^0_{[0,t]}, W^1_{[0,t]}, \dots, W^{i-1}_{[0,t]}, W^{i+1}_{[0,t]},\dots,W^N_{[0,t]}\big),
\end{align*}
for some collection of measurable maps $\widehat{\alpha}^{N,i}_t :  (\R^n)^{N} \times \cC_t \times \cC_t^{N} \to \R^k$. Now for any permutation $\pi : \{1,\dots,N\} \to \{1,\dots.,N\}$, we can define a process $\bm \alpha^{\pi,N} = (\alpha^{\pi, N,1},\dots,\alpha^{\pi,N,N})$ by 
\begin{align*}
    \alpha^{\pi,N,i}_t = \widehat{\alpha}^{N,\pi(i)}_t\big(\boldsymbol{\xi}, W^0_{[0,t]},\mathbf{W}_{[0,t]}\big).%  \big(\xi^i,\xi^1,\dots,\xi^{i-1},\xi^{i+1},\dots,\xi^N,W^i_{[0,t]}, W^0_{[0,t]}, W^1_{[0,t]}, \dots, W^{i-1}_{[0,t]}, W^{i+1}_{[0,t]},\dots,W^N_{[0,t]}\big).
\end{align*} 
It is easy to see that 
\begin{align*}
    J^{N,i}\big( ( \bm \alpha^{\pi, N})^{-i}, \alpha) = J^{N, \pi(i)} \big( \bm \alpha^{- \pi(i)}, \alpha), 
\end{align*}
and so $\bm \alpha^{\pi,N}$ is a Nash equlibrium.
In particular, by uniqueness we have $\alpha_t^{N,i} = \alpha_t^{N,j}$ a.e. for each $t$, $i$, and $j$. Thus we can find a version of $\bm \alpha$ and a collection of maps $\wt{\alpha}_t^N$ such that 
\begin{align*}
    \alpha^{\pi,N,i}_t = \widehat{\alpha}^{N}_t \big(\xi^i,\xi^1,\dots,\xi^{i-1},\xi^{i+1},\dots,\xi^N,W^i_{[0,t]}, W^0_{[0,t]}, W^1_{[0,t]}, \dots, W^{i-1}_{[0,t]}, W^{i+1}_{[0,t]},\dots,W^N_{[0,t]}\big).
\end{align*}
In addition, we already know that the independent copies $(\alpha^i)_{i = 1,\dots,N}$ satisfy
\begin{align*}
    \alpha^i_t = \widehat{\alpha}_t\big(\xi^i, W^i_{[0,t]}, W^0_{[0,t]}\big)
\end{align*}
for some collection of measurable maps $\widehat{\alpha}_t$. Thus,
\begin{equation*}
    \E\bigg[\int_0^T |\alpha_t^{N,i} - \alpha_t^i|^2 dt \bigg]
\end{equation*}
%   we can write 
% \begin{align*}
%      \E\bigg[&\int_0^T |\alpha_t^{N,i} - \alpha_t^i|^2 dt \bigg] \\
%      &= \E\bigg[ \int_0^T\big|\widehat{\alpha}^{N}_t \big(\xi^i,\xi^1,\dots,\xi^{i-1},\xi^{i+1},\dots,\xi^N,W^i_{[0,t]}, W^0_{[0,t]}, W^1_{[0,t]}, \dots, W^{i-1}_{[0,t]}, W^{i+1}_{[0,t]},\dots,W^N_{[0,t]}\big) \\
%      &\quad- \widehat{\alpha}_t\big(\xi^i, W^i_{[0,t]}, W^0_{[0,t]}\big) \big|^2 dt \bigg], 
% \end{align*}
%which is clearly independent of $i$.
is clearly independent of $i$.
 This establishes \eqref{alphasymmetry}. The proof of \eqref{xsymmetry} is very similar and so is omitted. 
\end{proof}

\subsubsection{Weak error estimates with common noise}

\begin{lemma} \label{lem:errorterms}
  Let $E$ be a Euclidean space, and let $\cU = \cU(x,m) : E \times \sP_2(\R^n)\to \R$ be a function of linear growth such that for each fixed $x$, the map $m\mapsto \cU(x,m)$ admits the derivatives
  \begin{align*}
  D_m \cU(x,m,p) = D_m\big[\cU(x,\cdot)](m,p), \quad D_{mm} D_x \cU(x,m,p,q) = D_{mm} \big[D_x \cU(x,\cdot)\big](m,p,q), 
\end{align*}
which are bounded and Lipschitz continuous, uniformly in $x$.
Then, for any $N$ i.i.d. random variables $(\xi^i)_{i = 1,\dots,N}$  with common law $m$,  we have
\begin{align*}
  \E\big[|\cU(x,m) - \cU(x,m_{\bm \xi}^N)|^2\big] \leq \frac{C}{N}(1 + |x|^2 + M_2(m)), 
\end{align*}
for each $x\in E$, where we set $m_{\bm \xi}^N := \frac{1}{N} \sum_{i = 1}^N \delta_{\xi^i}$ and $C$ is a constant independent of $N$ and $m$.
\end{lemma}

\begin{proof}
First, we expand the square
\begin{align*}
\E[|& \cU(x,m) - \cU(x,m_{\bm \xi}^N)|^2] = \E[ |\cU(x,m)|^2 + |\cU(x,m_{\bm \xi}^N)|^2 - 2 \cU(x,m) \cU(x,m_{\bm \xi}^N) ] \\
&=  \E[|\cU(x, m_{\bm \xi}^N)|^2] - | \cU(x,m)|^2 + 2\cU(x,m) \big(\cU(x,m) - \E[\cU(x,m_{\bm \xi}^N)] \big) \eqqcolon e_1 + e_2.
\end{align*}
The goal is now to estimate $e_1$ and $e_2$ using \cite[Theorem 2.11]{CST}. 
We start by noticing that under Assumption \ref{assump:cst}, $m \mapsto |\cU(x,m)|^2$ is twice differentiable in the measure argument, with derivatives given explicitly by
\begin{align} \label{derivsquare}
&D_m \big[ |\cU(x,\cdot)|^2\big] (m,p) = 2 U(x,m) D_m \cU(x,m,p), \nonumber \\
&D_{mm} \big[ |\cU(x,\cdot)|^2 \big] (m,p,q) = 2 D_m \cU(x,m,p) \otimes D_m \cU(x,m,q) + \cU(x,m) D_{mm} \cU(x, m,p,q). 
\end{align}
We note that the following bounds follow easily:
\begin{align} \label{dmbounds}
D_m \big[ |\cU(x,\cdot)|^2\big] (m,p) &\leq C\big(1 + |x| + M_2^{1/2}(m)\big) \nonumber \\
D_{mm} \big[ | \cU(x,\cdot)|^2 \big] (m,p,q) &\leq C\big(1 + |x| + M_2^{1/2}(m)\big).
\end{align}
 Suppose now that $|\cU|^2$ admits linear functional derivatives
\begin{align*}
\frac{\delta}{\delta m}[|\cU(x, \cdot)|^2](m,p), \quad \frac{\delta^2}{\delta m^2}[| \cU(x, \cdot)|^2](m,p,q)
\end{align*}
as defined in \cite{CST} which satisfy the fundamental relationship
\begin{align} \label{linderivrelation}
D_p \frac{\delta}{\delta m}[|\cU(x, \cdot)|^2](m,p) &= D_m |\cU|^2(x,m,p), \nonumber \\
D_p D_q \frac{\delta^2}{\delta m^2}[|\cU|^2(x, \cdot)](m,p,q) &= D_{mm} |\cU|^2(x,m,p,q).
\end{align}
Combining \eqref{dmbounds} and \eqref{linderivrelation}, we see that
\begin{align*}
\Big|\frac{\delta}{\delta m^2}[|\cU|^2(x, \cdot)](m,p,q) \Big| &\leq C\big(1 + |x| + M^{1/2}_2(m)\big)\big( |p| +|q| \big) \\
&\leq C\big(1 + |x|^2 + |p|^2 + |q|^2 + M_2(m)\big). 
\end{align*}
From (the proof of) \cite[Theorem 2.11]{CST}, it is now easy to conclude. Indeed, the explicit expression (2.21) for the error in the proof of \cite[Theorem 2.11]{CST} is easily bounded by $C(1 + |x|^2 + M_2(m))$ provided that $|\frac{\delta^2U}{\delta m^2}| \leq C(1 + |x|^2 + |p|^2 + |q|^2)$, so applying this reasoning with $U(m) = |\cU(x,m)|^2$ gives 
\begin{align*}
e_1 \leq C(1 + |x|^2 + M_2(m)),
\end{align*}
as desired. 

To handle the error term $e_2$ we use the same reasoning as above to control $\frac{\delta^2}{\delta m^2} \cU$. More precisely, the boundedness of $D_{mm} \cU$ gives the bound
\begin{align*}
\Big|\frac{\delta^2}{\delta m^2} [\cU(x,\cdot)](m,p,q) \Big| \leq C \big(|p| + |q| \big), 
\end{align*}
and so (equation (2.21) in the proof of) \cite[Theorem 2.11]{CST} gives
\begin{align*}
|\cU(x,m) - \E[\cU(x, m_{\bm \xi}^N)]| \leq \frac{C}{N} \big(1 + M_2^{1/2}(m)\big). 
\end{align*}
Combine this with $|\cU(x,m)| \leq C(1 + |x| + M_2(m)^{1/2})$ to get 
\begin{align*}
e_2 \leq C(1 + |x|^2 + M_2(m)). 
\end{align*}
We have now completed the proof under the additional assumption that $|D\cU(x, \cdot)|^2$ admits linear functional derivatives of order two which satisfy \eqref{linderivrelation}. But in fact this holds true under Assumption \ref{assump:cst}, thanks to \cite[Proposition 5.48]{cardelbook1}. Indeed, this result can be used (together with the smoothness assumptions on $D_m\cU$, $D_{mm} \cU$ to guarantee that $ \cU$ admits linear functional derivatives of order two satisfying
\begin{align} \label{linderivrelation1}
D_p \frac{\delta}{\delta m}[\cU(x, \cdot)](m,p) &= D_m  \cU(x,m,p), \nonumber \\
D_p D_q \frac{\delta^2}{\delta m^2} \cU(x, \cdot)](m,p,q) &= D_{mm} \cU(x,m,p,q)
\end{align}
and then the result follows. This completes the proof.
\end{proof}

\subsection{Proof of Theorem \ref{thm:pocdriftcontrol}}
\label{subsec:proofmaindrift}

This section is devoted to a proof of Theorem \ref{thm:pocdriftcontrol}, our main convergence result in the special case of linear drift control. We start with a key estimate which identifies the error term we must control in order to obtain Theorem \ref{thm:pocdriftcontrol}.

\begin{proposition} \label{prop:mainestdrift}
Let Assumptions \ref{assump:driftcontrolreg} and \ref{assump:driftcontrolstruct} hold. Let $\bm \alpha^N = (\alpha^{N,1},\dots,\alpha^{N,N})$ be a Nash equilibrium for the $N$-player game, and $(m,\alpha)$ be a MFE. Then there exist constants $N_0 \in \N$, $C  > 0$ such that for each $N \geq N_0$, we have 
\begin{align}
\label{eq:alpha.bound}
  \sum_{ i } \E\bigg[\int_0^T |\alpha^{N,i}_t - \alpha^i_t|^2 dt\bigg] \leq C \E \bigg[ \int_0^T  \sum_i |E_t^{F,i}|^2 dt +  \sum_i |E^{G,i}|^2 \bigg], 
\end{align}
where 
%$\Delta \alpha^i = \alpha^{N,i} - \alpha^i$ and 
$(\alpha^i)_{i = 1,\dots,N}$ are conditionally independent copies of the state control $\alpha$  and\begin{align}
\label{eq:def.EF}
  &E_t^{F,i} = \frac{1}{N} D_m F(X_t^{i}, \ov{m}_t^N, X_t^i) + D_x F(X_t^i, \ov{m}_t^N) -  D_x F(X_t^i, \ov{m}_t),   \\
  \vspace{.1cm}
  \label{eq:def.EG}
  &E^{G,i} = \frac{1}{N} D_m G(X_T^{i},\ov{m}_T^N, X_T^i) +D_x G(X_T^i, \ov{m}_T^N) - D_x G(X_T^i, \ov{m}_t), 
\end{align}
with $\ov{m}_t^N = \frac{1}{N} \sum_{i = 1}^{N} \delta_{X_t^i}$ and $\ov{m}_t = \cL(X^i_t\mid \cF^0_t)$.

As a consequence, we have
\begin{align} \label{penbound}
  \sum_{ i} \E\Big[\sup_{0 \leq t \leq T} |X^{N,i}_t - X^i_t|^2 \Big] \leq C \E \bigg[ \int_0^T  \sum_i |E_t^{F,i}|^2 dt +  \sum_i |E^{G,i}|^2 \bigg], 
\end{align}
where 
%$\Delta X^i = X^{N,i} - X^i$, 
$\bX^N = (X^{N,1},\dots,X^{N,N})$ is the state process corresponding to $\bm \alpha^N$, $(X^i)_{i = 1,\dots,N}$ are conditionally independent copies of the state process $X$ corresponding to $\alpha$.

\end{proposition}

\begin{proof}
%\subsubsection{The coupling procedure} \label{subsec.coup}

%{\color{blue}In this section we assume the existence of an equilibrium for the mean field game}.

The proof is relatively similar to the proof of the uniqueness result, Lemma \ref{lem:uniqueness}. In particular, we split the argument into roughly the same three steps.
\newline \newline \noindent
\textit{Step 1 - Applying the maximum principle:}
In view of Lemma \ref{lem:mpmfggen} and the fact that the unique maximizer of $a \mapsto L_0(x,a) + a\cdot y$ is given by $D_yH_0(x,y)$, it follows that the MFE $(m,\alpha)$ satisfies
\begin{equation*}
  \alpha_t = D_yH_0(X_t, Y_t)
 \end{equation*} 
 where the tuple $(X,Y,Z,Z^{0}) $ satisfies \eqref{eq:mkvfbsdedrift}
in the sense described in Lemma \ref{lem:mpmfggen}. We fix $N \in \N$ and for $i \in \{1,\dots,N\}$ we define $X^i, Y^i$, $Z^i$, $Z^{i,0}$, $\alpha^i$ to be the conditionally independent copies of $X$, $Y$, $Z$, $Z^0$, $\alpha$. It is clear from the definition of the conditionally independent copies that 
\begin{align*}
\sL\Big(X^i,Y^i,Z^i,Z^{i,0}, \alpha^i, W^0,W^i \Big) = \sL\Big(X,Y,Z,Z^0, \alpha, W^0,W \Big), 
\end{align*}
from which it follows that the tuple $(X^i,Y^i,Z^i,Z^{i,0})$ satisfies the BSDE \eqref{eq:mkvfbsdedrift} but with $W^i$ replacing $W^0$, and that $\alpha_t^i = D_y H_0(X_t^i, Y_t^i)$.

Likewise, by Lemma \ref{lem:N.player.smp}, we know that for all $N\in \N$ and every $i \in \{1, \dots, N\}$ it holds
\begin{equation*}
  \alpha^{N,i}_t = D_yH_0(X^{N,i}_t, Y^{N,i}_t)
\end{equation*}
with $(\bX^N, \bY^N, \bZ^N) = \big( (X^{N,1},\dots,X^{N,N}), (Y^{N,1},\dots,Y^{N,N}), (\bZ^{N,1},\dots,\bZ^{N,N})\big)$
satisfying \eqref{eq:nplayerfbsdedrift}. 
\newline \newline \noindent 
\textit{Step 2 - Computing $d \Delta X_t^i \cdot \Delta Y_t^i$:} To prove Theorem \ref{thm:pocdriftcontrol} it remains to obtain a (forward-backward) propagation of chaos for the particle system \eqref{eq:nplayerfbsdedrift}. We set
\begin{align*}
\Delta X_t^i = X_t^{N,i} - X_t^i, \quad \Delta Y_t^i = Y_t^{N,i} - Y_t^i, \quad \Delta \alpha_t^i = \alpha_t^{N,i} - \alpha_t^i.
\end{align*}
Note that since $\alpha_t^{N,i}$ minimizes $a \mapsto  L_0(X_t^{N,i}, a) + a \cdot Y_t^{N,i}$, we get
\begin{align*}
Y_t^{N,i} = - D_a L_0(X_t^{N,i}, \alpha_t^{N,i}), \text{ and similarly }
Y_t^i = - D_a L_0(X_t^i, \alpha_t^i).  
\end{align*}
With these identities in hand, we can compute
\begin{align} \label{xydynamics}
  d \Delta X_t^i \cdot \Delta Y_t^i = - &\bigg(\Delta \alpha^i_t \cdot  \big(D_a L_0(X_t^{N,i}, \alpha_t^{N,i}) - D_a L_0(X_t^i, \alpha_t^i) \big) + \Delta X_t^i \cdot \big(D_x L(X_t^{N,i}, \alpha_t^{N,i}) - D_x L_0(X_t^i, \alpha_t^i) \big) \nonumber \\
  &+ \Delta X_t^i \cdot \big(D_x F(X_t^{N,i}, m_{\bX_t^N}^N) - D_x F(X_t^i, \ov{m}_t^N)\big) & \nonumber  \\
  &+ \frac{1}{N} \Delta X_t^i \cdot \big(D_m F(X_t^{N,i}, m_{\bX_t^N}^N, X_t^{N,i}) - D_m F(X_t^{i}, \ov{m}_t^N, X_t^{i})  \big) + \Delta X_t^i \cdot E_t^{F,i} \bigg) dt + dM_t^i, 
\end{align}
where $E_t^{F,i}$ is defined by \eqref{eq:def.EF}.
% \begin{align*}
%   E_t^{F,i} = \frac{1}{N} D_m F(X_t^{N,i}, m_{\bX_t^N}^N, X_t^{N,i}) + D_x F(X_t^{i}, m_{\bX_t}^N) -  D_x F_t(X_t^i, \cL(X^i_t\mid \cF^0_t)). 
% \end{align*}
We can also write the terminal condition as
\begin{align*}
  \Delta X_T^i \cdot \Delta Y_T^i &= \Delta X_T^{i} \cdot \big(D_x G(X_T^{N,i}, m_{\bX_T^N}^N) - D_x G(X_T^i, \ov{m}_T^N) \big) \\
  &\quad+ \frac{1}{N} \Delta X_T^i \cdot \big(D_m G(X_T^{N,i}, m_{\bX_T^N}^N, X_T^{N,i}) - D_m G(X_T^{i}, \ov{m}_T^N, X_T^{i})  \big) \big) + \Delta X_T^i \cdot  E^{G,i}, 
\end{align*}
with $ E^{G,i}$ given by \eqref{eq:def.EG}.
% \begin{align*}
%   E^{G,i} = \frac{1}{N} D_m G(X_T^i, m_{\bX_T}^N, X_T^i) +D_x G(X_T^i, m_{\bX_T}^N) - D_x G(X_T^i, \cL(X^i_T\mid \cF^0_T)). 
% \end{align*}
Now we integrate \eqref{xydynamics}, sum over $i$, and take expectations to get
\begin{align*}
  -\E\bigg[& \int_0^T \sum_i \bigg(\Delta \alpha_t^i \cdot \big(D_a L_0(X_t^{N,i}, \alpha_t^{N,i}) - D_a L_0(X_t^i, \alpha_t^i) \big) + \Delta X_t^i \cdot \big(D_x L_0(X_t^{N,i}, \alpha_t^{N,i}) - D_x L_0(X_t^i, \alpha_t^i) \big) \nonumber \\
  & \quad \quad + \Delta X_t^i \cdot \big(D_x F(X_t^{N,i}, m_{\bX_t^N}^N) - D_x F(X_t^i, \ov{m}_t^N)\big) \\
  & +\frac{1}{N} \Delta X_t^i \cdot \big(D_m F(X_t^{N,i}, m_{\bX_t^N}^N, X_t^{N,i}) - D_m F(X_t^{i}, \ov{m}_t^N, X_t^{i})  \big)+ \Delta X_t^{N,i} \cdot E_t^{F,i} \bigg) dt \bigg] \\
  &= \E\Big[ \sum_i \Delta X_T^i \cdot \Delta Y_T^i\Big] 
  = \E\Big[\sum_i \Delta X_T^i \cdot\big(D_x G(X_T^{N,i}, m_{\bX_T^N}^N) - D_x G(X_T^{N,i}, \ov{m}_T^N) \big) \\
  & \qquad + \frac{1}{N} \Delta X_T^i \cdot \big(D_m G(X_T^{N,i}, m_{\bX_T^N}^N, X_T^{N,i}) - D_m G(X_T^{i}, \ov{m}_T^N, X_T^{i})  \big) \big)  +  \sum_i \Delta X_T^i \cdot E^{G,i}\Big]. 
\end{align*}
\textit{Step 3 - Using the monotonicity condition:}
Rearranging, and using the definitions of $C_L$, $C_{F}$, $C_{G}$ together with Lipschitz continuity of $D_m F$ and $D_m G$, we find that there is a constant $C$ independent of $N$ such that
\begin{align} \label{smallness}
C_L &\E \bigg[\int_0^T \sum_i |\Delta \alpha_t^i|^2 dt \bigg] \nonumber \\
&\leq \E \bigg[ \int_0^T \Big( -C_{F} \sum_i |\Delta X_t^i|^2 + \frac{C}{N} \sum_i |\Delta X_t^i|^2 + \sum_i |E_t^{F,i}||\Delta X_t^i|  
\Big) dt \nonumber  \\
&\qquad - C_{G} \sum_i |\Delta X_T^i|^2 +\frac{C}{N} \sum_i |\Delta X_T^i|^2 + \sum_i |E^{G,i}| |\Delta X_T^i| \bigg] \nonumber \\
&\leq -\Big((C_{G} \wedge 0)T + \frac{(C_{F} \wedge 0) T^2}{2} \Big) \E\bigg[\int_0^T \sum_i |\Delta \alpha_t^i|^2 dt\bigg] + \E \bigg[ \int_0^T \sum_i |E_t^{F,i}||\Delta X_t^i|  dt + \sum_i |E^{G,i}| |\Delta X_T^i| \bigg]  \nonumber \\
&\qquad + \frac{C}{N}\E\bigg[\int_0^T  \sum_i |\Delta X_t^i|^2 dt +  \sum_i |\Delta X_T^i|^2 \bigg]
\end{align}
where the last line follows from the fact that $\Delta X_t^i = \int_0^t \Delta \alpha_s^i ds$, and so 
\begin{align} \label{xalphabound}
  \E[|\Delta X_t^i|^2] = \E\bigg[\Big|\int_0^t  \Delta \alpha_s^i ds\Big|^2\bigg] \leq t \E\bigg[\int_0^t |\Delta \alpha_s^i|^2 ds\bigg]. 
\end{align}
From \eqref{smallness}, we conclude that 
\begin{align*}
\Big(C_L + (C_{G} \wedge 0) T &+ \frac{(C_{F} \wedge 0) T^2}{2}\Big) \E \bigg[ \int_0^T \sum_i |\Delta \alpha_t^i|^2 dt \bigg] \leq \E \bigg[ \int_0^T \sum_i |E_t^{F,i}||\Delta X_t^i|  dt + \sum_i |E^{G,i}| |\Delta X_T^i| \bigg] \\
&+ \frac{C}{N}\E\bigg[\int_0^T  \sum_i |\Delta X_t^i|^2 dt +  \sum_i |\Delta X_T^i|^2 \bigg].
\end{align*}
Next, we apply Young's inequality to find
\begin{align*}
\Big(C_L& + (C_{G} \wedge 0) T + \frac{(C_{F} \wedge 0) T^2}{2}\Big)\E \bigg[ \int_0^T \sum_i |\Delta \alpha_t^i|^2 dt \bigg] \\
&\leq \E \bigg[ \int_0^T \bigg( \frac{\delta}{2} \sum_i |\Delta X_t^i|^2 dt +  \frac{1}{2\delta}  \sum_i |E_t^{F,i}|^2 \bigg)dt +  \frac{\delta}{2} \sum_i  |\Delta X_T^i|^2 + \frac{1}{2\delta} |E^{G,i}|^2 \bigg] \\
&\quad + \frac{C}{N}\E\bigg[\int_0^T  \sum_i |\Delta X_t^i|^2 dt +  \sum_i |\Delta X_T^i|^2 \bigg] \\
&\leq \big(\frac{\delta T^2}{4} + \frac{\delta T}{2} \big) \E\bigg[ \int_0^T \sum_i |\Delta \alpha_t^i|^2 dt \bigg] + 
\frac{1}{2 \delta} \E \bigg[ \int_0^T  \sum_i |E_t^{F,i}|^2 dt +  \sum_i |E^{G,i}|^2 \bigg] \\
&\quad + \frac{C}{N}\E\bigg[\int_0^T  \sum_i |\Delta \alpha_t^i|^2 dt\bigg].
\end{align*}
By choosing $\delta$ small enough (independently of $N$), we conclude that because $C_L + (C_G \wedge 0) T + \frac{(C_F \wedge 0) T^2}{2} > 0$, there is a constant $C$ such that for all $N$ large enough, 
\begin{align*}
\E\bigg[\int_0^T \sum_i |\Delta \alpha_t^i|^2 dt \bigg] \leq C \E \bigg[ \int_0^T  \sum_i |E_t^{F,i}|^2 dt +  \sum_i |E^{G,i}|^2 \bigg],
\end{align*}
and thereby \eqref{penbound} easily follows.
\end{proof}

\begin{proof}[Proof of Theorem \ref{thm:pocdriftcontrol}]
We fix an $N$, and note that without loss of generality we can assume that $N$ is large enough that the conclusion of Lemma \ref{lem:uniqueness} applies.
  Thanks to Proposition \ref{prop:mainestdrift}, we need only to estimate the error terms $E_t^{F,i}$ and $E^{G,i}$ appearing in the statement of the same Proposition.  Recalling the definition, we have
\begin{align*}
  \E\bigg[\int_0^T |E_t^{F,i}|^2 dt\bigg] &\leq 2 \E\bigg[\int_0^T |\frac{1}{N} D_m F(X_t^{i}, \ov{m}_t^N, X_t^{i})|^2 dt \bigg] + 2 \E\bigg[\int_0^T |D_x F(X_t^i, \ov{m}_t^N) - D_x F(X_t^i, \ov{m}_t)|^2 dt \bigg] \\
  &\leq \frac{C}{N^2} + 2 \E\bigg[\int_0^T|D_x F(X_t^i, \ov{m}_t^N) - D_x F(X_t^i, \ov{m}_t)|^2dt\bigg].
\end{align*}
We used here the fact that $D_m F$ is uniformly Lipschitz and the square integrability of $X^i$.
Now the goal becomes to bound the quantity
\begin{align*}
\E\big[|D_x F(X_t^i, \ov{m}_t^N) - D_x F(X_t^i, \ov{m}_t)|^2\big]. 
\end{align*}
First, set $\ov{m}_t^{N,-i} = \frac{1}{N-1} \sum_{j \neq i} \delta_{X_t^j}$. It is easy to check that 
\begin{align*}
\E\big[\cW^2_2(\ov{m}_t^N, \ov{m}_t^{N,-i}) \big] \leq \frac{C}{N} \E[|X_t^1|^2] \leq \frac{C}{N} 
\end{align*}
and so 
\begin{align*}
\E\big[|D_x F(X_t^i, \ov{m}_t^N) - D_x F(X_t^i, \ov{m}_t)|^2\big] \leq \frac{C}{N} + 2\E\big[|D_x F(X_t^i, \ov{m}_t^{N,-i}) - D_x F(X_t^i, \ov{m}_t)|^2\big].  
\end{align*}
Next we notice that 
\begin{align*}
\E\big[|D_x F(X_t^i&, \ov{m}_t^{N,-i}) - D_x F(X_t^i, \ov{m}_t)|^2\big] = \E\Big[ \E\big[|D_x F(X_t^i, \ov{m}_t^{N,-i}) - D_x F(X_t^i, \ov{m}_t)|^2 | \sF_t^{0} \big] \Big] \\
&= \E\bigg[\int_{\R^n} \E\big[|D_x F(x, \ov{m}_t^{N,-i}) - D_x F(x, \ov{m}_t)|^2 | \sF_t^0 \big] \ov{m}_t(dx) \bigg] \\
& \leq \E\bigg[ \int_{\R^n} \frac{C}{N} \big(1 + |x|^2 + M_2(\ov{m}_t) \big) \ov{m}_t(dx) \bigg] \\
& \leq \frac{C}{N}\E\big[\big(1 + M_2(\ov{m}_t)\big) \big] \leq \frac{C}{N},
\end{align*}
where we have used the conditional independence of $(X^{i})_{i = 1,\dots,N}$ and Lemma \ref{lem:errorterms}. This shows that we have 
\begin{align*}
\E[|E_t^{F,i}|^2] \leq \frac{C}{N}.
\end{align*}
Obtaining a similar bound for $\E[|E^{G,i}|^2]$ by the same argument and then plugging these bounds into \eqref{eq:alpha.bound} and \eqref{penbound} gives 
\begin{align}
  \sum_i\E\Big[\sup_{t\in [0,T]}|X^{N,i}_t-X^i_t|^2\Big]\le C\quad \mathrm{and}\quad \E\bigg[\int_0^T \sum_i |\Delta \alpha_t^i|^2 dt \bigg] \leq C, 
\end{align}
and so by Lemma \ref{lem:exchangeable} the proof is complete.
\end{proof}

\subsection{Proof of Theorem \ref{thm:pocaffine}}
\label{subsec:proofaffine}

\begin{proposition} \label{prop:mainestaffine}
 Let Assumptions \ref{assump:affinereg} and \ref{affinestruct} hold, and suppose that for each $N$, $\bm\alpha^N = (\alpha^{N,1},\dots,\alpha^{N,N})$ is a Nash equilibrium for the $N$-player game, and that $(m,\alpha)$ be a MFE. Then there exist constants $N_0 \in \N$, $C  > 0$ such that for each $N \geq N_0$,
\begin{align}
\label{eq:alpha.bound2}
  \sum_{ i } \E\bigg[\int_0^T |\alpha^{N,i}_t - \alpha^i_t|^2 dt\bigg] \leq C \E \bigg[ \int_0^T  \sum_i |E_t^{F,i}|^2 dt +  \sum_i |E^{G,i}|^2 \bigg], 
\end{align}
where 
%$\Delta \alpha^i = \alpha^{N,i} - \alpha^i$ and 
$(\alpha^i)_{i = 1,\dots,N}$ are conditionally independent copies of the state control $\alpha$.
As a consequence, we have
\begin{align} \label{penbound2}
  \sum_{ i} \E\Big[\sup_{0 \leq t \leq T} |X^{N,i}_t - X^i_t|^2 \Big] \leq C \E \bigg[ \int_0^T  \sum_i |E_t^{L,i}|^2 dt +  \sum_i |E^{G,i}|^2 \bigg], 
\end{align}
where 
%$\Delta X^i = X^{N,i} - X^i$, 
$\bX^N = (X^{N,1},\dots,X^{N,N})$ is the state process corresponding to $\bm \alpha^N$, $(X^i)_{i = 1,\dots,N}$ are conditionally independent copies of the state process $X$ corresponding to $\alpha$, and
\begin{align}
\label{eq:def.errorsaffine}
  &E_t^{L,i} =  | \frac{1}{N} D_m L(X_t^{i}, \alpha_t^{i}, \ov{m}_t^N, X_t^{i}) + D_x L(X_t^i,\alpha_t^i, \ov{m}_t^N) -  D_x L(X_t^i,\alpha_t^i, \ov{m}_t)|  \nonumber \\ &\qquad + |D_a L(X_t^i,\alpha_t^i, \ov{m}_t^N) -  D_a L(X_t^i, \alpha_t^i, \ov{m}_t)|,   \\
  \vspace{.1cm}
  \label{eq:def.EG2}
  &E^{G,i} = \frac{1}{N} D_m G(X_T^{i}, \ov{m}_T^N. X_T^{i}) +D_x G(X_T^i, \ov{m}_T^N) - D_x G(X_T^i, \ov{m}_T), 
\end{align}
with $\ov{m}_t^N = \frac1N\sum_{i = 1}^{N} \delta_{X_t^i}$ and $\ov{m}_t = \cL(X^i_t\mid \cF^0_t)$.
\end{proposition}

\begin{proof}
The proof is very similar to that of Proposition \ref{prop:mainestdrift}. 
\newline  \newline 
\textit{Step 1 - Applying the maximum principle:} We note that $(m,\alpha)$ satisfies
\begin{equation*}
  \alpha_t = \widehat{\alpha}(X_t,Y_t,Z_t,Z_t^0, m_t)
 \end{equation*} 
 where the tuple $(X,Y,Z,Z^{0}) $ satisfies \eqref{eq:nplayerfbsdeaffine}
in the sense described in Lemma \ref{lem:mpmfggen}. We fix $N \in \N$ and for $i \in \{1,\dots,N\}$ we define $X^i, Y^i$, $Z^i$, $Z^{i,0}$, $\alpha^i$ to be the conditionally independent copies of $X$, $Y$, $Z$, $Z^0$, $\alpha$. We note that once again it is clear that the tuple $(X^i,Y^i,Z^i,Z^{i,0})$ satisfies the BSDE \eqref{eq:nplayerfbsdeaffine} but with $W^i$ replacing $W$, and that $\alpha_t^i = \widehat{\alpha}(X_t^i,Y_t^i,Z_t^i,Z_t^{i,0}, \sL(X_t^i | \sF_t^0))$.

Likewise, by Lemma \ref{lem:N.player.smp}, we know that for all $N\in \N$ and every $i \in \{1, \dots, N\}$ it holds
\begin{equation*}
  \alpha_t^{N,i} = \widehat{\alpha}(X_t^{N,i},Y_t^{N,i},Z_t^{N,i,i},Z_t^{N,i,0}, m_{\bX_t^N}^N)
\end{equation*}
with $(\bX^N, \bY^N, \bZ^N) = \big( (X^{N,1},\dots,X^{N,N}), (Y^{N,1},\dots,Y^{N,N}), (\bZ^{N,1},\dots,\bZ^{N,N})\big)$
satisfying \eqref{eq:nplayerfbsdeaffine}. 
\newline \newline  \noindent 
\textit{Step 2: Computing $d \Delta X_t^i \cdot \Delta Y_t^i$:} We again set
\begin{align*}
\Delta X_t^i = X_t^{N,i} - X_t^i, \quad \Delta Y_t^i = Y_t^{N,i} - Y_t^i, \quad \Delta \alpha_t^i = \alpha_t^{N,i} - \alpha_t^i. 
\end{align*}
Using the fact that $\alpha_t^i$ minimizes 
\begin{align*}
a \mapsto L(X_t^i, a, m_t) + y \cdot \big(B_1 a\big) + z \cdot \big( \Sigma_1 z \big) + z^0 \cdot \big(\Sigma_1^0 z^0 \big),  
\end{align*}
we deduce that 
\begin{align*}
D_a L(X_t^i, \alpha_t^i, m_t) = - \Big(B_1^\top Y_t^i + \Sigma_1^\top Z_t^i + (\Sigma_1^0)^\top Z_t^{i,0} \Big)
\end{align*}
and a similar argument gives 
\begin{align*}
D_a L(X_t^{N,i}, \alpha_t^{N,i}, m_{\bX_t^N}^N) = - \Big(B_1^\top Y_t^{N,i} + \Sigma_1^\top Z_t^{N,i,i} + (\Sigma_1^0)^\top Z_t^{N,i,0} \Big).
\end{align*}
Using these identities, we can write the dynamics of the process $\Delta X_t^i \cdot \Delta Y_t^i$ as
\begin{align} \label{firstxycompaffine}
  &d \Delta X_t^i \cdot \Delta Y_t^i = - \Big( \Delta X_t^i \cdot \big(D_x L(X_t^{N,i}, \alpha_t^{N,i}, m_{\bX_t^N}^N) - D_x L(X_t^{i}, \alpha_t^{i}, \ov{m}_t^N)\big) \nonumber  \\ 
  \nonumber &\qquad+ \Delta \alpha_t^i \cdot \big(D_a L(X_t^{N,i}, \alpha_t^{N,i}, m_{\bX_t^N}^N) - D_a L(X_t^{i}, \alpha_t^{i}, \ov{m}_t^N) \\
  &\qquad + \frac{1}{N} \Delta X_t^i \cdot \big(D_m L(X_t^{N,i}, \alpha_t^{N,i}, m_{\bX_t^N}^N, X_t^{N,i}) - L(X_t^{i}, \alpha_t^{i}, \ov{m}_t^N,X_t^i) \big)
  + \Delta X_t^i \cdot E_t^{L,i,x} + \Delta \alpha_t^i \cdot E_t^{L,i,a} \big) \Big)  dt + dM_t^i, 
\end{align}
where $M^i$ is a martingale and we have set 
\begin{align*}
E_t^{L,i,x} &= 
\frac{1}{N}  D_m L(X_t^{i}, \alpha_t^{i}, \ov{m}_t^N, X_t^{i}) + D_x L(X_t^i,\alpha_t^i, \ov{m}_t^N) -  D_x L(X_t^i,\alpha_t^i, \ov{m}_t)  \nonumber \\
E_t^{L,i,a} &= D_a L(X_t^i,\alpha_t^i, \ov{m}_t^N) -  D_a L(X_t^i, \alpha_t^i, \ov{m}_t).
\end{align*}
We can also write the terminal condition as
\begin{align*}
  \Delta X_T^i \cdot \Delta Y_T^i &= \Delta X_T^{i} \cdot \big(D_x G(X_T^{N,i}, m_{\bX_T^N}^N) - D_x G(X_T^i, \ov{m}_T^N) \big) \\
  &\quad+ \frac{1}{N} \Delta X_T^i \cdot \big(D_m G(X_T^{N,i}, m_{\bX_T^N}^N, X_T^{N,i}) - D_m G(X_T^{i}, \ov{m}_T^N, X_T^{i})  \big) \big) + \Delta X_T^i \cdot  E^{G,i}, 
\end{align*}
with $ E^{G,i}$ given by \eqref{eq:def.EG2}.
% \begin{align*}
%   E^{G,i} = \frac{1}{N} D_m G(X_T^i, m_{\bX_T}^N, X_T^i) +D_x G(X_T^i, m_{\bX_T}^N) - D_x G(X_T^i, \cL(X^i_T\mid \cF^0_T)). 
% \end{align*}
\newline \newline \noindent 
\textit{Step 3: Using the monotonicity condition:}
Now we integrate \eqref{xydynamics}, sum over $i$, use Lemma \ref{lem:dispmonotoneprojectedaffine} and take expectations to find that there is a constant $C$ independent of $N$ such that
\begin{align} %\label{smallnessaffine}
C_L &\E \bigg[\int_0^T \sum_i |\Delta \alpha_t^i|^2 dt \bigg] \nonumber \\
&\leq \E \bigg[ \int_0^T \sum_i \Big( |E_t^{L,i,x}||\Delta X_t^i|   + |E_t^{L,i,a}| |\Delta \alpha_t^i|
\Big) dt + \sum_i |E^{G,i}| |\Delta X_T^i| \bigg] \nonumber  \\
&\qquad + \frac{C}{N} \E\bigg[\sum_i \int_0^T |\Delta X_t^i|dt + \sum_i |\Delta X_T^i|  \bigg]
\nonumber \\
&\leq \frac{C_L}{2} \E\bigg[\int_0^T \sum_i |\Delta \alpha_t^i|^2 dt\bigg] + C \E\bigg[\int_0^T \sum_i \Big( |E_t^{L,i,a}|^2 + |E_t^{L,i,x}|^2 \Big) 
dt + \sum_i |E^{G,i}|^2 \bigg] \nonumber  \\
&\leq \big(\frac{C_L}{2} + \frac{C}{N}\big) \E\bigg[\int_0^T \sum_i |\Delta \alpha_t^i|^2 dt\bigg] + C \E\bigg[\int_0^T \sum_i |E_t^{L,i}|^2
dt + \sum_i |E^{G,i}|^2 \bigg] 
\end{align}
where the last line follows from the fact that
\begin{align} \label{xalphaboundaffine}
  \E[\sup_{0 \leq t \leq T} |\Delta X_t^i|^2] \leq C \E\bigg[\int_0^T |\Delta \alpha_t^i|^2 dt\bigg]. 
\end{align}
This concludes the proof.
\end{proof}

\begin{proof}[Proof of Theorem \ref{thm:pocaffine}]
The error terms $E^{L,i}$ and $E^{G,i}$ appearing in the statement of Proposition \ref{prop:mainestaffine} are estimated exactly as the terms $E^{F,i}$, $E^{G,i}$ are estimated in the proof of Theorem \ref{thm:pocdriftcontrol}.
\end{proof}

\subsection{Proof of Theorem \ref{thm:pocvolcontrol}}
\label{subsec:proofmain}

We start, as in the previous section, by using the structural condition to get the key error bound.

\begin{proposition} 
\label{prop:mainest}
Let Assumptions \ref{assump:controlledvolreg} and \ref{assump:controlledvolstruct} hold. Let $\bm \alpha^N = (\alpha^{N,1},\dots,\alpha^{N,N})$ be a Nash equilibrium for the $N$-player game, and $(m,\alpha)$ be a MFE. Then there exist constants $N_0 \in \N$, $C  > 0$ such that for each $N \geq N_0$,
\begin{align*}
\sum_{ i = 1}^N \E\Big[\sup_{0 \leq t \leq T} |X^{N,i}_t - X^i_t|^2 \Big] \leq C \E \bigg[ \int_0^T  \sum_i \Big(|E_t^{b,i}|^2 + |E_t^{\sigma,i}|^2 + |E_t^{\sigma^0,i}|^2  + |E_t^{f,i}|^2 \Big) dt  +  \sum_i |E^{G,i}|^2 \bigg], 
\end{align*}
where 
$\bX^N = (X^{N,1},\dots,X^{N,N})$ is the state process corresponding to $\bm \alpha^N$, $(X^i)_{i = 1,\dots,N}$ are conditionally independent copies of the state process $X$ corresponding to $\alpha$, and
\begin{align*}
  &E_t^{b,i} = D_y H(X_t^i, Y_t^i, Z_t^i, Z_t^{i,0}, \ov{m}_t^N) - D_y H(X_t^i, Y_t^i, Z_t^i, Z_t^{i,0}, \ov{m}_t)
  \\
  &E_t^{\sigma,i} = D_z H(X_t^i, Y_t^i, Z_t^i, Z_t^{i,0}, \ov{m}_t^N) - D_z H(X_t^i, Y_t^i, Z_t^i, Z_t^{i,0}, \ov{m}_t), 
  \\
  &E_t^{\sigma^0,i} = D_{z^0} H(X_t^i, Y_t^i, Z_t^i, Z_t^{i,0}, \ov{m}_t^N) - D_{z^0} H(X_t^i, Y_t^i, Z_t^i, Z_t^{i,0}, \ov{m}_t), 
  \\
  &E_t^{L,i} = \frac{1}{N} D_m H(X_t^{i}, Y_t^{i}, Z_t^{i}, Z_t^{i,0}, \ov{m}_t^N, X_t^{i})
  + D_x H(X_t^i, Y_t^i, Z_t^i, Z_t^{i,0}, \ov{m}_t^N) \\
  & \hspace{2cm}  - D_x H(X_t^i, Y_t^i, Z_t^i, Z_t^{i,0}, \ov{m}_t), \\
  \vspace{.1cm}
  &E^{G,i} = \frac{1}{N} D_m G(X_T^{i}, \ov{m}_T^N, X_T^i) + D_x G(X_T^{i}, \ov{m}_t^N) - D_x G(X_T^i, \ov{m}_T), 
\end{align*}
with $\ov{m}_t^N = \frac1N\sum_{i = 1}^{N} \delta_{X_t^i}$, and $\ov{m}_t = \sL(X_t^i | \sF_t^0)$.

Suppose that in addition the optimizer $\hat{\alpha} = \widehat{\alpha}(x,y,z,z^0,m)$ satisfies
\begin{enumerate}
\item $\hat{\alpha}$ is Lipschitz
\item For each fixed $(x,y,z,z^0)$, the map $m \mapsto \widehat{\alpha}(x,y,z,z^0,m)$ admits derivatives 
\begin{align*}
D_m \widehat{\alpha}(x,y,z,z^0,m,p) = D_m\big[\widehat{\alpha}(x,y,z,z^0, \cdot)](m,p), \\
 D_{mm} \widehat{\alpha}(x,y,z,z^0,m,p,q) = D_m\big[\widehat{\alpha}(x,y,z,z^0, \cdot)](m,p,q)
\end{align*}
which are bounded and Lipschitz continuous, uniformly in $(x,y,z,z^0)$.
\end{enumerate}
Then we have 
\begin{align*}
\sum_{ i = 1}^N \E\bigg[\int_0^T | \alpha_t^{N,i} - \alpha^i_t|^2 dt \bigg] \leq \frac{C}{N} + C \E \bigg[ \int_0^T  \sum_i \Big(|E_t^{b,i}|^2 + |E_t^{\sigma,i}|^2 + |E_t^{\sigma^0,i}|^2  + |E_t^{f,i}|^2 \Big) dt  +  \sum_i |E^{G,i}|^2 \bigg], 
\end{align*}
where 
%$\Delta \alpha^i = \alpha^{N,i} - \alpha^i$ and 
$(\alpha^i)_{i = 1,\dots,N}$ are conditionally independent copies of the control $\alpha$.
\end{proposition}

\begin{proof}
We again proceed using the same three steps as in the proofs of Propositions \ref{prop:mainestdrift} and \ref{prop:mainestaffine}. The main difference in this proof is that we cannot directly bound $\alpha^{N,i} - \alpha^i$, and must instead proceed by estimating $X^{N,i} - X^i$ directly. 
\newline \newline \noindent
\textit{Step 1 - Applying the maximum principle:} By Lemma \ref{lem:mpmfggen}, we have $\alpha_t = \widehat{\alpha}(X_t,Y_t,Z_t,Z_t^0,m_t)$ for some solution $(X,Y,Z,Z^0)$ to the McKean-Vlasov FBSDE \eqref{eq:mkvfbsde}. For each $N$ and $i \in \{1,\dots,N\}$ we set $(X^i,Y^i, Z^i, Z^{i,0})$ to be the conditionally independent copies of $(X,Y,Z,Z^0)$ defined in Definition \ref{def:condindcopies}. As in the proof of Proposition \ref{prop:mainestdrift}, we observe that $(X^i,Y^i, Z^i, Z^{i,0})$ solves the equation \eqref{eq:mkvfbsde} but with $W^i$ replacing $W$, and $\alpha_t^i = \widehat{\alpha}(X_t^i, Y_t^i, Z_t^{i}, Z_t^{i,0}, \sL(X_t^i | \sF_t^0))$. Likewise, we have by Lemma \ref{lem:N.player.smp} $\alpha_t^{N,i} = \widehat{\alpha}(X_t^{N,i}, Y_t^{N,i}, Z_t^{N,i,i}, Z_t^{N,i,0}, m_{\bX_t^N}^N)$ for some tuple $(\bm X^N, \bY^N, \bZ^N)$ satisfying \eqref{eq:nplayerfbsde}. 
\newline \newline \noindent
\textit{Step 2 - Computing $d \Delta X_t^i \cdot \Delta Y_t^i$:} We now set 
\begin{align*}
\Delta X_t^i = X_t^{N,i} - X_t^i, \quad \Delta Y_t^i = Y_t^{N,i} - Y_t^i, \quad \Delta Z_t^i = Z_t^{N,i,i} - Y_t^i, \quad \Delta Z_t^{i,0} = Z_t^{N,i,0} - Z_t^{i,0}
\end{align*}
and compute 
\begin{align} \label{xycompcont}
d \Delta& X_t^i \cdot \Delta Y_t^i \nonumber  
\\
&=  \bigg( - \Delta X_t^i\cdot\Big( D_x H (X_t^{N,i},Y_t^{N,i},Z_t^{N,i,i},Z_t^{N,i,0}, m_{\bX_t^N}^N) - D_x H(X_t^i,Y_t^i,Z_t^i,Z_t^{i,0}, \ov{m}_t) \bigg)  \nonumber \\
&\quad + \Delta Y_t^i\cdot\Big(D_y H (X_t^{N,i},Y_t^{N,i},Z_t^{N,i,i},Z_t^{N,i,0}, m_{\bX_t^N}^N) - D_y H(X_t^i,Y_t^i,Z_t^i,Z_t^{i,0}, \ov{m}_t) \Big)  \nonumber \\ 
&\quad +  \Delta Z_t^{i}\cdot \Big(D_z H (X_t^{N,i},Y_t^{N,i},Z_t^{N,i,i},Z_t^{N,i,0}, m_{\bX_t^N}^N) - D_z H(X_t^i,Y_t^i,Z_t^i,Z_t^{i,0}, \ov{m}_t) \Big)  \nonumber \\ 
&\quad + \Delta Z_t^{i,0} \cdot \Big( D_{z^0} H (X_t^{N,i},Y_t^{N,i},Z_t^{N,i,i},Z_t^{N,i,0}, m_{\bX_t^N}^N) - D_{z^0} H(X_t^i,Y_t^i,Z_t^i,Z_t^{i,0}, \ov{m}_t) \Big) dt  \nonumber \\
& \quad - \frac{1}{N} \Delta X_t^i \cdot D_m H(X_t^{N,i},Y_t^{N,i},Z_t^{N,i,i},Z_t^{N,i,0}, m_{\bX_t^N}^N) dt + dM_t^i \nonumber \\
&= 
\bigg( - \Delta X_t^i\cdot\Big( D_x H (X_t^{N,i},Y_t^{N,i},Z_t^{N,i,i},Z_t^{N,i,0}, m_{\bX_t^N}^N) - D_x H(X_t^i,Y_t^i,Z_t^i,Z_t^{i,0}, \ov{m}_t^N) \bigg)  \nonumber \\
&\quad + \Delta Y_t^i\cdot\Big(D_y H (X_t^{N,i},Y_t^{N,i},Z_t^{N,i,i},Z_t^{N,i,0}, m_{\bX_t^N}^N) - D_y H(X_t^i,Y_t^i,Z_t^i,Z_t^{i,0}, \ov{m}_t^N) \Big)  \nonumber \\ 
&\quad +  \Delta Z_t^{i}\cdot \Big(D_z H (X_t^{N,i},Y_t^{N,i},Z_t^{N,i,i},Z_t^{N,i,0}, m_{\bX_t^N}^N) - D_z H(X_t^i,Y_t^i,Z_t^i,Z_t^{i,0}, \ov{m}_t^N) \Big)  \nonumber \\ 
&\quad + \Delta Z_t^{i,0} \cdot \Big( D_{z^0} H (X_t^{N,i},Y_t^{N,i},Z_t^{N,i,i},Z_t^{N,i,0}, m_{\bX_t^N}^N) - D_{z^0} H(X_t^i,Y_t^i,Z_t^i,Z_t^{i,0}, \ov{m}_t^N) \Big) dt   \nonumber \\ 
& \quad 
- \frac{1}{N} \Delta X_t^i \cdot \big(D_m H(X_t^{N,i}, Y_t^{N,i}, Z_t^{N,i,i}, Z_t^{N,i,0}, m_{\bX_t^N}^N, X_t^{N,i})
 - D_m H(X_t^{i}, Y_t^{i}, Z_t^{i}, Z_t^{i,0}, \ov{m}_t^N, X_t^i)
  \big)
\nonumber \\ 
& \quad + \Big(- \Delta X_t^i \cdot E_t^{L,i} + \Delta Y_t^i \cdot E_t^{b,i} + \Delta Z_t^i \cdot E_t^{\sigma,i} + \Delta Z_t^0 \cdot E_t^{\sigma^0,i} \Big)dt +  dM_t^i 
\end{align}
with $M^i$ being a martingale.
Similarly, we can write the terminal condition as
\begin{align*}
  \Delta X_T^i \cdot \Delta Y_T^i &= \Delta X_T^{i} \cdot \big(D_x G(X_T^{N,i}, m_{\bX_T^N}^N) - D_x G(X_T^i, \ov{m}_T^N) \big) \\
  &\quad+ \frac{1}{N} \Delta X_T^i \cdot \big(D_m G(X_T^{N,i}, m_{\bX_T^N}^N, X_T^{N,i}) - D_m G(X_T^{i}, \ov{m}_T^N, X_T^{i})  \big) \big) + \Delta X_T^i \cdot  E^{G,i}. 
\end{align*}
We now integrate \eqref{xycompcont} and take expectations, to find 
\begin{align*}
  \E[\Delta &X_T^i \cdot \Delta Y_T^i] = \E\big[\Delta X_T^{i} \cdot \big(D_x G(X_T^{N,i}, m_{\bX_T^N}^N) - D_x G(X_T^i, \ov{m}_T^N) \big) + \Delta X_T^i \cdot  E^{G,i}\big] \\
&= \E\bigg[ \int_0^T \bigg( - \Delta X_t^i\cdot\Big( D_x H (X_t^{N,i},Y_t^{N,i},Z_t^{N,i,i},Z_t^{N,i,0}, m_{\bX_t^N}^N) - D_x H(X_t^i,Y_t^i,Z_t^i,Z_t^{i,0}, \ov{m}_t^N) \bigg)  \nonumber \\
&\quad + \Delta Y_t^i \cdot \Big(D_y H (X_t^{N,i},Y_t^{N,i},Z_t^{N,i,i},Z_t^{N,i,0}, m_{\bX_t^N}^N) - D_y H(X_t^i,Y_t^i,Z_t^i,Z_t^{i,0}, \ov{m}_t^N) \Big) \nonumber \\ 
&\quad +  \Delta Z_t^{i}\cdot\Big(D_z H (X_t^{N,i},Y_t^{N,i},Z_t^{N,i,i},Z_t^{N,i,0}, m_{\bX_t^N}^N) - D_z H(X_t^i,Y_t^i,Z_t^i,Z_t^{i,0}, \ov{m}_t^N) \Big) \nonumber \\ 
&\quad +  \Delta Z_t^{i,0}\cdot\Big( D_{z^0} H (X_t^{N,i},Y_t^{N,i},Z_t^{N,i,i},Z_t^{N,i,0}, m_{\bX_t^N}^N) - D_{z^0} H(X_t^i,Y_t^i,Z_t^i,Z_t^{i,0}, \ov{m}_t^N) \Big)   \nonumber \\
& \quad - \frac{1}{N} \Delta X_t^i \cdot \big(D_m H(X_t^{N,i}, Y_t^{N,i}, Z_t^{N,i,i}, Z_t^{N,i,0}, m_{\bX_t^N}^N, X_t^{N,i})
 - D_m H(X_t^{i}, Y_t^{i}, Z_t^{i}, Z_t^{i,0}, \ov{m}_t^N, X_t^i)
  \big) \nonumber \\
& \quad - \Delta X_t^i \cdot E_t^{L,i} + \Delta Y_t^i \cdot E_t^{b,i} + \Delta Z_t^i \cdot E_t^{\sigma,i} + \Delta Z_t^{i,0} \cdot E_t^{\sigma^0,i} \bigg) dt \bigg]. 
\end{align*}
\textit{Step 3 - Using the monotonicity condition:}
Summing over $i$ and using the definitions of $C_H$, $C_F$ and $C_G$, together with Lemma \ref{lem:finitedimmonotone} and the Lipschitz continuity of $D_m H$ and $D_m G$, we infer that there is a constant $C$ independent of $N$ such that
\begin{align*}
C_H& \E\bigg[\int_0^T \sum_i |\Delta X_t^i|^2 dt \bigg] + C_G \E\Big[\sum_i |\Delta X_T^i|^2\Big] \\
&\leq \E \bigg[ \int_0^T \sum_i \Big( |E_t^{L,i}||\Delta X_t^i| + |E_t^{b,i}||\Delta Y_t^i| + |E_t^{\sigma,i}||\Delta Z_t^i| + |E_t^{\sigma^0,i}||\Delta Z_t^{i,0}| \Big)  dt + \sum_i |E_t^{G,i}| |\Delta X_T^i| \bigg] \\
&\quad + \frac{C}{N} \E\bigg[\sum_i \int_0^T \Big( |\Delta X_t^i|^2 + |\Delta Y_t^i|^2 + |\Delta Z_t^i|^2 + |\Delta Z_t^{i,0}|^2 \Big) dt + \sum_i |\Delta X_T^i|^2 \bigg].
\end{align*}
Applying Young's inequality we find that for each $\varepsilon > 0$, 
\begin{align} \notag
\E\bigg[\int_0^T &\sum_i |\Delta X_t^i|^2 dt + \sum_i |\Delta X_T^i|^2 \bigg]\\\notag
 &\leq \frac{C}{\varepsilon} \E \bigg[ \int_0^T  \sum_i \Big(|E_t^{L,i}|^2 + |E_t^{b,i}|^2 + |E_t^{\sigma,i}|^2 + |E^{\sigma^0,i}|^2 \Big) dt +  \sum_i |E^{G,i}|^2 \bigg] \nonumber  \\\label{firstest}
& \qquad + C \varepsilon \E\bigg[\int_0^T \sum_i \Big( |\Delta X_t^i|^2 + |\Delta Y_t^i|^2 + |\Delta Z_t^i|^2 + |\Delta Z_t^{i,0}|^2 \Big) dt \bigg] \nonumber \\
& \qquad  + \frac{C}{N} \E\bigg[\sum_i \int_0^T \Big( |\Delta X_t^i|^2 + |\Delta Y_t^i|^2 + |\Delta Z_t^i|^2 + |\Delta Z_t^{i,0}|^2 \Big) dt + \sum_i |\Delta X_T^i|^2 \bigg].
\end{align} 
Next, we use Theorem 4.2.3 of \cite{Zhang} (here we use Lipschitz continuity of $D_x H$) to conclude that for each fixed $i$, 
%\begin{align} \label{bsdestab}
%\E\bigg[&\sup_{0 \leq t \leq T} |\Delta Y_t^i| + \int_0^T 
%\Big(|\Delta Z_t^i|^2 + |\Delta Z_t^{i,0}|^2 \Big) dt \bigg]
%\nonumber \\
%&\leq C \E \bigg[\int_0^T \Big(|D_x H(X_t^{N,i}, Y_t^i, Z_t^i, Z^{i,0}, m_{\bX_t^N}^N) - D_x H(X_t^{i}, Y_t^i, Z_t^i, Z^{i,0}, \ov{m}_t^N)|^2 
%+ |E_t^{L,i}|^2 \Big) dt 
%\nonumber \\
%& \quad + |D_x G( X_T^{N,i}, m_{\bX_T^N}^N) - D_x G( X_T^i, \ov{m}_T^N) |^2 + |E^{G,i}|^2
%\bigg]. 
%\end{align}
%The terms on the right-hand side of \eqref{bsdestab} can each be estimated by using \eqref{firstest} and the Lipschitz continuity of $D_x H$, $D_xG$, which leads to
\begin{align} \label{bsdestab}
\sum_i \E\bigg[\sup_{0 \leq t \leq T} |\Delta Y_t^i| &+ \int_0^T 
\Big(|\Delta Z_t^i|^2 + |\Delta Z_t^{i,0}|^2 \Big) dt \bigg] \nonumber \\
 &\leq C \sum_i \E \bigg[ \int_0^T \Big( |E_t^{L,i}|^2 + |\Delta X_t^i|^2 \Big) dt +   |E^{G,i}|^2 + |\Delta X_T^i|^2 \bigg].
\end{align}
Plugging this into \eqref{firstest}, we arrive at 
\begin{align} \notag
  \E\bigg[\int_0^T \sum_i |\Delta X_t^i|^2 dt &+ \sum_i |\Delta X_T^i|^2 \bigg] \\\notag
  &\leq \frac{C}{\varepsilon} \E \bigg[ \int_0^T  \sum_i \Big(|E_t^{L,i}|^2 + |E_t^{b,i}|^2 + |E_t^{\sigma,i}|^2 + |E^{\sigma^0,i}_t|^2 \Big) dt +  \sum_i |E^{G,i}|^2 \bigg]   \\\label{secondest}
  & \qquad + C \varepsilon \sum_i \E\bigg[\int_0^T |\Delta X_t^i|^2 dt + |\Delta X_T^i|^2 \bigg] \nonumber  \\
& \qquad  + \frac{C}{N} \E\bigg[\sum_i \int_0^T |\Delta X_t^i|^2  dt + \sum_i |\Delta X_T^i|^2 \bigg].
\end{align} 
Choosing $\varepsilon$ small enough, we find that for all $N$ large enough, 
\begin{align} \label{thirdest}
  \E\bigg[\int_0^T \sum_i |\Delta X_t^i|^2 dt + \sum_i |\Delta X_T^i|^2 \bigg] &\leq C \E \bigg[ \int_0^T  \sum_i \Big(|E_t^{L,i}|^2 + |E_t^{b,i}|^2 + |E_t^{\sigma,i}|^2 + |E^{\sigma^0,i}_t|^2  \Big) dt +  \sum_i |E^{G,i}|^2 \bigg],
\end{align} 
and then plugging this into \eqref{bsdestab} gives 
\begin{align} \label{bsdestab2}
\sum_i \E\bigg[&\sup_{0 \leq t \leq T} |\Delta Y_t^i| + \int_0^T 
\Big(|\Delta Z_t^i|^2 + |\Delta Z_t^{i,0}|^2 \Big) dt \bigg]
\nonumber \\
&\leq  C \E \bigg[ \int_0^T  \sum_i \Big(|E_t^{L,i}|^2 + |E_t^{b,i}|^2 + |E_t^{\sigma,i}|^2 + |E_t^{\sigma^0,i}|^2  \Big) dt +  \sum_i |E^{G,i}|^2 \bigg].
\end{align}

Now since $D_yH$, $D_z H$, $D_{z^0}H$ are all Lipschitz, we can easily use \eqref{thirdest} and \eqref{bsdestab2} to obtain
\begin{align*}
\sum_i \E\big[& \sup_{0 \leq t \leq T} |\Delta X_t^i|^2 \big] 
\\
&\leq C \sum_i \E\bigg[\int_0^T \Big( |\Delta X_t^i|^2 + |\Delta Z_t^i|^2 + |\Delta Z_t^{i,0}|^2 + |E_t^{b,i}|^2 + |E_t^{\sigma,i}|^2 + |E_t^{\sigma^0,i}|^2 \Big)dt \bigg]  \\
&\leq C \sum_i \E \bigg[ \int_0^T \Big(|E_t^{L,i}|^2 + |E_t^{b,i}|^2 + |E_t^{\sigma,i}|^2 + |E^{\sigma^0,i}_t|^2 \Big) dt +  \sum_i |E^{G,i}|^2 \bigg]
\end{align*}
which completes the proof thanks to Lemma \ref{lem:exchangeable}. 

The estimate on $\alpha_t^i - \alpha_t^{N,i}$ when $\widehat{\alpha}$ satisfies the additional Lipschitz and smoothness conditions follows by first using the bounds on $D_m \widehat{\alpha}$ and $D_{mm} \widehat{\alpha}$ to get
\begin{align} \label{mdependence}
\E\bigg[\int_0^T |\widehat{\alpha}(X_t^i, Y_t^i,Z_t^i, Z_t^{i,0}, \ov{m}_t^N) - \widehat{\alpha}(X_t^i, Y_t^i,Z_t^i, Z_t^{i,0}, \ov{m}_t^N)|^2 dt\bigg] \leq \frac{C}{N}
\end{align}
using the same argument used to bound the error terms $E_t^{F,i}$ and $E^{G,i}$ in the proof of Theorem \ref{thm:pocdriftcontrol}, and then concluding that 
\begin{align*}
\E\bigg[\int_0^T &|\Delta \alpha_t^i|^2 dt \bigg] = \E\bigg[\int_0^T |\widehat{\alpha}(X_t^{N,i}, Y_t^{N,i}, Z_t^{N,i,i}, Z_t^{N,i,0}, m_{\bX_t^N}^N) - \widehat{\alpha}(X_t^{i}, Y_t^{i}, Z_t^{i}, Z_t^{i,0}, \ov{m}_t)  |^2 dt \bigg] \\
&\leq C\E\bigg[\int_0^T |\widehat{\alpha}(X_t^{N,i}, Y_t^{N,i}, Z_t^{N,i,i}, Z_t^{N,i,0}, m_{\bX_t^N}^N) - \widehat{\alpha}(X_t^{i}, Y_t^{i}, Z_t^{i}, Z_t^{i,0}, \ov{m}_t^N)  |^2 dt \bigg] \\
&\qquad + C\E\bigg[\int_0^T |\widehat{\alpha}(X_t^i, Y_t^i,Z_t^i, Z_t^{i,0}, \ov{m}_t) - \widehat{\alpha}(X_t^i, Y_t^i,Z_t^i, Z_t^{i,0}, \ov{m}_t^N)|^2 dt \bigg] \\
&\leq C \E\bigg[\int_0^T \Big(|\Delta X_t^i|^2 + |\Delta Y_t^i|^2 + |\Delta Z_t^i|^2 + |\Delta Z_t^{i,0}|^2 + \frac{1}{N} \sum_{j = 1}^N |\Delta X_t^j|^2 \Big) \bigg] + \frac{C}{N}, 
\end{align*}
where the last inequality uses both the Lipschitz property of $\widehat{\alpha}$ and \eqref{mdependence}. We conclude the proof by applying the estimates on $\Delta X^i$, $\Delta Y^i$, $\Delta Z^i$, and $\Delta Z^{i,0}$ obtained above in \eqref{bsdestab2}.
\end{proof}

\begin{proof}[Proof of Theorem \ref{thm:pocvolcontrol}]
The error terms $E^{L,i}$, $E^{b,i}$, $E^{\sigma,i}$, $E^{\sigma^0,i}$, and $E^{G,i}$ appearing in the statement of Proposition \ref{prop:mainest} are estimated exactly as the terms $E^{F,i}$, $E^{G,i}$ are estimated in the proof of Theorem \ref{thm:pocdriftcontrol}.
%\todo[inline]{check that $\alpha^i=\hat\alpha(X^i, Y^i,Z^i, Z^{0,i})$}
\end{proof}

\begin{proof}[Proof of Corollary \ref{cor:q.moment}]
  First assume that $q\neq 4$.
  As in the proof of Theorem \ref{thm:pocdriftcontrol}, we use Proposition \ref{prop:mainestdrift} to obtain
  \begin{align} \label{eq:bound.four.corol.proof}
  \notag
  \sum_{ i } \E\Big[ &\sup_{0 \leq t \leq T} |X_t^{N,i} - X_t^i|^2 \Big] \leq C \sum_i \E \bigg[ \int_0^T \Big(|E_t^{L,i}|^2 + |E_t^{b,i}|^2 + |E_t^{\sigma,i}|^2 + |E^{\sigma^0,i}_t|^2 \Big) dt +  \sum_i |E^{G,i}|^2 \bigg].
\end{align}
We can now use Lipschitz continuity of the derivatives of $H$ and $G$, together with the integrability of $m_t$ to estimate each of the terms on the right-hand side of \eqref{eq:bound.four.corol.proof}. For example, we have 
\begin{align}
\notag
\E\Big[ |E_t^{L,i}|^2 \Big] &\leq \frac{C}{N^2} \E\Big[ \big|D_m H(X_t^{i}, Y_t^{i}, Z_t^{i},\ov{m}_t^N, X_t^{i})\big|^2 \Big] \\
&\quad + C \E\Big[ |D_x H(X_t^i, Y_t^i, Z_t^i, Z_t^{i,0}, \ov{m}_t^N) - D_x H(X_t^i, Y_t^i, Z_t^i, Z_t^{i,0}, \ov{m}_t)|^2 \Big] \nonumber  \\
&\leq \frac{C}{N^2} \E\Big[1 + |X_t^{N,i}|^2 + |Y_t^{N,i}|^2 + |Z_t^{N,i,i}|^2 + |Z_t^{N,i,0}|^2 + \frac{1}{N} \sum_j |X_t^j|^2 \Big] + C \E[\cW_2^2(\ov{m}_t, \ov{m}_t^N)].
\end{align}
We next sum over $i$ and integrate in $t$, plug the result into \eqref{eq:bound.four.corol.proof}. This gives 
\begin{align*}
    \sum_{ i } \E\Big[ &\sup_{0 \leq t \leq T} |X_t^{N,i} - X_t^i|^2 \Big] \leq \frac{C}{N} + CN \E\bigg[\int_0^T \cW_2^2(\ov{m}_t^N, \ov{m}_t) dt + \cW_2^2(\ov{m}^N_T,\ov{m}_T) \bigg].
\end{align*}
To conclude, we use \cite[Theorem 1]{fournier2015rate} to see that
\begin{align}
  \E\Big[\E[\cW_2^2(\ov{m}_t^N, \ov{m}_t)|\sF^0_T] \bigg] 
  \label{eq:bound.FG.proof}
  &\le C\sup_{t\in [0,T]}\E[M_q(m_t)]r_{N,q}.
\end{align}
If $q=4$, then it holds $\sup_{t\in [0,T]}\E[M_q(m_t)]<\infty$ for every $q\in (2,4)$.
Thus, we still have \eqref{eq:bound.FG.proof}.
  The estimate for $|\alpha^{N,i}-\alpha^i|$ follows as in the proof of Proposition \ref{prop:mainest}. 
\end{proof}

\section{Games with infinite horizon} \label{subsec:infinite}
In this section we turn our attention to non-cooperative games with infinite horizon.
Using again the displacement semi--monotonicity property of the coupling function $F$ as well as convexity of the Lagrangian $L$, we will show that the infinite horizon mean field game admits a unique mean field equilibrium, and that any sequence of Nash equilibria of the associated finite player games converges, in a strong sense, to the mean field equilibrium with an explicit rate.
This seems to be the first convergence result for games with infinite horizon (although for non-degenerate volatility and Lasry-Lions monotone data, a convergence result for closed-loop Nash equilibria should follow from the existence of a smooth solution to the master equation, as constructed in \cite{cardporetta}.)

Let $r>0$ be a fixed discounting factor.
In this section we continue working in the probabilistic setup introduced in Section \ref{subsec:probsetup}, with the only exception that the independent Brownian motions $(W^0,W^1,\dots,W^N)$ as well as the various filtrations are defined on the time interval $[0,\infty)$.
Furthermore, we let the set of admissible controls in the $N$--player game be defined as
\begin{align} \label{def.an.inf}
  \sA_N \coloneqq \bigg\{\text{$\bbF^N$-progressive $\R^k$-valued processes $\alpha = (\alpha_t)_{0 \leq t <\infty }$ such that $\E\Big[\int_0^\infty\e^{-rt} |\alpha_t|^2 dt\Big] < \infty$ } \bigg\}.
\end{align}
%{\color{red}we need to impose $\int_0^\infty\e^{-rt}(L(0,0) + F_t(0,0))dt <\infty$ if time dependence.}
In the infinite horizon case, our data consists of functions 
\begin{align*}
L = L(x,a) : \R^n \times \R^n \to \R, \quad F = F(x,m) : \R^n \times \spt \to \R, \quad \Sigma, \Sigma^0 \in \R^{n \times d}
\end{align*}
together with the discount factor $r$ and the initial condition $m_0 \in \spt$. 
Player $i$ minimizes the cost functional $J^{N,i}:\sA_N^N\to \R$ given by
\begin{align*}
  J^{N,i}(\bm \alpha) = J^{N,i}(\alpha^1,\dots,\alpha^N) = \E\bigg[ \int_0^\infty\e^{-rt} \big(L(X_t^i,\alpha_t^i) + F(X_t^i,m_{\bX_t^N}^N) \big) dt  \bigg]
\end{align*}
subject to
\begin{align*}
  dX_t^i =  \alpha_t^i dt + \Sigma dW_t^i + \Sigma^0 dW_t^0, \quad X_0^i = \xi^i.
\end{align*}
The associated MFG (for a generic idiosyncratic noise $(W, \xi)$) is defined as in the finite horizon case.
In fact, we let $\sA$ be the set of admissible controls with
\begin{align} \label{def.an.mfg}
  \sA \coloneqq \bigg\{\text{$\bbF$-progressive $\R^k$--valued processes $\alpha = (\alpha_t)_{0 \leq t<\infty}$ such that $\E\Big[\int_0^\infty\e^{-rt} |\alpha_t|^2 dt\Big] < \infty$ } \bigg\}.
\end{align}
Given an $\bbF^0$--progressive $\sP_2(\R^n)$--valued process $m = (m_t)_{0\le t<\infty}$, consider the cost
\begin{align} \label{mfoptinf}
  J_m(\alpha) = \E\bigg[  \int_0^\infty\e^{-rt} \big( L(X_t, \alpha_t) + F(X_t,m_t) \big) dt  \bigg]
\end{align}
that is minimized over $\sA$
subject to the controlled state dynamics
\begin{align} \label{dynamicsdef.inf}
  dX_t =  \alpha_t dt +\Sigma d W_t + \Sigma^0d W_t^0, \quad X_0 = \xi. 
\end{align}
A mean  field equilibrium is a pair $(m, \alpha)$ where $m$ is an $\bbF^0$-progressive $\sP_2(\R^n)$-valued process and $\alpha \in \sA$ is a minimizer of $J_{m}$ with corresponding state process $X$, such that
\begin{align}
  m_t = \sL(X_t | \sF_t^0), \,\, d\bP\text{--a.e. for each } t>\infty.
\end{align}

Here, we let the (reduced) Hamiltonian $H$ be defined as
\begin{equation}
  \label{eq.H.inf}
  H(x,y) := \inf_{a\in \R^k}(L(x,a) + a\cdot y).
\end{equation}
We will make the following assumption:
\begin{assumption}[Main condition, infinite horizon case] \label{assump:driftcontrolstruct.inf.T}
The functions $L, F$ and $H$ satisfy the following conditions in all arguments:
\begin{itemize}
    \item[(i)] The Lagrangian $L$ is $C^1$, and the derivatives $D_xL$ and $D_a L$ are Lipschitz continuous. The coupling term $F$ is $C^1$, and the maps 
\begin{align*}
&\R^n \times \spt \ni (x,m) \mapsto D_x F(x,m)
%, \quad \R^n \times \spt \ni (x,m) \mapsto D_x G(x,m), \\
\quad\mathrm{and}\quad \R^n \times \spt \times \R^n \ni (x,m,y) \mapsto D_m F(x,m,y) %, \quad \R^n \times \spt \times \R^n \ni (x,m,y) \mapsto D_m G(x,m,y)
\end{align*}
are each Lipschitz continuous in all arguments.
  \item[(ii)] The Lagrangian $L$ is jointly convex in $(x,a)$, and there is a constant $C_L > 0$ such that
\begin{align} \label{jointconvex.Tinf}
\big(D_x L(x,a) - D_x L(x',a')\big)\cdot(x - x') + \big(D_a L(x,a) - D_a L(x',a')\big)\cdot(a - a') \geq C_L |a - a'|^2
\end{align}
for all $x,x'\in \R^n$ and $,a,a' \in \R^k$. 
  \item[(iii)] There is a real number $C_{F}$ such that $F$ is $ C_{F}$-displacement semi-monotone, and we have
\begin{align} \label{convexmonotone.Tinf}
  r^2> 4\frac{C_F^-}{C_L}, \quad \text{where } C_F^- = - ( C_F \wedge 0).
\end{align} 
  \item[(iv)] For each fixed $x$, the map $m \mapsto D_x F(x,m)$ admits two derivatives 
\begin{align*}
D_m D_x F(x,m,y) = D_m\big[D_x F(x,\cdot)](m,y), \quad D_{mm} D_x F(x,m,y,z) = D_{mm} \big[D_x F(x,\cdot)\big](m,y,z), %\\
%D_m D_x G(x,m,y) = D_m\big[D_x G(x,\cdot)](m,y), \quad D_{mm} D_x G(x,m,y,z) = D_{mm} \big[D_x G(x,\cdot)\big](m,y,z)
\end{align*}
which are bounded and Lipschitz continuous in $m$, uniformly in $x$.
\end{itemize}
\end{assumption}
\begin{lemma}
  \label{lem.existence.inf}
%  Assume that there is $C_L>0$ such that \eqref{jointconvex} is satisfied, that $F$ is $C_F$--displacement semi monotone for some $C_F\in \R$ and that the function $(x,a)\mapsto L(x,a) + F(x,m)$ is convex for all $m$.
Let Assumption \ref{assump:driftcontrolstruct.inf.T}.(iii) hold.
    Then, if a MFE $(m,\alpha)$ exists, it satisfies
  \begin{equation}
  \label{eq:mfe.rep.inf}
    \alpha_t = D_yH(X_t, Y_t)\quad dt\times d\P\text{--a.s.}\quad \mathrm{and}\quad m_t = \sL(X_t|\sF^0_t) \quad \P\text{--a.s. for all } t\ge0
  \end{equation}
  where $(X,Y,Z,Z^0)$ solves the conditional McKean--Vlasov FBSDE
\begin{align} \label{mpmfg.inf}
\begin{cases} \vspace{.1cm}
  dX_t = \alpha_t dt + \Sigma dW_t + \Sigma_0 dW_t^0,  \\ \vspace{.1cm}
  dY_t = - \Big(D_x L(X_t, \alpha_t) + D_x F(X_t, \sL(X_t | \sF_t^0)) - rY_t \Big) dt + Z_t  dW_t + Z_t^{0}  dW_t^0, \\
  X_0 = \xi. 
\end{cases}
\end{align}
Moreover, if for every $m$ the function $(x,a)\mapsto L(x,a) + F(x,m)$
 is convex, and \eqref{mpmfg.inf} admits a solution $(X,Y, Z,Z^0)$, then $(m,\alpha)$ given by \eqref{eq:mfe.rep.inf} is a mean field equilibrium.
\end{lemma}
\begin{remark}
  A solution to \eqref{mpmfg.inf} is understood as a tuple $(X, Y, Z,Z^0)$ such that $(X,Y)$ are $n$-dimensional continuous and $\F$-adapted processes, $(Z,Z^0)$ are $n\times d$--dimensional progressive processes such that
  \begin{equation*}
    \E\bigg[\sup_{t\in[0,\infty)}\e^{-rt}\big(|X_t|^2+|Y_t|^2\big) + \int_0^\infty\e^{-rt}\big(|Z_t|^2 + |Z^0_t|^2\big)dt\bigg]<\infty.
  \end{equation*}
\end{remark}

\begin{proof}[Proof of Lemma \ref{lem.existence.inf}]
  First recall that the unique minimizer of the function $a\mapsto L(x,a) + a\cdot y$ is given by $a = D_yH(x,y)$.
  Thus, if a MFE $(m,\alpha)$ exists, then it follows from the maximum principle for infinite horizon control problems, see \cite[Section 3.1]{Bayraktar-zhang23} that \eqref{eq:mfe.rep.inf} is satisfied.

Reciprocally, by the necessary part of the maximum principle,   (see e.g. \cite[Proposition 3.1]{Bayraktar-zhang23} or \cite[Theorem 2]{Maslow-Veverka14}) it follows that if the FBSDE \eqref{mpmfg.inf} has a solution $(X, Y, Z, Z^0)$ then $\alpha_t = D_yH(X_t, Y_t)$ minimizes the cost $J_m$ with $m_t := \sL(X_t|\sF_t^0)$ and thus, $(m,\alpha)$ is a MFE.
%  Under our assumptions, the existence of \eqref{mpmfg.inf} is guaranteed by XXX.
%  Reciprocally, by \cite[Section 3.1]{Bayraktar-zhang23}, any other MFE $(\bar m,\bar\alpha)$ would satisfy $\bar\alpha_t = D_yH(\bar X_t, \bar Y_t)$ with $(\bar X, \bar Y, \bar Z, \bar Z^0)$ solving \eqref{mpmfg.inf}.
 % The latter equation having a unique solution (see XXXX) implies that the MFE is unique.
\end{proof}
\begin{remark}
  Both \cite{Bayraktar-zhang23} and \cite{Maslow-Veverka14} consider stochastic control problems with deterministic coefficient whereas in the present common noise case, the function $F$ is random due to $\sL(X_t|\sF_t^0)$.
  An inspection of the arguments in \cite[Section 3.1]{Bayraktar-zhang23} show that the same computations allow to derive the maximum principle when $F$ is random.
\end{remark}

\begin{remark}
  We do not address existence of the mean field game with infinite horizon here as it follows by essentially the same arguments as those of \cite[Theorem 2.1]{Bayraktar-zhang23}.
\end{remark}

\begin{theorem}
  \label{thm:conver.inf}
  Suppose that Assumption \ref{assump:driftcontrolstruct.inf.T} holds.
  Suppose further that for each $N\in \N$ there is a Nash equilibrium $\bm\alpha^N = (\alpha^{N,1},\dots, \alpha^{N,N})$ for the $N$ player game with infinite horizon and a MFE $(m,\alpha)$ for the MFG with infinite horizon.
  Let $\bm X^N = (X^{N,1},\dots, X^{N,N})$ be the state process corresponding to $\bm \alpha^N$ and $X$ be the state process corresponding to $\alpha$.
  Then, there is a constant $C$ independent of $N$ such that for each $N\in \N$ and $i\in \{1,\dots,N\}$,
  \begin{align*}
    \E\Big[\sup_{ t\in [0,\infty)}\e^{-rt} |X_t^{N,i} - X_t^i|^2\Big] \leq \frac{C}{N} \quad \mathrm{and}\quad \E\bigg[\int_0^\infty\e^{-rt} |\alpha_t^{N,i} - \alpha_t^{i}|^2 dt \bigg] \leq \frac CN
  \end{align*}
  where $(X^i, Y^i, Z^i, Z^{0,i})_{i \in \N}$ are the conditionally independent copies of $(X, Y, Z, Z^0)$ and $(\alpha^i)_{i=1,\dots, N}$ are conditionally independent copies of $\alpha$. 
\end{theorem}

\begin{proof}
  \textit{Step 1 - Applying the maximum principle:} 
  By the necessary maximum principle for infinite horizon control problems derived in \cite[Section 3.1]{Bayraktar-zhang23} (and applied here with law--independent coefficients), it follows that, since $\alpha^{N,i}$ minimizes $\alpha\mapsto J^{N,i}(\bm \alpha^{N,-i},\alpha)$, we have
  \begin{equation*}
    \alpha^{N,i}_t = \argmin_{a \in \R^n}\big( L(X^{N,i}_t, a) + a\cdot Y^{N,i}_t \big) %- r X^{N,i}_t\cdot Y^{N,ii}_t
  \end{equation*}
  for some $(X^{N,i}, Y^{N,i}, Z^{N,i,j}, Z^{N,i,0})_{i=1,\dots,N}$ satisfying
  \begin{align} \label{nplayermp.inf}
\begin{cases} \vspace{.1cm}
dX_t^{N,i} = \alpha_t^{N,i} dt + \Sigma dW_t^i + \Sigma_0 dW_t^0,\quad X_0^i = \xi^i, \\ \vspace{.1cm}
dY_t^{N,i} = - \Big(D_x L(X_t^{N,i}, \alpha_t^{N,i}) + D_x F(X_t^{N,i}, m_{\bX_t^N}^{N}) + \frac{1}{N} D_m F(X_t^{N,i}, m_{\bX_t^N}^N, X_t^{N,i}) - rY^{N,i}_t \Big) dt  \\ \qquad \qquad \qquad  + \sum_{j = 1}^N Z_t^{N,i,j}  dW_t^j + Z_t^{N,i,0}  dW_t^0.
%\hspace{2cm} \vspace{.1cm}  \\
 %\quad Y_T^{N,i} = D_x G(X_T^{N,i}, m_{\bX_T^N}^N) + \frac{1}{N} D_m G(X_T^{N,i}, m_{\bX_T^N}^N, X_T^{N,i}).
\end{cases} 
\end{align}
  On the other hand, from Lemma \ref{lem.existence.inf}, we have 
  \begin{equation*}
    \alpha^{i}_t = \argmin_{a \in \R^n}\big( L(X^{i}_t, a) + a\cdot Y^{i}_t \big) %- r X^{N,i}_t\cdot Y^{N,ii}_t
  \end{equation*}
  and the tuple $(X^i,Y^i,Z^i)$ satisfies 
  \begin{align}
\begin{cases} \vspace{.1cm}
  dX_t^i = \alpha_t^i dt + \Sigma dW_t + \Sigma_0 dW_t^0, \quad  X_0 = \xi, \\ \vspace{.1cm}
  dY_t^i = - \Big(D_x L(X_t^i, \alpha_t^i) + D_x F_t(X_t^i, \sL(X_t^i | \sF_t^0)) - rY^i_t \Big) dt + Z_t^i  dW_t + Z_t^{i,0}  dW_t^0. %\\
 % \quad Y_T^i = D_x G(X_T^i, \sL(X_T^i| \sF_T^0)). 
\end{cases}
\end{align}

 \textit{Step 2 - Computing $d \big( e^{-rt} \Delta X_t^i \cdot \Delta Y_t^i \big)$:}
  Let us put 
  $$\Delta X^i_t:= X^{N,i}_t - X^i_t,\quad \Delta Y^i_t:=Y^{N,i}_t - Y^i_t\quad \mathrm{and} \quad \Delta \alpha^{i}_t:= \alpha^{N,i}_t - \alpha^i_t.
  $$
  Using It\^o's formula, we compute the dynamics of $\e^{-rt}\Delta X\cdot\Delta Y$ as
\begin{align*}  
  d \e^{-rt}\Delta X_t^i \cdot \Delta Y_t^i &= \e^{-rt}\Delta Y_t^i \cdot \Delta \alpha_t^i dt - \e^{-rt}\Delta X_t^i \cdot \bigg(D_x L(X_t^{N,i}, \alpha_t^{N,i}) + D_x F(X_t^{N,i}, m_{\bX_t^N}^N) \nonumber \\
  &\quad + \frac{1}{N} D_m F(X_t^{N,i}, m_{\bX_t^N}^N, X_t^{N,i}) - D_x L(X_t^i, \alpha_t^i) - D_x F(X_t^i, \cL(X_{t}|\sF^0_t)) - r\Delta Y^i_t \bigg)\\ 
  &\quad - r\e^{-rt}\Delta X^i_t\cdot\Delta Y^i_t dt + dM_t^i, 
\end{align*}
where $M^i$ is a martingale. 
Since $\alpha_t^{N,i}$ minimizes $a \mapsto  L(X_t^{N,i}, a) + a \cdot Y_t^{N,i}$, we get 
\begin{align*}
Y_t^{N,i} = - D_a L(X_t^{N,i}, \alpha_t^{N,i}), \text{ and similarly }
{Y}_t^i = - D_a L({X}_t^{i}, {\alpha}_t^{i}).  
\end{align*}
We can use these identities to rewrite the above expression of $ d \e^{-rt}\Delta X_t^i \cdot \Delta Y_t^i$ as
\begin{align*} 
  d \e^{-rt}\Delta X_t^i \cdot \Delta Y_t^i = - &\e^{-rt}\bigg(\Delta \alpha^i_t \cdot  \big(D_a L(X_t^{N,i}, \alpha_t^{N,i}) - D_a L(X_t^i, \alpha_t^i) \big) + \Delta X_t^i \cdot \big(D_x L(X_t^{N,i}, \alpha_t^{N,i})\\
   &- D_x L(X_t^i, \alpha_t^i) \big) 
  + \Delta X_t^i \cdot \big(D_x F(X_t^{N,i}, m_{\bX_t^N}^N) - D_x F(X_t^i, m_{\bX_t^N}^N)\big) + \Delta X_t^i \cdot E_t^{F,i}  \\
  & + \frac{1}{N} \Delta X_t^i  \big( D_m F(X_t^{N,i}, m_{\bX_t^N}^N, X_t^{N,i}) - D_m F(X_t^{i}, \ov{m}_t^N, X_t^i)\big) \bigg) dt + dM_t^i
\end{align*}
where 
\begin{align*}
  E_t^{F,i} = \frac{1}{N} D_m F(X_t^{i}, \ov{m}_t^N, X_t^{i}) + D_x F(X_t^i, \ov{m}_t^N) -  D_x F(X_t^i, \ov{m}_t), 
\end{align*}
with $\ov{m}_t^N = \frac{1}{N} \sum_{i = 1}^N \delta_{X^i_t}$ and $\ov{m}_t = \sL(X_t^i | \sF_t^0)$. 
\newline \newline 
\textit{Step 3 - Using the monotonicity conditions: }
  Summing up and integrating over $[0,T]$ for any $T>0$, we obtain
  \begin{align*}
  &\sum_i\E[\e^{-rT}\Delta X^i_T\cdot\Delta Y^i_T]\\
  & =-\E\bigg[ \int_0^T \sum_i \e^{-rt}\bigg(\Delta \alpha_t^i \cdot \big(D_a L(X_t^{N,i}, \alpha_t^{N,i}) - D_a L(X_t^i, \alpha_t^i) \big) + \Delta X_t^i \cdot \big(D_x L(X_t^{N,i}, \alpha_t^{N,i}) - D_x L(X_t^i, \alpha_t^i) \big) \nonumber \\
  & \quad \quad + \Delta X_t^i \cdot \big(D_x F(X_t^{N,i}, m_{\bX_t^N}^N) - D_x F(X_t^i, m_{\bX^N_t}^N)\big) + \Delta X_t^{i} \cdot E_t^{F,i} \bigg) dt \bigg] \\
  & \quad \quad + \sum_i \frac{1}{N} \Delta X_t^i \cdot \big( D_m F(X_t^{N,i}, m_{\bX_t^N}^N, X_t^{N,i}) - D_m F(X_t^{i}, \ov{m}_t^N, X_t^i)\big) \bigg) dt \bigg] \\
 & \le -\E\bigg[\int_0^T\sum_i\e^{-rt}\Big(C_L|\Delta \alpha^i_t|^2 + C_F|\Delta X^i_t|^2\Big)dt\bigg] + \E\bigg[\int_0^T\sum_i\e^{-rt}|\Delta X^{i}_t||E^{F,i}_t|dt 
 + \frac{C}{N} \int_0^T  \sum_i e^{-rt} |\Delta X_t^i|^2 dt \bigg].
\end{align*}
  By square integrability over $[0,\infty)$ of $\Delta X^i$ and $\Delta Y^i$, (which follows from admissibility of $\alpha^{N,i}$ and $\alpha^i$) we can find a sequence $T_k\uparrow \infty$ such that $\E[\e^{-rT_k}\Delta X^i_{T_k}\cdot\Delta Y^i_{T_k}] \to 0$ for $k\to \infty$.
  Thus, it follows by monotone convergence that
\begin{align}
\nonumber
  C_L \E \bigg[\int_0^\infty \sum_i \e^{-rt}|\Delta \alpha_t^i|^2 dt \bigg]  \label{eq:estim.C_F} 
  &\leq - C_F \E \bigg[\int_0^\infty \sum_i \e^{-rt}|\Delta X_t^i|^2 dt \bigg] + \E \bigg[ \int_0^\infty\e^{-rt}  \sum_i |E_t^{F,i}||\Delta X_t^i|  dt \bigg] \nonumber  \\
  &\qquad + \frac{C}{N}\E\bigg[  \int_0^T  \sum_i e^{-rt} |\Delta X_t^i|^2 dt \bigg]
  \nonumber \\
  &\leq \E\bigg[ \int_0^\infty\e^{-rt} \Big\{ \Big(-C_F+\frac{\delta}{2}\Big) \sum_i |\Delta X_t^i|^2 + \frac{1}{2\delta}\sum_i |E_t^{F,i}|^2  \Big\} dt \bigg] \nonumber \\
  &\qquad + \frac{C}{N} \E\bigg[  \int_0^\infty  \sum_i e^{-rt} |\Delta X_t^i|^2 dt \bigg]
\end{align}
for every $\delta>0$.
Next, using It\^o's formula allows to write
\begin{align*}
    \e^{-rt}|\Delta X^{i}_t|^2 &=  \int_0^t\e^{-rs}\Big( -r|\Delta X^{i}_s|^2 + 2\Delta X^{i}_s\cdot\Delta \alpha^{i}_s  \Big)ds \\
    &\le \int_0^t\e^{-rs}\Big( (\varepsilon-r)|\Delta X^{i}_s|^2 + \frac{1}{\varepsilon}|\Delta \alpha^{i}_s|^2  \Big)ds \nonumber
\end{align*}
for every $t\ge0$ and every $\varepsilon>0$.
Choosing $\epsilon = \frac{r}{2}$ and then sending $t$ to infinity, we find that 
\begin{align} \label{alphaxboundinf}
    \int_0^{\infty} e^{-rt} |\Delta X_t^i|^2 dt \leq \frac{4}{r^2} \int_0^{\infty} e^{-rt} |\Delta X_t^i|^2 dt.
\end{align}
Plugging this into \eqref{eq:estim.C_F}, we get 
\begin{align*}
C_L \E \bigg[\int_0^{\infty} \sum_i e^{-rt} |\Delta \alpha_t^i|^2 dt \bigg] 
 &\leq  \frac{4}{r^2} \big(- (C_F \wedge 0) + \frac{\delta}{2} + \frac{C}{N} \big) \E\bigg[\int_0^{\infty}\sum_i  e^{-rt}  |\Delta \alpha_t^i|^2 dt \bigg]\\
 &\quad + \frac{1}{2 \delta} \E\bigg[\int_0^{\infty} \sum_i |E_t^{F,i}|^2 dt \bigg].
\end{align*}
We now see clearly that if $C_L - \frac{4}{r^2} C_F^- > 0$, then by choosing $\delta$ small enough (independently of $N$), we have 
\begin{align}
     \E\bigg[\int_0^{\infty} \sum_i e^{-rt} |\Delta \alpha_t^i|^2 dt \bigg] \leq C \E\bigg[\int_0^{\infty} \sum_i |E_t^{F,i}|^2 dt \bigg]
\end{align}
We next estimate 
\begin{align*}
  \E\big[ |E_t^{F,i}|^2 \big] &\leq \frac{2}{N^2} \E\big[|D_m F(X_t^i, \ov{m}_t^N, X_t^i)|^2\big] +2 \E[|D_x F(X_t^i, \ov{m}_t^N) - D_x F(X_t^i, \ov{m}_t)|^2] \\
  &\leq \frac{C}{N^2} \big(1 + \E[M_2(\ov{m}_t)]  \big)  + \frac{C}{N}\big(1 + \E[M_2(\ov{m}_t)]  \big), 
\end{align*}
with the second inequality using Lipschitz continuity of $D_m F$ and Lemma \ref{lem:errorterms} as in the proof of Theorem \ref{thm:pocdriftcontrol}. Thus we have 
\begin{align*}
  \E\bigg[\int_0^{\infty} \sum_i e^{-rt} |\Delta \alpha_t^i|^2 dt \bigg] \leq C \E\bigg[\int_0^{\infty}e^{-rt} \big( 1 + M_2(\ov{m}_t) \big) dt \bigg] \leq C. 
\end{align*}
This completes the claimed bound on $\Delta \alpha^i$, provided that we know $\E\Big[\int_0^{\infty}e^{-rt} |\Delta \alpha_t^i|^2 dt \Big] = \E\Big[\int_0^{\infty}e^{-rt} |\Delta \alpha_t^j|^2 dt  \Big]$ for $i \neq j$. This will follow from the fact that $N$-player Nash equilibria are unique as in the proof of Lemma \ref{lem:exchangeable}. Uniqueness in turn can be established as in the proof of Lemma \ref{lem:uniqueness}. Since the proofs are very similar to the finite-horizon case, we omit the details. Finally, the claimed bound on $\Delta X^i$ follows from the bound on $\Delta \alpha^i$ thanks to \eqref{alphaxboundinf}. 
  \end{proof}

\appendix

\section{Existence of mean field equilibria}
\label{sec:existence}
In this final section we prove a well-posedness result for a general system of McKean-Vlasov FBSDEs with displacement monotone generator.
This result played a key role in our proof of the existence and uniqueness of the mean field game with common noise and controlled volatility.
The proof will follow the continuity method laid down by \cite{pengwumonotone}.
Let us also refer to \cite{ahuja2016,cardelbook1,Zhang} for other generalizations of the idea.

Recall that to complete the main result on existence and uniqueness of mean field games given in Theorem \eqref{thm:existence}, we need to prove well-posedness of the conditional McKean-Vlasov equation \eqref{eq:mkvfbsde}.
%, which we recall here for the reader's convenience: 
%\begin{align} \label{mpmfggen2}
%\begin{cases} \vspace{.1cm}
%dX_t = D_y H(X_t, Y_t, Z_t, Z_t^0) dt + D_z H(X_t, Y_t, Z_t, Z_t^0) dW_t + D_{z^0} H(X_t, Y_t, Z_t, Z_t^0) dW_t^0,  \\ \vspace{.1cm}
%dY_t = - D_x H(X_t, Y_t, Z_t, Z_t^0, m_t)dt  + Z_t  dW_t + Z_t^{0}  dW_t^0, \\
%X_0 = \xi, \quad Y_T = D_x G(X_t, m_T), \vspace{.1cm} \\ 
%m_t = \sL(X_t | F_t^0). 
%\end{cases}
%\end{align}

To simplify the exposition, we will consider the generic McKean-Vlasov FBSDE given by

\begin{equation}
\label{eq:fbsde.gen}
  \begin{cases}\vspace{.1cm}
    dX_t =  B_t(X_t, Y_t, Z_t, \mu_t)dt + \Sigma_t(X_t, Y_t, Z_t, \mu_t)d\ov W_t\\\vspace{.1cm}
    dY_t = F_t(X_t, Y_t, Z_t, \mu_t)dt  - Z_td \ov W_t\\\vspace{.1cm}
    X_0 = \xi,\quad Y_T= g(X_T, m_T),\quad \mu_t =\cL(X_t|\cF^0_t)
  \end{cases}
\end{equation}
where $\ov W$ is the $d + d$-dimensional Brownian motion $(W, W^0)$.
We consider the following assumptions:
\begin{assumption}
\label{ass.fbsde.gen}
 The functions $B,F: [0,T] \times \R^n\times \R^{n}\times \R^{2(n\times d)} \times \cP_2(\R^n)\to \R^n $, $\Sigma:[0,T] \times \R^n\times \R^{n}\times \R^{2(n\times d)} \times \cP_2(\R^n)\to \R^{n\times 2d} $, and $g: \R^n\times \cP_2(\R^n)\to \R^n$ satisfy
 \begin{itemize}
  %\item[(i)] the integrability conditions: $B_\cdot(0,0,0,\delta_0), F_\cdot(0,0,0,\delta_0)\in \H^1(\R^m, \P)$, $\sigma(0,0,0,\delta_0)\in \H^2(\R^{m\times(d+d)},\P)$ and $g(0,\delta_0)\in L^2(\P)$;
  \item[(i)] the Lipschitz--continuity conditions: there is a constant $\ell_f$ such that
  \begin{equation*}
    |f_t(x,y, z,\mu) - f_t(x', y',z', \mu')|\le \ell_f\big(|x - x'| + |y - y'| + |z - z'| + \cW_2(\mu, \mu') \big) 
  \end{equation*}
  for every $(x, x', y,y', z,z', \mu, \mu')\in (\R^n)^2\times (\R^n)^2\times (\R^{n\times (d + d)})^d \times (\cP_2(\R^n))^2$ where $f$ is any function $f \in \{B, F, \sigma, g\}$ with appropriate domain and image space.
  \item[(ii)] the monotonicity condition: there is a constant $C_f>0$ such that
  \begin{align*}
    \E\Big[&\Delta Y \cdot\big(B_t(X^1, Y^1, Z^1, \cL(X^1)) - B_t(X^2, Y^2, Z^2, \cL(X^2))\big) \\
    &+ \Delta Z\cdot\big(\Sigma_t(\xi^1, Y^1, Z^1, \cL(X^1)) - \Sigma_t(X^2, Y^2, Z^2, \cL(X^2)) \big)\\
     &- \Delta X\cdot \big(F_t(X^1, Y^1, Z^1, \cL(X^1)) - F_t(X^2, Y^2, Z^2, \cL(X^2))\Big] \le - C_f\E\big[|X^1 - X^2 |^2\big]
  \end{align*}
  and
  \begin{equation*}
    \E\Big[ \Delta X\cdot\big( g(X^1, \cL(X^1)) - g(X^2, \cL(X^2))\big)\Big] \ge C_f\E[|X^1 - X^2|^2]
  \end{equation*}
  for any square-integrable random variables $X^1,X^2,Y^1,Y^2,Z^1,Z^2$ of appropriate dimensions.
\end{itemize}
\end{assumption}
Given a Euclidean space $E$, and $p\ge1$ further consider the space $\H^p(E)$ of $E$-valued progressive processes $X$ such that
\begin{equation*}
  \|X\|_{\H^p(E)}^p := \E\bigg[\int_0^T|X_t|^pdt\bigg]<\infty.
\end{equation*}
\begin{theorem}
\label{thm.gen-existence}
  Assume that $\xi\in \L^2$ and that Assumption \ref{ass.fbsde.gen} is satisfied. 
  Then Equation \ref{eq:fbsde.gen} admits a unique solution $(X, Y, Z) \in \H^2(\R^n)\times \H^2(\R^n)\times \H^2(\R^{2(n\times d)})$.
\end{theorem}

\begin{proof}
  Throughout the proof we put $\Theta := (X, Y, Z)$ and $m_{X_t}:=\cL(X_t|\cF^0_t)$.
  Given $\delta \in [0,1]$, and $b^0, f^0\in \H^2(\R^n)$, $\Sigma^0\in \H^2(\R^{2(n\times d)})$ and $g^0\in \L^2$, consider the FBSDE 
  \begin{equation}
    \label{eq:fbsde.delta}
    \begin{cases}\vspace{.1cm}
      X_t = \xi + \int_0^t\big[ B^\delta_s(\Theta_s, m_{X_s}) + b^0_s \big] ds + \int_0^t\big[\Sigma^\delta_s(\Theta_s,m_{X_s}) + \Sigma^0_s\big]d\ov W_s\\\vspace{.1cm}
    Y_t = g^\delta(X_T, m_{X_T}) + g^0 + \int_t^T\big[F^\delta_s(\Theta_s,m_{X_s}) + f^0_s\big]ds - \int_s^TZ_sd\ov W_s
    \end{cases}
  \end{equation}
  where we defined the functions
  \begin{align*}
    B^\delta_t(x, y, z,m) &:= \delta B_t(x, y, z,m) - (1 - \delta)y,\quad \Sigma^\delta_t(x, y, z,m) := \delta\Sigma_t(x, y, z,m) -(1 - \delta) z\\
    F^\delta_t(x, y, z, m) &:= \delta F_t(x, y, z, m) + (1 - \delta)x,\quad g^\delta(x, m) := \delta g(x,m) + (1 - \delta)x
  \end{align*}
  for every $(x, y, z, m)\in \R^n\times \R^n \times \R^{2(n\times d)}\times \cP_2(\R^n)$.
  We will use the notation FBSDE$(\delta)$ to denote Equation \eqref{eq:fbsde.delta}.
  Observe that FBSDE$(1)$ is exactly \eqref{eq:fbsde.gen} and FBSDE$(0)$ is a standard linear FBSDE (i.e. not of McKean--Vlasov type).
  %It follows by \cite[Lemma 8.4.3]{zhang2017backward} that FBSDE$(0)$ admits a unique solution $(X^0, Y^0, Z^0) \in XXXX$.

  \emph{Step 1:} Continuation. 
  In this first step, we will show that for any given $\delta_0\in [0,1]$, if for every $b^0, f^0\in \H^2(\R^n)$, $\sigma^0\in \H^2(\R^{2(n\times d)})$ and $g^0\in \L^2$ the FBSDE$(\delta_0)$ admits a priori a unique solution in $\H^2(\R^n)\times \H^2(\R^n)\times \H^2(\R^{2\times(n\times d)})$, then there is $\eta>0$ depending only on the constants $c_f,\ell_f $ and $T$ such that for every $\delta \in [\delta_0, \delta_0+\eta]$ FBSDE$(\delta)$ also admits a unique solution in $\H^2(\R^n)\times \H^2(\R^n)\times \H^2(\R^{2\times(n\times d)})$.

  To this end, let $\delta \in [\delta_0, \delta_0+\eta]$ for some $\eta>0$ to be determined and put $\varepsilon := \delta - \delta_0$.
  Consider the mapping $\Psi:\H^2(\R^n)\times \H^2(\R^n)\times \H^2(\R^{2\times(n\times d)})\to \H^2(\R^n)\times \H^2(\R^n)\times \H^2(\R^{2\times(n\times d)})$ such that $\Psi(x, y, z) = (X,Y,Z)$ with $(X,Y,Z)$ satisfying the FBSDE
  \begin{equation}
  \begin{cases}\vspace{.2cm}
    X_t = \xi + \int_0^tB^{\delta_0}_s(\Theta_s,m_{X_s}) + \varepsilon\big[y_s + B_s(\vartheta_s,m_{x_s})\big] ds + \int_0^t\Sigma^{\delta_0}_s(\Theta_s,m_{X_s}) + \varepsilon\big[z_s + \Sigma_s(\vartheta_s, m_{x_s}) \big]dW_s\\\vspace{.2cm}
    Y_t = g^{\delta_0}(X_T, m_{X_T}) + \varepsilon\big[-x_T + g(x_T, m_{x_T})\big] + \int_t^TF^{\delta_0}_s(\Theta_s, m_{X_s}) + \varepsilon\big[-x_s + F_s(\vartheta_s, m_{x_t})\big]ds \\
    \qquad - \int_t^TZ_sd\ov W_s
  \end{cases}
  \end{equation}
  where we used the notation
   $$\vartheta = (x, y, z),\quad m_{X_t} = \cL(X_t|\cF_t^0)\quad \text{and}\quad m_{x_t} = \cL(x_t|\cF_t^0).$$
  The assumption on existence of FBSDE$(\delta_0)$ for every $b^0, f^0\in \H^2(\R^n)$, $\Sigma^0\in \H^2(\R^{2(n\times d)})$ and $g^0\in \L^2$ as well as Assumption \ref{ass.fbsde.gen}.$(i), (ii)$ guarantee that $(X, Y, Z)$ is well defined.
  We will show that for an adequate choice of $\eta$, the functional $\Psi$ admits a unique fixed point in $\H^2(\R^n)\times \H^2(\R^n)\times \H^2(\R^{2\times(n\times d)})$.
  Let us use the shorthand notation 
  $$\Delta \zeta:= \zeta^1 - \zeta^2$$ 
  for any vectors $\zeta^1$ and $ \zeta^2$ irrespective of the dimension. 
  Let $(x^1, y^1, z^1), (x^2, y^2, z^2)\in \H^2(\R^n)\times \H^2(\R^n)\times \H^2(\R^{2\times(n\times d)})$ and $(X^i, Y^i, Z^i) = \Psi(x^i, y^i, z^i)$, $i=1,2$.
  %Let us put $m^i_t(\cdot) = \P(X^i_t\in \cdot|\Fc^0_t)$ and $\nu^i_t(\cdot) = \P(x^i_t\in \cdot|\Fc^0_t)$ for $i=1,2$.
  Applying It\^o's formula, we have
  \begin{align*}
    &d\Delta X_t\cdot\Delta Y_t\\
     &= \Delta X_t\cdot\Big\{ -\big(F^{\delta_0}_t(\Theta_t^1,m_{X^1_t}) - F^{\delta_0}_t(\Theta^2_t, m_{X^2_t})\big) + \varepsilon\big( \Delta x_t - F_t(\vartheta^1_t, m_{x^1_t}) + F_t(\vartheta_t^2, m_{x^2_t})) \Big\}dt\\
    &\quad + \Delta Y_t\cdot\Big\{ B^{\delta_0}_t(\Theta^1_t, m_{X^1_t}) - B^{\delta_0}_t(\Theta^2_t, m_{X^2_t}) + \varepsilon\big( \Delta y_t + B_t(\vartheta_t^1, m_{x^1_t}) - B_t(\vartheta_t^2, m_{x^2_t}) \big) \Big\}dt\\
    &\quad + \Delta Z_t\cdot\Big\{ \Sigma^{\delta_0}_t(\Theta^1_t, m_{X^1_t}) - \Sigma^{\delta_0}_t(\Theta^2_t, m_{X^2_t}) + \varepsilon\big( \Delta z_t + \Sigma_t(\vartheta_t^1,m_{x^2_t}) - \Sigma_t(\vartheta_t, m_{x^2_t}) \big) \Big\}dt + dM_t
  \end{align*}
  for a martingale $M$.
  Observe that because $B, F, \sigma$ and $g$ satisfy Assumption \ref{ass.fbsde.gen}$(ii)$, it follows that $B^\delta, F^\delta, \sigma^\delta$ and $g^\delta$ also satisfy Assumption \ref{ass.fbsde.gen}$(ii)$ with the constant $C_f$ replaced by
  \begin{equation}
  \label{eq:disp.mon.cons.bound}
    C_{\delta} := \delta c_f + 1 - \delta\ge \min(c_f, 1) =:\bar C_f.
  \end{equation}
  Thus, using that $\Delta X_0 = 0$ we have
  \begin{align*}
    \E\big[\Delta X_T\cdot \Delta Y_T\big] + C_{\delta_0} \E\bigg[\int_0^T|\Delta X_t|^2dt \bigg]&
    \le  \varepsilon \E\bigg[\int_0^T -\Delta X_t \cdot\big(F_t(\vartheta^1_t, m_{x^1_t}) - F_t(\vartheta^2_t, m_{x^2_t}) \big)\\
     &\qquad \qquad + \Delta Y_t\cdot\big( B_t(\vartheta^1_t, m_{x^1_t}) - B_t(\vartheta^2_t, m_{x^2_t}) \big)\\
    &\qquad \qquad  + \Delta Z_t\cdot\big(\sigma_t(\vartheta^1_t, m_{x^1_t}) - \sigma_t(\vartheta^2_t, m_{x^2_t}) \big)\\
    &\qquad \qquad + \Delta X_t\cdot \Delta x_t + \Delta Y_t\cdot \Delta y_t + \Delta Z_t\cdot \Delta z_t dt \bigg].
  \end{align*}
    Thanks to Lipschitz continuity of $B, F, \sigma$ and $g$ and Young's inequality, it follows that
  \begin{align*}
    \E\big[\Delta X_T\cdot \Delta Y_T\big] +(C_{\delta_0}-\varepsilon)\E\bigg[\int_0^T|\Delta X_s|^2ds \bigg] %&\le  \varepsilon\E\bigg[ \int_0^T|\Delta \Theta_s|^2 + C \Big(|\Delta \vartheta_s|^2 + \cW_2^2(m_{x^1_s}, m_{x^2_s})\Big)ds \bigg]\\
    &\le %(\varepsilon-c_{\delta_0})\E\bigg[\int_0^T|\Delta X_s|^2ds \bigg] +
     \varepsilon C\E\bigg[\int_0^T|\Delta \Theta_s|^2 + |\Delta \vartheta_s|^2ds\bigg]
  \end{align*}
  for some constant $C>0$, and  where we also used the fact that $\E[\cW^2_2(m_{x^1_s}, m_{x^2_s})]\le \E[|x^1_s - x^2_s|^2]$ and Fubini's theorem.
  In the enusing computation, constant $C$ may change from line to line, it will depend only from $C_f, \ell_f$ and $T$, but never on $\delta_0$.
  On the other hand, using that $g^{\delta_0}$ satisfies Assumption \ref{ass.fbsde.gen}$(ii)$ with constant $C_f$ replaced by $C_{\delta_0}$, we have
  \begin{align*}
    \E[\Delta X_T\cdot \Delta Y_T] \ge C_{\delta_0}\E[|\Delta X_T|^2] + \varepsilon \E\big[-\Delta X_T\cdot\Delta x_T + \Delta X_T\cdot(g(x^1_T, m_{x^1_T}) - g(x^2,m_{x^2_T})) \big].
  \end{align*}
  Combining the last two estimates thus yields
  % \begin{align*}
  %   C_{\delta_0}\E[|\Delta X_T|^2] &+ (C_{\delta_0} - \varepsilon)\E\bigg[\int_0^T|\Delta X_s|^2 ds\bigg]\\
  %   %&\le \varepsilon C\E\big[\Delta X_T\Delta x_T - \Delta X_T(g(x^1_T, \nu^1_T) - g(x^2,\nu^2_T)) \big] + \varepsilon\E\bigg[\int_0^T|\Delta \Theta_s|^2 + (3\ell_f^2 +1)|\Delta \vartheta_s|^2ds \bigg]\\
  %   &\le \varepsilon C\E\bigg[\int_0^T|\Delta \Theta_s|^2 + |\Delta \vartheta_s|^2ds \bigg] + \varepsilon C\E\Big[|\Delta X_T|^2 + |\Delta x_T|^2 \bigg].
  % \end{align*}
  %Thus, it holds that
  \begin{align}
  \label{eq:estim.mono.1}
    (C_{\delta_0} - \varepsilon)\E[|\Delta X_T|^2] + (C_{\delta_0} - \varepsilon)\E\bigg[\int_0^T|\Delta X_s|^2ds\bigg]\le \varepsilon C\E\bigg[\int_0^T|\Delta \Theta_s|^2 + |\Delta \vartheta_s|^2ds \bigg] + \varepsilon C\E\big[|\Delta x_T|^2\big].
  \end{align}

  Now, observe that, since $B, F, \sigma$ and $g$ are $\ell_f$-Lipschitz continuous and $\delta_0\le 1$, 
  %the functions $B^{\delta_0}, F^{\delta_0}, \sigma^{\delta_0}$ and $g^{\delta_0}$ are $\delta_0\ell_f + |1 - \delta_0|$--Lipschitz continuous
%  Since $\delta_0\le 1$, 
  it follows that the functions $B^{\delta_0}, F^{\delta_0}, \sigma^{\delta_0}$ and $g^{\delta_0}$  are $\ell_f + 2$--Lipschitz continuous.
  Without loss of generality we will nevertheless denote their Lipschitz constant again by $\ell_f$.
  Thus, applying It\^o's formula to $\e^{\kappa t}|\Delta Y_t|^2$ for some $\kappa>0$ to be specified below, we have
  \begin{align*}
    \e^{\kappa t}\E\big[|\Delta Y_t|^2\big] &+ \E\bigg[\int_t^T\e^{\kappa s}|\Delta Z_s|^2ds \bigg] \le C\e^{\kappa T}\E\big[|\Delta X_T|^2\big] + \varepsilon^2C\e^{\kappa T}\E\big[|\Delta x_T|^2\big]\\
    % + |\Delta X_T|^2 \Big) + \e^{\kappa T}\varepsilon^2|\Delta x_T|^2\Big]\\
    &+\E\bigg[\int_t^T\e^{\kappa s}\Big(\frac{4\ell_f^2 + 1}{a} - \kappa\Big)|\Delta Y_s|^2 + 2a\e^{\kappa s}|\Delta \Theta_s|^2 + 3a\varepsilon\e^{\kappa s}|\Delta \vartheta_s|^2 ds\bigg] 
  \end{align*}
  for any $a>0$.
  Choosing $a<1/2$ and $\kappa\ge \frac{4\ell_f^2 + 1}{a}$, we obtain
  \begin{align*}
    \e^{\kappa t}\E\big[|\Delta Y_t|^2\big] + (1-2a)\E\bigg[\int_t^T\e^{\kappa s}|\Delta Z_s|^2ds\bigg] &\le C\e^{\kappa T}\E\big[|\Delta X_T|^2\big] + \varepsilon^2C\e^{\kappa T}\E\big[|\Delta x_T|^2\big]\\
     %\e^{\kappa T}\E\Big[\ell_f^2|\Delta X_T|^2 + (\varepsilon + \ell_g^2)|\Delta x_T|^2\Big]\\
      &\quad + a\E\bigg[\int_t^T2\e^{\kappa s}(|\Delta X_s|^2 + |\Delta Y_s|^2) + 3\varepsilon\e^{\kappa s}|\Delta \vartheta_s|^2ds\bigg].
  \end{align*}
  Thus, we have the estimates
  \begin{align*}
    \E\bigg[\int_0^T|\Delta Y_t|^2dt\bigg] + \E\bigg[\int_0^T|\Delta Z_t|^2dt \bigg] &\le C\E\big[|\Delta X_T|^2\big] + \varepsilon^2C\E\big[|\Delta x_T|^2\big]\\
    % \e^{\kappa T}T\E\Big[\ell_f^2|\Delta X_T|^2 + (\varepsilon + 2\ell_f^2)|\Delta x_T|^2\Big]
    &\quad + a C\E\bigg[\int_0^T|\Delta X_t|^2dt\bigg] + aC\E\bigg[\int_0^T|\Delta\vartheta_t|^2dt\bigg].
  \end{align*}
  % and
  % \begin{align*}
  %   \E\bigg[\int_0^T|\Delta Z_t|^2dt \bigg] &\le C\E\big[|\Delta X_T|^2\big] + \varepsilon^2\e^{\kappa T}\E\big[|\Delta x_T|^2\big]\\
  %   &\quad + a\e^{\kappa T}\E\bigg[\int_0^T|\Delta X_t|^2 +|\Delta Y_t|^2dt\bigg] + \frac{3a\varepsilon\e^{\kappa T}}{1 - 2a} \E\bigg[\int_0^T|\Delta\vartheta_t|^2dt\bigg].
  % \end{align*}

  Thus, coming back to Equation \ref{eq:estim.mono.1}, the latter estimates allow to obtain
  \begin{align*}
    (C_{\delta_0} - \varepsilon)\E[|\Delta X_T|^2] & + (C_{\delta_0} - \varepsilon)\E\bigg[\int_0^T|\Delta X_s|^2ds\bigg] + \E\bigg[\int_0^T|\Delta Y_s|^2 + |\Delta Z_s|^2ds \bigg]\\
    %&\le \varepsilon\E\bigg[\int_0^T|\Delta \Theta_s|^2ds\bigg] + (3\ell_f^2+1)\varepsilon\E\bigg[\int_0^T|\Delta\vartheta_s|^2ds\bigg] + \big(2\ell_f^2 + 1\big)\varepsilon\E\big[|\Delta x_T|^2\big]\\
    %&\quad + 2aT\e^{\kappa T}\E\bigg[\int_0^T|\Delta X_s|^2 + |\Delta Y_s|^2ds\bigg] + 3a\varepsilon T\e^{\kappa T}\E\bigg[\int_0^T|\Delta \vartheta_s|ds\bigg]\\
    %&\quad + \frac{2a\e^{\kappa T}}{1 - 2a}\e^{\kappa T}\E\bigg[\int_0^T|\Delta X_s|^2 + |\Delta Y_s|^2ds \bigg] + \frac{3a\varepsilon\e^{\kappa T}}{1 - 2a}\E\bigg[\int_0^T|\Delta \vartheta_t|^2dt\bigg]\\
    %&\quad + 16T\ell_f^2\e^{\kappa T}\E\big[|\Delta X_T|^2\big] + \varepsilon^2(16\ell_f^2 + 4)\e^{\kappa T}T\E\big[|\Delta x_T|^2\big]\\
    %&\quad + \frac{16\ell_f^2\e^{\kappa T}}{1-2a}\E\big[|\Delta X_T|^2\big] + \varepsilon^2\frac{(16\ell_f^2 + 4)\e^{\kappa T}}{1-2a}\E\big[|\Delta x_T|^2\big]\\
    &\le (\varepsilon + a)C\E\bigg[\int_0^T|\Delta \Theta_t|^2dt \bigg] + \varepsilon C\E\bigg[\int_0^T|\Delta\vartheta_t|^2dt\bigg]\\
    &\quad+ C \E\big[|\Delta X_T|^2\big]+ \varepsilon C\E[|\Delta x_T|^2].
  \end{align*}
  % In particular, introducing the constant
  % \begin{equation*}
  %   C_{\ell_f, T}:=  16T\ell_f^2 \e^{\kappa T} + \frac{16\ell_f^2\e^{\kappa T}}{1-2a}
  %  \end{equation*} 
  %  this implies in that
  % \begin{align*}
  %   (c_{\delta_0} - \varepsilon)\E[|\Delta X_T|^2]& + \Big\{(c_{\delta_0} - \varepsilon) - \Big(\varepsilon + 2aT\e^{\kappa T} + \frac{2a\e^{\kappa T}}{1 - 2a} \Big) \Big\}\E\bigg[\int_0^T|\Delta X_t|^2dt\bigg] + \Big\{ 1 - \Big( \varepsilon + 2aT\e^{\kappa T} + \frac{2a\e^{\kappa T}}{1-2a} \Big) \Big\}\E\bigg[\int_0^T|\Delta Y_t|^2 + |\Delta Z_t|^2dt\bigg]\\
  %   & \le \varepsilon\Big(3\ell_f^2 +1 + 3aT\e^{\kappa T} + \frac{3a\e^{\kappa T}}{1 - 2a} \Big)\E\bigg[\int_0^T|\Delta \vartheta_t|^2dt\bigg] + \varepsilon\e^{\kappa T}(16\ell_f^2+4)\Big(1 +\frac{1}{1-2a}\Big)\E[|\Delta x_T|^2]\\
  %   &\quad + \varepsilon \frac{C_{\ell_f, T}}{c_{\delta_0} - \varepsilon}\E\bigg[\int_0^T|\Delta \Theta_s|^2 + (3\ell_f^2 +1)|\Delta \vartheta_s|^2ds \bigg] +\varepsilon\frac{C_{\ell_f,T}}{c_{\delta_0} - \varepsilon} \big(2\ell_f^2+1\big)\E\big[|\Delta x_T|^2\big],
  % \end{align*}
  Using again Equation \ref{eq:estim.mono.1} and the fact that $c_{\delta_0}\ge \bar c_f$ (recall \eqref{eq:disp.mon.cons.bound}), we obtain
  \begin{align*}
    (\bar C_{f} - \varepsilon)\E[|\Delta X_T|^2]& + \Big\{(\bar C_{f} - 2\varepsilon) - (\varepsilon + a)C \Big\}\E\bigg[\int_0^T|\Delta X_t|^2dt\bigg]\\
     &\quad+ \Big\{ 1 - ( \varepsilon + a)C \Big\}\E\bigg[\int_0^T|\Delta Y_t|^2 + |\Delta Z_t|^2dt\bigg]\\
    & \le \varepsilon C\E\bigg[\int_0^T|\Delta \vartheta_t|^2dt\bigg] + \varepsilon C\E[|\Delta x_T|^2].
  \end{align*}
  Now, since $\bar C_f>0$ choosing $a$ and $\varepsilon$ small enough allows to get
  % \begin{align*}
  %   (\bar c_{f} - \varepsilon)\E[|\Delta X_T|^2]& + \Big\{(\bar c_{f} - 2\varepsilon) - \Big( 2aT\e^{\kappa T} + \frac{2a\e^{\kappa T}}{1 - 2a} + \varepsilon \frac{2C_{\ell_f, T}}{\bar c_{f} } \Big) \Big\}\E\bigg[\int_0^T|\Delta X_t|^2dt\bigg]\\
  %    &\quad+ \Big\{ 1 - \Big( \varepsilon + 2aT\e^{\kappa T} + \frac{2a\e^{\kappa T}}{1-2a} + \varepsilon \frac{2C_{\ell_f, T}}{\bar c_{f}} \Big) \Big\}\E\bigg[\int_0^T|\Delta Y_t|^2 + |\Delta Z_t|^2dt\bigg]\\
  %   & \le \varepsilon\Big(3\ell_f^2 +1 + 3aT\e^{\kappa T} + \frac{3a\e^{\kappa T}}{1 - 2a} + \varepsilon \frac{2C_{\ell_f, T}}{\bar c_{f} }(3\ell_f^2+1)\Big)\E\bigg[\int_0^T|\Delta \vartheta_t|^2dt\bigg]\\
  %   &\quad + \varepsilon\Big\{\e^{\kappa T}(16\ell_f^2+4)\Big(1 +\frac{1}{1-2a}\Big) + \frac{2C_{\ell_f,T}}{\bar c_{f}} \big(2\ell_f^2+1\big)\Big\}\E[|\Delta x_T|^2].
  % \end{align*}
  % Further choosing $\varepsilon< \frac{\bar c_f}{2 + \frac{2C_{\ell_f,T}}{\bar c_f}}$, we can find $a$ small enough that
  % %  which implies that $\varepsilon<c_{\delta_0}/2$ and choose $a$ to be such that
  % \begin{equation*}
  %   C^1_{\ell_f, c_f,T,\varepsilon} := (\bar c_{f} - 2\varepsilon) - \Big( 2aT\e^{\kappa T} + \frac{2a\e^{\kappa T}}{1 - 2a} + \varepsilon \frac{2C_{\ell_f, T}}{\bar c_{f} } \Big)>0
  % \end{equation*}
  % and 
  % \begin{equation*}
  %   C^2_{\ell_f, c_f,T,\varepsilon} := 1 - \Big( \varepsilon + 2aT\e^{\kappa T} + \frac{2a\e^{\kappa T}}{1-2a} + \varepsilon \frac{2C_{\ell_f, T}}{\bar c_{f}} \Big)>0.
  % \end{equation*}
  % Thus, it holds
  \begin{align*}
    \E[|\Delta X_T|^2] + \E\bigg[\int_0^T|\Delta \Theta_s|^2ds\bigg]
     \le \varepsilon C%\frac{D_{\ell_f, T}}{C^3_{\ell_f, c_f, T,1}}
     \bigg\{ \E\bigg[\int_0^T|\Delta \vartheta_t|^2dt\bigg] + \E[|\Delta x_T|^2]\bigg\}.
  \end{align*}
  % with the constants
  % \begin{equation*}
  %   D_{\ell_f,c_f, T} := \Big\{3\ell_f^2 +1 + 3aT\e^{\kappa T} + \frac{3a\e^{\kappa T}}{1 - 2a} +  \frac{2C_{\ell_f, T}}{\bar c_{f} }(3\ell_f^2+1)\Big\} \vee \Big\{\e^{\kappa T}(16\ell_f^2+4)\Big(1 +\frac{1}{1-2a}\Big) + \frac{2C_{\ell_f,T}}{\bar c_{f} } \big(2\ell_f^2+1\big)\Big\}
  % \end{equation*}
  % and $C^3_{\ell_f, c_f, T, \varepsilon} := C^1_{\ell_f, c_f, T, \varepsilon}\wedge  C^2_{\ell_f, c_f, T, \varepsilon}$.
  % \begin{align*}
  %   \E\big[|\Delta X_T|^2\big] + \E\bigg[\int_0^T|\Delta \Theta_t|^2dt\bigg] &\le \varepsilon\frac{C^2_a}{C^1_{a,\varepsilon,\delta_0}}\bigg\{ \E\bigg[\int_0^T|\Delta \vartheta_t|^2 dt\bigg] + \E\big[|\Delta x_T|^2\big] \bigg\}\\
  %   &\le \varepsilon\frac{C^2_a}{C^1_{a,1}}\bigg\{ \E\bigg[\int_0^T|\Delta \vartheta_t|^2 dt\bigg] + \E\big[|\Delta x_T|^2\big] \bigg\}.
  % \end{align*}
  %If $\varepsilon < \frac{C^3_{\ell_f, c_f, T, 1}}{D_{\ell_f,T}}$, then $\Psi$ is a contraction mapping.
  % For that, it suffices to take $\eta = \frac{C^3_{\ell_f, c_f, T, 1}}{D_{\ell_f,T}}$, a constant that depends only on $c_f,\ell_f $ and $T$.
  % In particular, $\eta$ does not depend on $\delta_0$.
  Thus, for $\varepsilon$ sufficiently small, $\Psi$ is a contraction mapping.
  Therefore, it suffices to take $\eta$ as the upper bound allowed for $\varepsilon$, which does not depend on $\delta_0$ as $C$ does not depend on $\delta_0$.

  \medskip

  \emph{Step 2: existence for Equation \eqref{eq:fbsde.gen}.} In this step, we show that FBSDE$(1)$ is solvable.
  Notice that FBSDE$(1)$ is precisely Equation \eqref{eq:fbsde.gen}.
  By \cite[Lemma 8.4.3]{Zhang}, FBSDE$(0)$ is solvable.
  If $\eta\ge T$, then the result follows from \emph{Step 1}.
  Let us assume that $\eta<T$ and let $n\in \N^\star$ be such that $(n-1)\eta<T \le n\eta$.
  Then, by step 1 repeated n times, FBSDE$(1)$ is also solvable.
  %\emph{Step 3: uniqueness}
  This concludes the proof.
\end{proof}

\begin{remark}
\label{rem.proof.affine}
Using Lemma \ref{lem:mpmfggen} and Remark \ref{rmk:fbsdeaffinemfe}, in order to prove Theorem \ref{thm:existenceaffine} it suffices to show that the FBSDE 
\begin{align} \label{eq:nplayerfbsdeaffinemfe.appen}
 \begin{cases} \vspace{.1cm}
   dX_t = \Big(B_0 + B_1 \alpha_t + B_2 X_t \Big) dt + \Big( \Sigma_0 + \Sigma_1 \alpha_t + \Sigma_2 X_t\Big) dW_t^i + \Big(\Sigma^0_0 + \Sigma^0_1 \alpha_t + \Sigma^0_2 X_t \Big) dW_t^0   \\ \vspace{.1cm}
dY_t = - \Big( D_x L(X_t,\alpha_t, m_t) + B_2^\top Y_t + \Sigma_2^\top Z_t + (\Sigma_2^0)^\top Z_t^0 \Big) dt
+ Z_t  dW_t + Z_t^0  dW_t^0 \vspace{.1cm}  \\
X_0 = \xi, \quad Y_T = D_x G(X_T,m_T)
\end{cases} 
\end{align}
with
\begin{align*}
    m_t = \sL(X_t | \sF_t^0), \quad \alpha_t = \hat{\alpha}(X_t,Y_t,Z_t,Z_t^0, \sL(X_t | \sF_t^0))
\end{align*}
admits a unique solution.
While this well--posedness of this system is not covered by Theorem \ref{thm.gen-existence}, thanks to the monotonicity assumption on $L$ essentially the same argument allows to derive the desired result in the present case.
The proof is thus omitted.
\end{remark}

\bibliographystyle{abbrvnat} %amsalpha %abbrv $amsplain
\bibliography{dispmonotone}

\end{document}